\documentclass[10pt,a4paper]{article}
\usepackage[utf8]{inputenc}
\usepackage{amsmath}
\usepackage{amsfonts}
\usepackage{amssymb}
\usepackage[margin=2.5cm,nohead]{geometry}
\usepackage{graphicx}
\usepackage{amsthm}
\usepackage{dsfont}
\usepackage{color}

\newtheorem{thm}{Theorem}[section]
\newtheorem{prop}{Proposition}[section]
\newtheorem{cor}{Corollary}[section]
\newtheorem{rmk}{Remark}[section]
\newtheorem{lma}{Lemma}[section]

\newcommand{\E}{\mathbb{E}}
\newcommand{\Q}{\mathbb{Q}}
\newcommand{\Prob}{\mathbb{P}}

\newcommand \Esp {\mathbb{E}}

\def\N{{\rm I\kern-0.16em N}}
\def\R{{\rm I\kern-0.16em R}}
\def\E{{\rm I\kern-0.16em E}}
\def\P{{\rm I\kern-0.16em P}}
\def\F{{\rm I\kern-0.16em F}}
\def\B{{\rm I\kern-0.16em B}}
\def\C{{\rm I\kern-0.46em C}}
\def\G{{\rm I\kern-0.50em G}}

\newcommand\blfootnote[1]{%
  \begingroup
  \renewcommand\thefootnote{}\footnote{#1}%
  \addtocounter{footnote}{-1}%
  \endgroup
}

\author{J\"urgen Angst, Thibault Pautrel, Guillaume Poly}

\title{Real zeros of random trigonometric polynomials \\  with dependent coefficients} 

\begin{document}
\maketitle
\blfootnote{jurgen.angst@univ-rennes1.fr, thibault.pautrel@univ-rennes1.fr, guillaume.poly@univ-rennes1.fr}
\blfootnote{Univ Rennes, CNRS, IRMAR - UMR 6625, F-35000 Rennes, France.}
\blfootnote{This work was supported by the ANR grant UNIRANDOM, ANR-17-CE40-0008.}

\abstract{We further investigate the relations between the large degree asymptotics of the number of real zeros of random trigonometric polynomials with dependent coefficients and the underlying correlation function. 
We consider trigonometric polynomials of the form\par
\vspace{-0.2cm}
\[
f_n(t):= \frac{1}{\sqrt{n}}\sum_{k=1}^{n}a_k \cos(kt)+b_k\sin(kt), ~x\in [0,2\pi],
\]
where the sequences $(a_k)_{k\geq 1}$ and $(b_k)_{k\geq 1}$ are two independent copies of a stationary Gaussian process centered with variance one and correlation function $\rho$ with associated spectral measure $\mu_{\rho}$. We focus here on the case where $\mu_{\rho}$ is not purely singular and we denote by $\psi_{\rho}$ its density component with respect to the Lebesgue measure $\lambda$. Quite surprisingly, we show that the asymptotics of the number of real zeros $\mathcal{N}(f_n,[0,2\pi])$ of $f_n$ in $[0,2\pi]$ is not related to the decay of the correlation function $\rho$ but instead to the Lebesgue measure of the vanishing locus of $\psi_{\rho}$. Namely, assuming that $\psi_{\rho}$ is $\mathcal{C}^1$ with H\"older derivative on an open set of full measure, one establishes that
\[
\lim_{n \to +\infty} \frac{\Esp\left[\mathcal{N}(f_n,[0,2\pi])\right]}{n}= \frac{\lambda(\{\psi_{\rho}=0\})}{\pi \sqrt{2}} + \frac{2\pi - \lambda(\{\psi_{\rho}=0\})}{\pi\sqrt{3}}.
\]
On the other hand, assuming a sole log-integrability condition on $\psi_{\rho}$, which implies that it is positive almost everywhere, we recover the asymptotics of the independent case:
\[
\lim_{n \to +\infty} \frac{\Esp\left[\mathcal{N}(f_n,[0,2\pi])\right]}{n}= \frac{2}{\sqrt{3}}.
\]
The latter asymptotics thus broadly generalizes the main result of \cite{ADP19} where the spectral density was assumed to be continuous and bounded from below. Besides, with further assumptions of regularity and existence of negative moment for $\psi_{\rho}$, which encompass e.g. the case of random coefficients being increments of fractional Brownian motion with any Hurst parameter, we moreover show that the above convergence in expectation can be strengthened to an almost sure convergence:
\[
\lim_{n \to +\infty} \frac{\mathcal{N}(f_n,[0,2\pi])}{n}= \frac{2}{\sqrt{3}}.
\] 
}

\newpage
\tableofcontents
\section{Introduction and statements of the result}
\subsection{Introduction} 
This article focuses on the study of the number of real zeros of random trigonometric polynomials which is the object of a vast literature. In this framework, of particular interest is the question of the universality of the large degree asymptotics of this number of zeros, which generally consists in determining whether the latter asymptotic behavior depends or not on the specific choice of the joint distribution of the random coefficients.
\par
\medskip
The case of random trigonometric polynomials whose coefficients are independent and identically distributed has been studied intensively. In this setting, the asymptotics of the expected number of real zeros is known to display an universal behavior, at both local and global scales, as established in a serie of papers, see for example \cite{AP15,IKM16,Fla17,DNV18} and recently \cite{NV18} which provides the most general conditions.
\par
\medskip
Though it is a rather natural extension, much less is known in the case of random coefficients that are correlated. Most of the techniques developed in the aforementioned references seems hard to adapt to the context of correlated coefficients, roughly because one cannot exploit anymore the independence to provide accurate enough estimates of characteristic functions, which enable one to deal with anti-concentration problems that naturally arise in this context. Nevertheless, in the case of Gaussian coefficients, one can bypass these difficulties which explain that the content of the available literature for dependent coefficients is so far restricted to Gaussianity.
\par 
\medskip
Let us detail our model. We consider a probability space $(\Omega, \mathcal{F}, \Prob)$ on which we define two independent stationary Gaussian processes $(a_k)_{k\geq 1}$ and $(b_k)_{k\geq 1}$, where the variables $a_k$ and $b_k$ are centered with unit variance and  with covariance function 
\[
\rho(|k-\ell|):= \Esp[a_k a_{\ell}]=\Esp[b_k b_{\ell}]~,~~k, \ell \geq 1.
\] 
By Bochner--Herglotz Theorem, the sequence $\rho$ is then associated to a so-called spectral measure $\mu_\rho$ and since $\rho$ is real and $\rho(0)=1$, the measure $\mu_\rho$ is in fact a symmetric probability measure on $[-\pi,\pi]$. We then set $f_n$ the associated random trigonometric polynomial 
\begin{equation}\label{def.model}
f_n(x)= \frac{1}{\sqrt{n}} \sum_{k=1}^n a_k \cos(kx) + b_k \sin(kx), \quad x \in [-\pi, \pi].
\end{equation}
The number of real zeros of a function $f$ in a given interval $[a,b]$ will be denoted by
\[\mathcal{N}(f,[a,b]):=\#\left\{t\in [a,b], f(t)=0\right\}.\]

The asymptotics of the number of real zeros in the case of Gaussian independent coefficients, i.e. the case where $\rho(k)=\delta_0(k)$ and $\mu_{\rho}(dx)=\frac{1}{2\pi} \mathds{1}_{[-\pi, \pi]}(x)dx$, was first studied by Dunnage in \cite{Dun66}. The analogue question for Gaussian dependent coefficients was first investigated in \cite{Sam78,RS84} where the coefficients $(a_k)_{k\ge 1}$ followed a stationary Gaussian process with both constant correlation i.e. $\rho(k)=r$ with $|r|<1$ for $k \neq 0$, and geometric correlation i.e. $\rho(k)=r^k$ with $|r|<1$. Although these two types of correlations are of seemingly very different nature, it was shown that the asymptotic behavior of the expected number of roots coincides with the one of i.i.d. Gaussian coefficients. In the reference \cite{ADP19}, the authors examined the case where the spectral measure admits a continuous density on $(0,2\pi)$ which is bounded from below on $[0,2\pi]$ for which the asymptotic also coincides with the one of i.i.d. Gaussian coefficients:

\begin{equation} \label{eq.asym.univ}
\lim_{n\to +\infty} \frac{\Esp\left[\mathcal{N}(f_n,[0,2\pi])\right]}{n}= \frac{2}{\sqrt{3}}.
\end{equation}

Recently in \cite{Pau20}, the author studied the case where the random coefficients still form a stationary Gaussian process, but the associated spectral measure in purely singular, namely $\rho(k)=\cos(k\alpha)$ with $\alpha \not \in \mathbb{Q}$ such that $\mu_{\rho} = \frac{\delta_{\alpha} + \delta_{-\alpha}}{2}$. He then established that the normalized expected number of zeros is not converging and in fact admits a whole continuum of possible limits, namely:
\begin{equation}\label{pau.20}
\text{Adh}\left(\frac{\Esp\left[\mathcal{N}(f_n,[0,2\pi])\right]}{n}\,\Big{|}~n\ge 1\right)=[\sqrt{2},2].
\end{equation}

As $\frac{2}{\sqrt{3}}\notin [\sqrt{2},2]$, it then clearly appears that the asymptotic behavior of the number of real zeros is strongly related to the properties of the underlying correlation function $\rho$ and hence the associated spectral measure $\mu_{\rho}$.
Conversely, some recent papers provide examples of non-stationary Gaussian entries for which the universal asymptotics \eqref{eq.asym.univ} does not hold, namely by considering palindromic Gaussian entries as in \cite{Pir19b}, or special pairwise block Gaussian entries in \cite{Pir19a}. 

\par 
\medskip

Investigating further the relations between the asymptotics of the number of real zeros and the underlying spectral measure is the main object of the present article. 
Unless otherwise stated, throughout the whole article we shall assume that the spectral measure $\mu_\rho$ is not purely singular with respect to the Lebesgue measure and we will denote by $\psi_\rho$ its absolutely continuous component, i.e. the Radon--Nikodym derivative $\psi_\rho:=d \mu_\rho/d \lambda\neq 0$.
In the case where the measure $\mu_\rho$ is purely absolutely continuous and provided that the spectral density $\psi_{\rho}$ is $\mathcal C^1$ with a H\"older derivative, we establish that $n^{-1} \Esp\left[\mathcal{N}(f_n,[0,2\pi])\right]$ converges as $n$ goes to infinity, and quite surprisingly that the limit is a linear function of $\lambda\left(\{\psi_\rho=0\}\right)$, the Lebesgue measure of the nodal set associated with the spectral density, see Theorems \ref{theo.main} and \ref{theo.main.C1} below. In particular, this limit does not depend on the shape of the spectral density and neither on the speed of decay of the correlation $\rho$. Otherwise, in the case where $\log(\psi_\rho)$ is integrable, whatever the singular part of $\mu_\rho$ is, we establish that the universal asymptotics \eqref{eq.asym.univ} holds, see Theorem \ref{thm.asym.mean} below. This can be seen as a strong improvement of the assumptions required in \cite{ADP19} since (i) we do not require anymore any assumption of continuity of $\psi_\rho$, (ii) we do not require that it is bounded from below but instead that it is logarithmically integrable and (iii) we may deal with spectral measures having possibly a non zero singular part. Moreover, with some additional regularity and integrability assumptions on the spectral density, we reinforce the convergence in an almost sure sense, see Theorem \ref{thm.as.asym}. 

\par
\medskip
Contrarily to the main methods used in the previously quoted literature, which usually rely crucially on the celebrated Kac--Rice formula, our approach to get the aforementioned extensions of the main result of \cite{ADP19} uses a significantly different strategy. Namely we exploit the following equality
\[
\frac{\mathcal{N}(f_n,[0,2\pi])}{n}=\frac{1}{2\pi}\int_0^{2\pi} \mathcal{N}\left(f_n,\left[x,x+\frac{2\pi}{n}\right]\right) dx,
\]

which enables one to reinterpret the quantity of interest as the expected number of roots of $g_n(u)=f_n\left(X+\frac{u}{n}\right)$ in the set $[0,2\pi]$ and where the expectation is computed through and independent and uniformly distributed random variable $X$ on $[0,2\pi]$. Bearing this in mind, we then study the limit in distribution of $g_n(\cdot)$ in the functional space $\mathcal{C}^1\left([0,2\pi]\right)$ endowed with the $\mathcal{C}^1$ topology. This allows to deal with the expected number of roots of the limit process which in turn leads to the universal asymptotics \eqref{eq.asym.univ}. Quite remarkably, it seems that this strategy provides more precise statements than the one obtained via the Kac--Rice strategy as bounding the integrand in the Kac--Rice integral usually imposes non necessary assumptions on the spectral density $\psi_\rho$ which are roughly due to the presence of some denominator that must be bounded from below in the estimates. We stress that this approach is inspired by Theorem 3.1.1 of Salem and Zygmund in \cite{SZ54} and by the point of view and strategy adopted in \cite{AP19} in the case of independent coefficients. 
\newpage

\subsection{Main results and comments}\label{sec.stat.res}
The first main result of the article exhibits a remarkable interplay between the Lebesgue measure of the vanishing locus of the spectral density $\psi_{\rho}$  and the asymptotics of the number of real zeros. In particular, if $\psi_{\rho}$ vanishes on a set of positive Lebesgue measure, the asymptotics of the normalized expected number of real zeros is not universal, i.e. differs from the one obtained in the independent case.

\begin{thm} \label{theo.main}
Suppose that $\mu_\rho(dx)=\psi_\rho(x) dx$ where the spectral density $\psi_{\rho}$ is $\mathcal{C}^1$ with H\"{o}lder derivative on an open set of full Lebesgue measure, then we have
\[
\lim_{n \to +\infty} \frac{\Esp\left[\mathcal{N}(f_n,[0,2\pi])\right]}{n}= \frac{\lambda(\{\psi_{\rho}=0\})}{\pi \sqrt{2}} + \frac{2\pi - \lambda(\{\psi_{\rho}=0\})}{\pi\sqrt{3}}.
\]
\end{thm}

With some mild assumptions on the topology of the nodal set $\{\psi_\rho=0\}$, one can moreover relax the above assumptions on the regularity of $\psi_\rho$.

\begin{thm}\label{theo.main.C1}
Suppose that $\mu_\rho(dx)=\psi_\rho(x) dx$ where the spectral density $\psi_{\rho}$ is piecewise continuous and that its nodal set can be decomposed as a finite union of intervals and points
\[
\{\psi_\rho=0\}=\bigcup_{i=1}^p [a_i,b_i] \cup \bigcup_{j=1}^q \{c_j\}.
\]
Then, the same asymptotics holds
\[
\lim_{n \to +\infty} \frac{\Esp\left[\mathcal{N}(f_n,[0,2\pi])\right]}{n}= \frac{\lambda(\{\psi_{\rho}=0\})}{\pi \sqrt{2}} + \frac{2\pi - \lambda(\{\psi_{\rho}=0\})}{\pi\sqrt{3}}.
\]
\end{thm}

\begin{rmk}
As mentioned above, these two first results show that the asymptotics of the number of real zeros do not particularly depend on the decay of the correlation function $\rho$, namely 
\begin{itemize}
\item under the hypotheses of Theorem \ref{theo.main} or \ref{theo.main.C1}, choosing the density $\psi_{\rho}$ as a smooth function with a compact support strictly included in $(0,2\pi)$, the associated correlation function $\rho$ then decays arbitrarily fast at infinity and the asymptotics of the number of zeros is still non-universal since it differs from the one given by Equation \eqref{eq.asym.univ}. 
\item in the opposite case, as detailed in the discussion after Theorem 1 in \cite{ADP19}, there exists some correlation function $\rho$ with arbitrarily slow decay at infinity such that the nodal asymptotics is universal. 
\end{itemize}
It is also remarkable that the large degree asymptotics of the expected number of zeros depends on $\psi_{\rho}$ only through the Lebesgue measure of the its nodal set. Two spectral densities with different shapes and possibly disjoint supports will yield to the same asymptotics as soon as the Lebesgue measure of their zero sets coincide.
\end{rmk}

Choosing in particular the spectral density $\psi_{\rho}$ of the form $\psi_{\rho}(x)=\frac{1}{2a}\mathds{1}_{[-a,a]}(x)$ with $a \in (0,\pi)$, the last Theorem \ref{theo.main.C1} yields the following corollary, which in fact can be seen as a first step towards the more general situation covered by Theorems \ref{theo.main} or \ref{theo.main.C1}.
\begin{cor} \label{cor.creneau}
Suppose that  $\rho(k)= \sin(ka)/ka$ i.e. $\mu_\rho(dx)=\psi_\rho(x) dx$ with $\psi_{\rho}(x)=\frac{1}{2a}\mathds{1}_{[-a,a]}(x)$ with $a \in (0,\pi)$, then we have
\[\
\lim_{n \to +\infty} \frac{\Esp[\mathcal{N}(f_n,[0,2\pi])]}{n}=\frac{2\pi-2a}{\pi\sqrt{2}}+ \frac{2a}{\pi\sqrt{3}}.
\]
\end{cor}

\begin{rmk}\label{rmk-densite-creneau}
The above corollary entails in particular that for any $\ell \in \left [\frac{2}{\sqrt{3}}, \sqrt{2}\right)$, there exists a spectral density $ \psi_{\rho}$ such that 
\[\
\lim_{n \to +\infty} \frac{\Esp[\mathcal{N}(f_n,[0,2\pi])]}{n}=\ell.
\]
Hence in the case of a spectral measure admitting a spectral density with respect to the Lebesgue measure, the expected number of real zeros has a whole spectrum of possible values. This has to be compared with the purely discrete case studied in \cite{Pau20}, where as recalled above, it is shown that, choosing $\mu_{\rho}$ as purely atomic of the form $\mu_{\rho}=\frac{1}{2}(\delta_{\alpha}+\delta_{-\alpha})$, with $\alpha \not \in \pi \Q$ also yields non-universal nodal asymptotics ranging this time in the interval $[\sqrt{2}, 2]$. 
 \end{rmk}
 
\begin{rmk}\label{rmk-pirhadi}
Note also that, a direct corollary of the proof of Theorem \ref{theo.main} is that if $\psi_{\rho}$ is $\mathcal{C}^1$ with H\"{o}lder derivative on an open set of full Lebesgue measure and that $\lambda\left(\{\psi_{\rho}=0\}\right)>0$, then we have 
\[\
\liminf_{n \to +\infty} \frac{\Esp[\mathcal{N}(f_n,[0,2\pi])]}{n} \geq \frac{2}{\sqrt{3}}.
\]
This lower bound goes in the direction of the conjecture raised in \cite{Pir19a}, asserting that the universal limit $2/\sqrt{3}$ is the minimum possible value for the asymptotics of the expected number of real zeros when dealing with Gaussian  trigonometric polynomials with dependent coefficients.
\end{rmk}

The second main result of this article consists in relaxing the hypotheses of \cite{ADP19} on the spectral density in order to obtain a universal asymptotics for the expected number of real zeros. Recall that in the latter reference, the spectral density $\psi_{\rho}$ was assumed to be continuous and lower bounded by a positive constant. We establish here that the expected asymptotics is universal as soon as $\psi_{\rho}$ satisfies a $\log-$integrability condition, {and with no condition on the singular component $\mu_{\rho}^s$}.
\par
\medskip

\begin{thm}\label{thm.asym.mean} {Suppose that $\mu_\rho(dx)=\mu_{\rho}^s +\psi_\rho(x) dx$} and assume that there exists $\eta \in (0,1)$ such that 
\[
\log(\psi_{\rho})\in L^{1+\eta}([0,2\pi]),
\]
 then
\[
\lim_{n\to +\infty}\Esp\left[ \frac{\mathcal{N}(f_n,[0,2\pi])}{n}\right]= \frac{2}{\sqrt{3}}.
\]
\end{thm}
\par
\medskip
\begin{rmk}
Such a $\log-$integrability condition is rather natural if we keep in mind that the limiting case where $\eta=0$ corresponds to the fact that the stationary Gaussian sequences $(a_k)_{k \geq 0}$ and $(b_k)_{k \geq 0}$ are ``purely non-deterministic", see e.g. Chapter 1, p. 9 of \cite{Pal07} and Section 2.1 of \cite{HNTX14}, or equivalently, that these sequences have ``finite entropy", see e.g. \cite{Iha00}. 
\end{rmk}
\par
\medskip
In the case of an absolutely continuous spectral measure, with additional assumptions of regularity and integrability on the spectral density $\psi_{\rho}$, the previous result  can be even strengthened in an almost-sure asymptotics. 
\par
\medskip
\begin{thm}\label{thm.as.asym}
Assume that $\mu_{\rho}$ is purely absolutely continuous, i.e. $\mu_{\rho}(dx)=\psi_{\rho}(x)dx$ and that
\begin{description}
\item[A.1] there exists $\alpha>0$ such that $\psi_{\rho}$ satisfies a Besov regularity property of order $\alpha$, i.e.
\[
\text{for}~\delta>0,~\sup_{|h|\leq \delta} \|\psi_{\rho}(\cdot+h)+\psi_{\rho}(\cdot-h)-2\psi_{\rho}(\cdot)\|_{L^{1}([0,2\pi])}= O(\delta^{\alpha}),\]
\item[A.2] there exists $\gamma >0$ such that 
\[
\frac{1}{\psi_{\rho}(X)} \in L^{\gamma}([0,2\pi]).
\]
\end{description}
Then, $\Prob$-almost surely, we have
\begin{equation} \label{as.asym}
\lim_{n\to +\infty} \frac{\mathcal{N}(f_n,[0,2\pi])}{n}= \frac{2}{\sqrt{3}}.
\end{equation}
\end{thm}
\begin{rmk}
Note that the Besov-type Assumption A.1 is naturally satisfied if the spectral density $\psi_{\rho}$ is H\"older continuous, and that Assumption A.2 on the existence of a negative moment implies the log-integrability condition required in the above Theorem \ref{thm.asym.mean}. Note moreover that these two hypotheses are weaker than the ones made \cite{ADP19}. In particular, Theorem \ref{thm.as.asym} is valid in the emblematic case of Gaussian coefficients that are increments of fractional Brownian motion of any Hurst parameter. Indeed, in this last case, the spectral density in smooth except at zero and is lower bounded by a positive constant, so that it trivially admits negative moments. 
\end{rmk}
\par 
\medskip
The proofs of Theorems \ref{thm.asym.mean} and \ref{thm.as.asym} above are based on Central Limit Theorems of Salem--Zygmund type, which extend the recent results of \cite{AP19} obtained in the independent case to the setting of dependent coefficients. Indeed, in the spirit of Theorem 3.1.1 of \cite{SZ54}, almost surely in the random coefficients, when evaluated at a uniform and independent random point $X$, the sequence $f_n(X)$ converges in distribution to a mixture of Gaussian variables. Note that in the two following statements, no assumption is required on the spectral density. 
\par 
\medskip
\begin{thm}\label{prop.conv.fc} $\Prob$-almost surely, for all $t\in \mathbb R$, we have
\[
\lim_{n \to +\infty} \Esp_X\left[e^{itf_n(X)}\right]=\frac{1}{2\pi} \int_0^{2\pi} e^{-\frac{t^2}{2}{\times 2\pi}\psi_{\rho}(x)}dx =\Esp_{X ,N}\left[e^{it \sqrt{{2\pi}\psi_{\rho}(X)}N}\right],
\]
where $N$ is a standard Gaussian variable, independent of $X$.
In other words, $\Prob$-almost surely, the sequence of random variables $f_n(X)$ converges in distribution under $\Prob_X$ towards $\sqrt{{2\pi}\psi_{\rho}(X)}N$. 
\end{thm}
\begin{rmk}
We first observe that, in the independent case where $\rho(k)=\delta_0(k)$ i.e. $\psi_\rho \equiv 1/2\pi$ on $[-\pi, \pi]$, we recover the Central limit Theorem by Salem--Zygmund with a standard Gaussian limit distribution. Note that the same central asymptotics was recently observed in another model of 
dependent coefficients obtained via arithmetic functions,  see \cite{BNR20}. 
Going back to our model, in the non-independent case, i.e. if  $\psi_\rho$ in non-constant,  then the limit in distribution in Theorem \ref{prop.conv.fc} is not Gaussian anymore but rather a continuous mixture of Gaussian distributions and thus provides a natural instance where Salem--Zygmund central convergence fails.  
Finally, note that if $\psi_{\rho} \equiv 0$, i.e. $\mu_{\rho}$ is purely singular, the above limit is trivial so that another renormalization is needed.
\end{rmk}
\par 
\medskip
This last result is in fact the consequence of the following more general functional Central Limit Theorem which is the analogue of Theorem 3 in \cite{AP19}. As above, no assumption is required here on the spectral measure $\mu_{\rho}$.

\par 
\medskip
\begin{thm} \label{thm.SZ.func} 
$\Prob$-almost surely, the localized process $(g_n(t))_{t \in [0,2\pi]}:=\left(f_n\left(X+\frac{t}{n}\right)\right)_{t \in [0,2\pi]}$ converges in distribution under $\Prob_X$ for the $\mathcal{C}^1$ topology to a limit process $(g_{\infty}(t))_{t\in[0,2\pi]}$ given by 
$$g_{\infty}:=\sqrt{{2\pi}\psi_{\rho}(X)}N,$$
where $N=(N_t)_{t\in[0,2\pi]}$ is the standard Gaussian process with $\sin_c$ covariance function, independent of the uniform variable $X$.
\end{thm}

The plan of the article is the following. In the next Section \ref{sec.nonuniversal}, after recalling some basics on trigonometric kernels and Kac--Rice formula, we give the detailed proofs of Theorem \ref{theo.main} and \ref{theo.main.C1}. Then, Section \ref{sec.SZCLT} is devoted to the proofs of both Theorems \ref{prop.conv.fc} and \ref{thm.SZ.func}, i.e. the Central limit Theorems à la Salem--Zygmund. Finally, in the last Section \ref{sec.nodal}, we detail how these last theorems allow to deduce the universality of the nodal asymptotics, first under expectation and then in an almost sure sense. For the readability of the paper, some of the technical estimates and lemmas have been postponed in Appendix in Section \ref{sec.appendix}.

\section{Nodal asymptotics with a vanishing spectral  density}\label{sec.nonuniversal}

The object of this section is to give the detailed proof of Theorems \ref{theo.main} and \ref{theo.main.C1}. In order to do so, let us first recall some basics on trigonometric kernels and the celebrated Kac--Rice formula, which allows to express the expected number of zeros as an explicit functional of the correlation functions.

\subsection{Preliminaries on trigonometric kernels and Kac--Rice formula}\label{preli-Kac}

Recall that $\rho$ is the covariance function of the independent stationary sequences $(a_k)$ and $(b_k)$, i.e. $\mathbb E[a_k a_l]=\mathbb E[b_k b_l] =\rho(k-l)$ and that $\mu_{\rho}$ is the associated symmetric spectral probability measure via Bochner--Herglotz Theorem, i.e.
\[
\rho(k)=\int_{-\pi}^{\pi} e^{-i k x} \mu_{\rho}(dx).
\]
We adopt the following normalization conventions, if $f$ and $g$ bounded $2\pi-$periodic functions, $k \in \mathbb Z$ 
\[
f \ast g(x) = \frac{1}{2\pi} \int_{-\pi}^{\pi} f(x-y)g(y)dy, \quad \hat{f}(k)= \frac{1}{2\pi}  \int_{-\pi}^{\pi} e^{-i k x} f(x)dx, 
\]
so that 
\[
\widehat{f \ast g}(k) = \hat{f}(k) \hat{g}(k).
\]
The correlation functions associated to our model can be expressed in terms of trigonometric kernels. Namely, if the function $f_n$ is given by Equation \eqref{def.model}, then we have
\begin{equation}
\begin{array}{ll}
\mathbb{E}\left[f_n^2(x)\right]& \displaystyle{=\frac{1}{n}\sum_{k,l=1}^n \Esp\left[a_k a_l\right]\cos(k  x)\cos(l x)+\Esp\left[b_k b_l\right]\sin(k  x)\sin(l x)}\\
& \displaystyle{=\frac{1}{n}\sum_{k,l=1}^n \rho(k-l)\cos((k-l)x)
=\sum_{r=-n}^n \left(1-\frac{|r|}{n}\right)\rho(r)e^{i r x}= 2\pi \, K_n \ast \mu_\rho(x),}\\
\end{array}
\label{eq.corel1}
\end{equation}
where $K_n$ denotes the celebrated Fejer Kernel given by
\[
K_n(x):=  \sum_{r=-n}^n \left(1-\frac{|r|}{n}\right)e^{i r x}= \frac{1}{n}\left(\frac{\sin\left(\frac{nx}{2}\right)}{\sin\left(\frac{x}{2}\right)}\right)^2.
\]
Similarly, setting $ \alpha_n:=\frac{6}{(n+1)(2n+1)}\sim \frac{3}{n^2}$, we have 
\begin{equation}
\mathbb{E}\left[{f'_n}(x)^2\right]=\frac{2\pi}{\alpha_n} L_n\ast \mu_\rho(x),
\label{eq.corel2}
\end{equation}
where
\[
\begin{array}{l}
\displaystyle{L_n(x)}  \displaystyle{:=\frac{\alpha_n}{n}\left|\sum_{k=0}^n k e^{i k x}\right|^2 =  \frac{\alpha_n}{n}\sum_{k,\ell=1}^{n}k\ell \cos((k-\ell)x)} 
 =\displaystyle{\frac{\alpha_n}{n} \frac{(n+1)^2}{4 \sin^2\left(\frac x 2\right)}\left|1-\frac{(1-e^{i (n+1)x})e^{-i n x}}{(n+1)(1-e^{i x})}\right|^2.}
\end{array}
\]
The functions $K_n$ and $L_n$ satisfy
\[
\frac{1}{2\pi} \int_{-\pi}^{\pi} K_n(x)dx = 1, \qquad \frac{1}{2\pi} \int_{-\pi}^{\pi} L_n(x)dx = 1.
\]
\noindent
Finally we will also have 
\begin{equation}
\Esp\left[f_n(x) f'_n(x)\right]=\pi \, K_n'\ast\mu_\rho(x),
\label{eq.corel3}
\end{equation}
with 
\[
K_n'(x) = \frac{2}{n} \left( \frac{\sin(nx/2)}{\sin(x/2)}\right)\left( \frac{n \cos(nx/2)}{2\sin(x/2) } -  \frac{ \sin(nx/2)\cos(x/2)}{2\sin(x/2)^2 } \right).
\]
Both functions $K_n$ and $L_n$ are in fact good regularizing kernels as shown in  Lemma 1 of \cite{ADP19}. In the sequel, we will make use of the following uniform estimates.
\medskip
\begin{lma}\label{lm.prop.L}The two kernels $K_n$, $L_n$ are non-negative and even, the derivative $K_n'$ is odd and these three functions satisfy the following upper bounds
\begin{enumerate}
\item For all integer $n \geq 1$
\[
\sup_{x \in [-\pi,\pi]} K_n(x) \leq  n, \quad \sup_{x\in [-\pi,\pi]} L_n(x) \leq  n, \quad \sup_{x\in [-\pi,\pi]} \frac{|K_n'(x)|}{n} \leq   n.
\]
\item There exists a constant $C>0$ such that uniformly in $x \in [-\pi,\pi]$
 \[
 \begin{array}{l}
 \displaystyle{K_n(x) \leq \frac{C}{n x^2} , \quad  L_n(x) \leq C\left(\frac{1}{n x^2}  + \frac{1}{n^2|x| ^3} +\frac{1}{n^3x^4}\right), \quad \frac{|K_n'(x)|}{n}  \leq C \left(\frac{1}{n x^2}  + \frac{1}{n^2|x|^3}\right)}.
 \end{array}
 \]
\end{enumerate}
\end{lma}

\begin{proof}
From the expression of the functions in terms complex exponentials, we have directly 
\[
K_n(x) \leq K_n(0)=n, \quad L_n(x) \leq L_n(0)=\frac{6n(n+1)}{4(2n+1)} \leq n, 
\]
and since $u(1-u) \leq 1/4$ for $u \in [0,1]$, in the same way we get
\[
 \sup_{x\in [-\pi,\pi]} |K_n'(x)| \leq 2n \sum_{r=1}^n \frac{r}{n} \left(  1-\frac{r}{n} \right) \leq n^2/2 \leq n^2.
\]
By concavity of the function $x \mapsto \sin(x)$ on $[0, \pi/2]$, we have $|\sin(x/2)| \geq \frac{|x|}{\pi}$ for $x \in [-\pi, \pi]$, and injecting this estimates in the denominators of $K_n$, $L_n$ and $K_n'$, we get uniformly in $x \in [-\pi, \pi]$
\[
K_n(x)\leq \frac{\pi^2}{n x^2}, \quad L_n(x) \leq \frac{\alpha_n}{n} \frac{\pi^2(n+1)^2}{4 x^2} \left( 1 + \frac{\pi}{n |x|}\right)^2, \quad 
\frac{|K_n'(x)|}{n} \leq \frac{\pi^2}{n x^2} + \frac{\pi^3}{n^2 |x|^3}.
\]
\end{proof}

Suppose that the spectral measure $\mu_{\rho}$ admits the decomposition $\mu_{\rho} = \mu_{\rho}^s + \psi_{\rho}(x)dx$, where $\mu_{\rho}^s$ is its singular part and $\psi_{\rho}$ is the density component, with the convention that $\psi_{\rho}=0$ if 
$\mu_{\rho}$ is purely singular. From the above estimates on the trigonometric kernels, one can then deduce the following Fej\'er--Lebesgue type asymptotics.

\begin{lma} \label{lm.fej.leb}
For Lebesgue almost every $ x \in [-\pi,\pi]$, 
\[
\lim_{n \to +\infty} K_n \ast \mu_{\rho}(x)=\psi_{\rho}(x), \quad \lim_{n \to +\infty} L_n \ast \mu_{\rho}(x)=\psi_{\rho}(x), \quad \lim_{n \to +\infty}  \frac{1}{n} K_n'(t) \ast \mu_{\rho}(x)=0.
\]
\end{lma}

\begin{proof}The first estimates is precisely the celebrated Fej\'er--Lebesgue Lemma, see e.g. Theorem 8.1,  page 105 of \cite{Zyg03}. For the sake of self-containedness, let us detail the proof for the two other kernels $L_n$ and $K_n'/n$. 
If we set 
\[
\varphi_x(t):= \mu_{\rho}([0,x+t])-\mu_{\rho}([0,x-t])-2t \psi_{\rho}(x),
\]
then $t \mapsto \phi_x(t)$ is of bounded variation and for $\varepsilon>0$, we denote by $\Phi_x(\varepsilon)$ its total variation on $[-\varepsilon,\varepsilon]$, namely
\[
\Phi_x(\varepsilon):=\int_{-\varepsilon}^{\varepsilon}|d\varphi_x(u)|.
\] 
By Theorem 8.4, p 106 of \cite{Zyg03}, Lebesgue almost all point $x$ in $[-\pi,\pi]$ belong to the set 
\[
E:=\left \lbrace x \in [-\pi,\pi], \; \Phi_x(\varepsilon)=o(\varepsilon)\;  \text{as} \;  \varepsilon \; \text{ goes to zero}\right \rbrace. 
\]
From now on, we fix $x \in E$. Since $\frac{1}{2\pi}\int_{-\pi}^{\pi} L_n(t)dt=1$, we have the representation:
\begin{equation}\label{eq.zyg0}
L_n \ast \mu_{\rho}(x)-\psi_{\rho}(x) = \frac{1}{2} \times \left( \frac{1}{2\pi} \int_{-\pi}^{\pi} L_n(t) d\varphi_x(t) \right). 
\end{equation}
By the first point of Lemma \ref{lm.prop.L}, we have $L_n(t) \leq n$ uniformly so that as $n$ goes to infinity
\begin{equation}\label{eq.zyg00}
\left| \int_{-\frac{1}{n}}^{\frac{1}{n}} L_n(t) d\phi_x(t)\right| \leq n \times \int_{-\frac{1}{n}}^{\frac{1}{n}}  \left| d\varphi_x(t)\right| =n \,  \Phi_x\left( \frac{1}{n} \right) =o(1).
\end{equation}
Moreover, by the second point of Lemma \ref{lm.prop.L}, we also have
\begin{equation}\label{eq.zyg}
\left| \int_{\frac{1}{n} \leq |t| \leq \pi}  L_n(t) d\varphi_x(t) \right|\leq C\left(\int_{\frac{1}{n} \leq |t| \leq \pi} \frac{|d\varphi_x(t)|}{n t^2} +\int_{\frac{1}{n} \leq |t| \leq \pi} \frac{|d\varphi_x(t)|}{n^2 t^3} +\int_{\frac{1}{n} \leq |t| \leq \pi} \frac{|d\varphi_x(t)|}{n^3t^4}\right).
\end{equation}
Then, integrating by parts the first term on the right hand side of \eqref{eq.zyg}, we obtain
\[
\begin{array}{ll}
\displaystyle{\int_{\frac{1}{n} \leq |t| \leq \pi} \frac{|d\varphi_x(t)|}{n t^2} }& = \displaystyle{\left[\Phi_x(t) \times \frac{1}{nt^2}\right]_{1/n}^{\pi}+2\int_{\frac{1}{n} \leq |t| \leq \pi} \frac{\Phi_x(t)dt}{n t^3} }.
\end{array}
\]
On the one hand, since $x \in E$, as $n$ goes to infinity, we have
\[
\left[\Phi_x(t) \times \frac{1}{nt^2}\right]_{1/n}^{\pi} = O \left(\frac{1}{n} \right) + n \,  \Phi_x\left( \frac{1}{n} \right)=o(1), 
\]
and on the other hand, for all $\delta>0$ and $n$ large enough,  we have
\[
\int_{\frac{1}{n} \leq |t| \leq \pi} \frac{\Phi_x(t)dt}{n t^3}  = \underbrace{\int_{\frac{1}{n} \leq |t| \leq \delta} \frac{\Phi_x(t)}{t} \times \frac{dt}{n t^2} }_{o(\delta)\times O(1)} + \underbrace{\int_{\delta \leq |t| \leq \pi} \frac{\Phi_x(t)}{t} \times \frac{dt}{n t^2} }_{=O\left( \frac{1}{n \delta^2}\right)},
\]
so that letting first $n$ go to infinity and then $\delta$ to zero, we get
\[
\int_{\frac{1}{n} \leq |t| \leq \pi} \frac{|d\varphi_x(t)|}{n t^2} = o(1).
\]
Proceeding in the exact same way for the two other terms on the right hand side of \eqref{eq.zyg}, we obtain that if $x \in E$, as $n$ goes to infinity
\[
\int_{\frac{1}{n} \leq |t| \leq \pi} \frac{|d\varphi_x(t)|}{n^2 t^3} = o(1), \qquad \int_{\frac{1}{n} \leq |t| \leq \pi} \frac{|d\varphi_x(t)|}{n^3 t^4} = o(1).
\]
As a result, from Equations \eqref{eq.zyg0}, \eqref{eq.zyg00} and \eqref{eq.zyg}, we get indeed that $L_n \ast \mu_{\rho}(x) -\psi_{\rho}(x)$ goes to zero as $n$ goes to infinity. The proof for the kernel $K_n'$ is very similar. Since $K_n'$ is odd, we have a similar representation 
\[ 
\frac{1}{n}K_n' \ast \mu_{\rho}(x)=\frac{1}{2} \times \left( \frac{1}{2\pi} \int_{-\pi}^{\pi} \frac{1}{n}K_n'(t) d\varphi_x(t) \right). 
\]
and one can conclude as above using the uniform estimates for $K_n'$ given by Lemma \ref{lm.prop.L}.
\end{proof}

In the case where the measure $\mu_{\rho}$ has a singular component, it is not possible to give general quantitative estimates for the convergences in Lemma \ref{lm.fej.leb}. On the opposite case, if  $\mu_{\rho}$ is absolutely continuous $\mu_{\rho}(dx)=\psi_{\rho}(x)dx$ and $\psi_{\rho}$ is regular enough, then there exists simple quantitative bounds. For $\alpha>0$, we denote by $\mathcal C^{1,\alpha}$ the set of $\mathcal C^1$ functions $\psi$ such that the derivative $\psi'$ is $\alpha-$H\"older with H\"older norm $[\psi']_{\alpha}$
\[
[\psi']_{\alpha} := \sup_{|x-y|>0} \frac{|\psi'(x)-\psi'(y)|}{|x-y|^{\alpha}}<+\infty.
\]
\begin{lma}\label{lem.speedfejer}
Suppose that $\mu_{\rho}(dx)=\psi_{\rho}(x)dx$ with $\psi_{\rho}$ of class $\mathcal C^{1,\alpha}$ with $\alpha >0$, then uniformly in $x$
\[
K_n \ast \mu_{\rho}(x)-\psi_{\rho}(x)=O(1/n), \quad L_n \ast \mu_{\rho}(x)-\psi_{\rho}(x)=O(1/n).
\]
\end{lma}

\begin{proof}
Let us note that for any $x,y \in [-\pi, \pi] $, by the mean value Theorem, there exists $c$ between $x$ and $x-y$ such that  $\psi_{\rho}(x-y)-\psi_{\rho}(x)=-y \psi_{\rho}'(c)$. Since $\psi_{\rho}'$ is a  $\alpha-$H\"older function, we have moreover $|\psi_{\rho}'(c)-\psi_{\rho}'(x)| \leq [\psi']_{\alpha}  |c-x|^{\alpha} \leq  [\psi']_{\alpha} |y|^{\alpha}$, so that 
\[
| \psi_{\rho}(x-y)-\psi_{\rho}(x)+y \psi_{\rho}'(x) | \leq  [\psi']_{\alpha} |y|^{1+\alpha}.
\]
Therefore, integrating against the even kernel $K_n$, we thus get that 
\[
|K_n \ast \mu_{\rho}(x)-\psi_{\rho}(x)| \leq \frac{[\psi']_{\alpha}}{2\pi} \int_{-\pi}^{\pi} K_n(y) |y|^{1+\alpha}dy.
\]
Now, using the uniform bounds of the second point in Lemma \ref{lm.prop.L}, we get 
\[
\int_{-\pi}^{\pi} K_n(y) |y|^{1+\alpha}dy \leq \frac{C}{n} \int_{-\pi}^{\pi}|y|^{\alpha-1}dy = O \left( \frac{1}{n} \right).
\]
In the same way, since $L_n$ is also even, we have 
\[
|L_n \ast \mu_{\rho}(x)-\psi_{\rho}(x)| \leq \frac{[\psi']_{\alpha}}{2\pi} \int_{-\pi}^{\pi} L_n(y) |y|^{1+\alpha}dy,
\]
and by Lemma \ref{lm.prop.L} again, we get 
\[
\int_{-\pi}^{\pi} L_n(y) |y|^{1+\alpha}dy = \int_{0\leq |y| \leq 1/n}  L_n(y) |y|^{1+\alpha}dy+\int_{ |y| > 1/n}  L_n(y) |y|^{1+\alpha}dy
\]
with 
\[
\begin{array}{ll} 
\displaystyle{\int_{0\leq |y| \leq 1/n}  L_n(y) |y|^{1+\alpha}dy} &\displaystyle{ \leq n \int_{0\leq |y| \leq 1/n}   |y|^{1+\alpha}dy = O\left( \frac{1}{n^{\alpha+1}} \right),}\\
\\
\displaystyle{\int_{|y| > 1/n}  L_n(y) |y|^{1+\alpha}dy} &\displaystyle{ \leq \int_{ |y| > 1/n}   |y|^{1+\alpha}\left( \frac{1}{n|y|^2}+ \frac{1}{n^2|y|^3} +\frac{1}{n^3|y|^4}\right)
dy = O\left( \frac{1}{n} \right).}

\end{array}
\]

\end{proof}

\begin{lma}\label{lem.nondeg}
Suppose that the spectral measure $\mu_{\rho} = \mu_\rho^s + \psi_{\rho}(x)dx$ is not purely singular, i.e. $\psi_{\rho}$ is positive on a set of positive Lebesgue measure, then there exists a positive constant $C$ such that, uniformly in $x \in [-\pi, \pi]$, for $n$ large enough
\[
K_n \ast \mu_{\rho}(x) \geq K_n \ast \psi_{\rho}(x) \geq C/n.
\]
\end{lma}

\begin{proof}
Since the Fej\'er kernel $K_n$ is non-negative, we have 
\[
K_n \ast \mu_{\rho}(x)  \geq K_n \ast \psi_{\rho}(x) = \frac{1}{2\pi} \int_{-\pi}^{\pi} K_n(t) \psi_{\rho}(x-t) dt.
\]
Now, we can rewrite $K_n$ as
\[
K_n(t)= \frac{1}{2n}\times \frac{1-\cos(nt)}{\sin^2(t/2)}
\]
so that, since $\sin(t/2) \leq 1$, we get
\[
K_n \ast \mu_{\rho}(x) \geq  \frac{1}{2n} \frac{1}{2\pi} \int_{-\pi}^{\pi}  \psi_{\rho}(t) dt-\frac{1}{2n} \frac{1}{2\pi}\int_{-\pi}^{\pi} \cos(nt) \psi_{\rho}(x-t)  dt.
\] 
Then, since 
\[
\int_{-\pi}^{\pi} \cos(nt) \psi_{\rho}(x-t)  dt = \cos(nx)\int_{-\pi}^{\pi} \cos(nt) \psi_{\rho}(t)  dt  + \sin(nx)\int_{-\pi}^{\pi} \sin(nt) \psi_{\rho}(t)  dt ,
\]
by Riemann--Lebesgue Lemma, we deduce that uniformly in $x$, we have as $n$ goes to infinity
\[
\left| \int_{-\pi}^{\pi} \cos(nt) \psi_{\rho}(x-t)  dt \right| \leq \left| \int_{-\pi}^{\pi} \cos(nt) \psi_{\rho}(t)  dt  \right|+ \left| \int_{-\pi}^{\pi} \sin(nt) \psi_{\rho}(t)  dt \right|= o(1).
\]
Setting $C:=\frac{1}{4\pi} \int_{-\pi}^{\pi}  \psi_{\rho}(t) dt>0$, we thus obtain that
\[
K_n \ast \psi_{\rho}(x) \geq \frac{C}{n} + o\left(\frac{1}{n}\right), 
\]
hence the result.
\end{proof}

In fact we have the more general lower bound.

\begin{lma}
Uniformly in $x \in [-\pi, \pi]$, and for any integer $n \geq 1$
\[
K_n \ast \mu_{\rho}(x) \geq \frac{1}{4\pi n} \left(  1- \rho(n)\cos(nx) \right).
\]
\end{lma}

\begin{proof}
As above, we use the fact that uniformly in $t$
\[
K_n(t)\geq \frac{1}{2n}\times \left( 1-\cos(nt)\right)
\]
to deduce that 
\[
\begin{array}{ll}
\displaystyle{K_n \ast \mu_{\rho}(x)} & \geq \displaystyle{\frac{1}{2n} \frac{1}{2\pi} \times \int_{-\pi}^{\pi} \left( 1-\cos(n(t-x))\right) d\mu_{\rho}(t)}\\
\\
& \displaystyle{=\frac{1}{2n} \frac{1}{2\pi} \left( 1 - \cos(nx) \int_{-\pi}^{\pi} \cos(nt) d\mu_{\rho}(t) - \sin(nx) \int_{-\pi}^{\pi} \sin(nt) d\mu_{\rho}(t)\right) }\\
\\
& \displaystyle{=\frac{1}{2n} \frac{1}{2\pi} \left( 1 - \cos(nx) \rho(n) \right). }
\end{array}
\]
\end{proof}

We now recall the celebrated Kac--Rice formula which, in the present context, allows to express the expected number real zeros of $f_n$ as a simple functional of the covariance function of the Gaussian vector $(f_n(x),f_n'(x))$. 

\begin{prop}Suppose that the spectral measure $\mu_{\rho} = \mu_\rho^s + \psi_{\rho}(x)dx$ is not purely singular, i.e. $\psi_{\rho}$ is positive on a set of positive Lebesgue measure, then for any $[a,b] \subset [-\pi,\pi]$
\begin{equation}\label{Kac-Rice formula}
\begin{array}{ll}
\displaystyle{\frac{\Esp\left[\mathcal{N}\left(f_n,[a,b]\right)\right]}{n}} & =\displaystyle{\frac{1}{\pi}\int_a^b \sqrt{\frac{\Esp[{f'_n}^2(x)]}{n^2\Esp[f_n^2(x)]}-\left(\frac{\Esp[f_n(x)f'_n(x)]}{n\Esp[f_n^2(x)]}\right)^2}dx}\\
\\
& =\displaystyle{\frac{1}{\pi}\int_a^b \sqrt{\frac{1}{n^2 \alpha_n }\frac{L_n \ast \mu_{\rho}(x) }{ K_n \ast \mu_{\rho}(x)}-\left(\frac{K_n' \ast \mu_{\rho}(x)}{2n K_n \ast \mu_{\rho}(x)}\right)^2}dx}.
\end{array}
\end{equation}
\end{prop}

\begin{proof}
In the Gaussian context, Kac--Rice formula holds as soon as the function $x \mapsto f_n(x)$ is regular enough and the variables $(f_n(x))_{x \in [-\pi, \pi]}$ are non-degenerate, see e.g. \cite{AW09}. Here, we observe that $f_n(\cdot)$ has $\mathcal{C}^1$ paths and besides, in virtue of Lemma \ref{lem.nondeg}, for every $x\in[-\pi,\pi]$, we have the lower bound on the variance $\mathbb{E}[f_n^2(x)]=2\pi K_n\ast \mu_\rho(x)>C/n$, hence the validity of the formula. The equality between the two integrals is due to the expression of the covariance functions in terms of convolutions, see Equations \eqref{eq.corel1}, \eqref{eq.corel2}, \eqref{eq.corel3} above.
\end{proof}

\begin{cor}\label{cor.low}
Suppose that the spectral measure $\mu_{\rho}$  is such that the spectral density $\psi_{\rho}$ is positive almost everywhere on an interval $[a,b]$, then 
\[
\liminf_{n \to +\infty} \frac{\Esp\left[\mathcal{N}\left(f_n,[a,b]\right)\right]}{n}\geq \frac{b-a}{\pi\sqrt{3}}.
\]
\end{cor}

\begin{proof}
The result is an immediate consequence of the representation formula \eqref{Kac-Rice formula}, associated with the Fej\'er--Lebesgue type estimates of Lemma \ref{lm.fej.leb}. Indeed, under the assumption that $\psi_{\rho}(x)>0$ for almost  all $x \in [a,b]$, as $n$ goes to infinity we have
\[
\frac{1}{n^2 \alpha_n } \to \frac{1}{3}, \quad \frac{L_n \ast \mu_{\rho}(x) }{ K_n \ast \mu_{\rho}(x)} \to \frac{\psi_{\rho}(x)}{\psi_{\rho}(x)}=1, \quad \frac{K_n' \ast \mu_{\rho}(x)}{n K_n \ast \mu_{\rho}(x)}\to 0,
\]
and one concludes by using Fatou Lemma in Equation \eqref{Kac-Rice formula}.
\end{proof}

\begin{rmk}As the lower bound mentioned in Remark \ref{rmk-pirhadi} above, the last Corollary \ref{cor.low} is a local estimate which goes in the direction of the conjecture raised by Pirhadi in \cite{Pir19a}, on the minimal value of the expected number of zeros of Gaussian trigonometric polynomials with dependent coefficients. Note that the result is independent of the singular component of the spectral measure.
\end{rmk}

\newpage
\subsection{Towards a non-universal asymptotics}
We can now give the detailed proofs of both Theorems \ref{theo.main} and \ref{theo.main.C1} stated in the introduction.

\subsubsection{An illustrating example}
In order to highlight the main ideas behind the proof, let us first establish Corollary \ref{cor.creneau}, i.e. let us examine as a preliminary step the specific case where 
$\rho(k)= \sin(ka)/ka$ i.e. $\mu_\rho(dx)=\psi_\rho(x) dx$ with $\psi_{\rho}(x)=\frac{1}{2a}\mathds{1}_{[-a,a]}(x)$ with $a \in (0,\pi)$.
\par
\medskip
The proof relies on the explicit Kac--Rice formula \eqref{Kac-Rice formula} and on the compilation of three separate regimes: (i) roots far from $[-a,a]$, (ii) roots inside $[-a,a]/[-\delta,\delta]$ (with $\delta<<1$) and (iii) roots at some neighborhood of $\{0,-a,a\}$ of size $\delta$. Each of these regimes will require a different argument.
\par
\medskip
\noindent
\underline{i) Far from $[-a,a]$:}
\par
\medskip

Let $\delta>0$ and $x \in [-\pi,\pi]$ such that $\text{dist}(x,[-a,a])\ge\delta$ and where the distance is modulo $2\pi$.
Since in our case, the spectral measure is given by $\mu_{\rho}(dt)=\frac{1}{2\pi}\mathds{1}_{[-a,a]}(t)dt$, Equation \eqref{eq.corel1} reads
\[
\Esp[f_n(x)^2]={2\pi} \, K_n \ast \mu_{\rho}(x)=\frac{1}{2a}\int_{-a}^{a} K_n(x-y)dy=\frac{1}{2a}\int_{x-a}^{x+a}K_n(y)dy,
\]
where $K_n$ is the Fej\'er kernel which can be alternatively written as
\[K_n(x)= \frac{1}{2n}\frac{1}{\sin^2(x/2)}-\frac{1}{2n}\frac{\cos(nx)}{\sin^2(x/2)}.\]
In other words,
\[n\Esp[f_n(x)^2]=\frac{1}{2}\times \left( \underbrace{\frac{1}{2a}\int_{x-a}^{x+a}\frac{dy}{\sin^2(y/2)}}_{\ge 1} -\frac{1}{2a}\int_{x-a}^{x+a} \frac{\cos(ny)}{\sin^2(y/2)}\right).\]
We shall examine the second integral and establish that it is negligible with respect to the first one (which is greater than $1$) as $n\to\infty$. To do so,  we must rely on the Riemann--Lebesgue Lemma. More precisely, by integration by part, since the function $y \mapsto \frac{1}{\sin^2(y/2)}$ is $\mathcal{C}^{\infty}$ on any interval that does not contain zero, we have
\[ \frac{1}{2a} \int_{x-a}^{x+a}\frac{\cos(ny)}{\sin^2(y/2)}dy= \left[\frac{\sin(ny)}{n}\frac{1}{\sin^2(y/2)}\right]_{x-a}^{x+a}+\frac{1}{n}\int_{x-a}^{x+a} \sin(ny)\frac{\cos(y/2)}{\sin^3(y/2)}dy.\]
We observe that for all $y \in [x-a,x+a]$, we have $\text{dist}(y,2\pi\mathbb{Z})\ge\text{dist}(x,[-a,a])\ge \delta$ so that $|\sin(y/2)| \geq \inf_{u\in [\delta,\pi-\delta]}|\sin(u)|\ge \frac 1 2\delta$. We thus obtain that
\[ \frac{1}{2a} \int_{x-a}^{x+a}\frac{\cos(ny)}{\sin^2(y/2)}dy = O\left(\frac{1}{n\delta^3}\right).\] 
In the same way, Equation \eqref{eq.corel2} reads
\[\Esp[f_n'(x)^2]=\frac{{2\pi}}{\alpha_n}L_n\ast \mu_{\rho}(x)=\frac{1}{\alpha_n}\frac{1}{2a}\int_{x-a}^{x+a}L_n(y)dy,\]
where we recall that
\[L_n(x):= \frac{\alpha_n}{n}\left|\sum_{k=0}^{n}ke^{i k x}\right|^2=\frac{\alpha_n}{n}\frac{(n+1)^2}{4\sin^2(x/2)} \left|1-\frac{(1-e^{i(n+1)x})e^{-inx}}{(n+1)(1-e^{ix})}\right|^2.\]
From the above equation, by expanding the square, we have uniformly on $y \in [x-a,x+a]$ that
\[ \frac{1}{\alpha_n n} L_n(y)=\frac{1}{4\sin^2(y/2)}+ O\left(\frac{1}{n \delta^3}\right)+ O\left(\frac{1}{n^2 \delta^4}\right).\]
Hence
\[\frac{1}{n}\Esp[f_n'(x)^2]=\frac{1}{4}\times\underbrace{\frac{1}{2a}\int_{x-a}^{x+a}\frac{dy}{\sin^2(y/2)}}_{\ge 1}+ O \left(\frac{1}{n\delta^{3}}\right)+ O\left(\frac{1}{n^2 \delta^4}\right).\]
Gathering the previous computations yields, {uniformly in $x \in [-\pi,\pi]$ such that $\text{dist}(x,[-a,a])\ge\delta$}
\[\frac{1}{n^2}\frac{\Esp[f_n'(x)^2]}{\Esp[f_n(x)^2]}=\frac{1}{2}+O\left(\frac{1}{n\delta^{3}}\right)+ O\left(\frac{1}{n^2 \delta^4}\right).\]
Similarly, we have the uniform estimate
\[ 
K_n'(y)=\frac{\sin(ny)}{2\sin^2(y/2)}+ O\left(\frac{1}{n\delta^3} \right),
\]
which leads to
\[
\Esp[f_n(x)f_n'(x)]={\pi} \, K_n'\ast \mu_{\rho}(x)= \frac{1}{4}\times\frac{1}{2a}\int_{x-a}^{x+a}\frac{\sin(ny)}{\sin^2(y/2)}dy +O\left(\frac{1}{n\delta^3} \right). 
\]
Proceeding as before, we have the speed of convergence in Riemann--Lebesgue Lemma:
\[\frac{1}{2a}\int_{x-a}^{x+a}\frac{\sin(ny)}{\sin^2(y/2)}dy= O\left(\frac{1}{n\delta^{3}}\right),\]
which gives
\[\frac{\Esp[f_n(x)f_n'(x)]}{n\Esp[f_n^2(x)]} = O\left(\frac{1}{n\delta^{3}}\right).\]
Therefore, the previous estimates imply that {uniformly in $x \in [-\pi,\pi]$ such that $\text{dist}(x,[-a,a])\ge\delta$}, the integrand in Kac--Rice formula obeys the asymptotics
\begin{eqnarray*}
\frac{1}{n}\sqrt{I_n(x)}&:=&\sqrt{\frac{\Esp[f_n'(x)^2]}{n^2\Esp[f_n(x)^2]}-\left(\frac{\Esp[f_n(x)f_n'(x)]}{n\Esp[f_n(x)^2]}\right)^{2}}\\
&=&\frac{1}{\sqrt{2}}+ O\left(\frac{1}{n\delta^{3}}\right)+ O\left(\frac{1}{n^2 \delta^4}\right).
\end{eqnarray*}
Hence, if $[\alpha,\beta]$ is a subset of $[-\pi,\pi]$ such that $\text{dist}([\alpha,\beta],[-a,a])\ge \delta$, then
\[\frac{\Esp[\mathcal{N}(f_n,[\alpha,\beta])]}{n}= \frac{1}{\pi}\int_{\alpha}^{\beta} \frac{1}{n}\sqrt{I_n(x)}dx =\frac{\beta-\alpha}{\pi}\frac{1}{\sqrt{2}}+ O \left(\frac{1}{n\delta^{3}}\right)+ O\left(\frac{1}{n^2 \delta^4}\right).\]

\par
\medskip
\noindent
\underline{ii) Inside $[-a+\delta,a-\delta]/[-\delta,\delta]$:}
\par
\medskip

Let us take $\delta>0$ such that $[-a+\delta,a-\delta]/[-\delta,\delta ]$ is unempty. Now, if $[\alpha,\beta] \subset [-a+\delta,a-\delta]/[-\delta,\delta ]$, since $K_n$ and $L_n$ are regularizing trigonometric kernels, by Lemma \ref{lm.fej.leb}, for $x \in [\alpha,\beta]$:
$$
\begin{array}{l}
\Esp[f_n(x)^2]={2\pi} \, K_n\ast\psi_\rho(x)\xrightarrow[n\to\infty]~{2\pi} \, \psi_\rho(x)=\frac{\pi}{a},\\
\alpha_n \Esp[F'_n(x)^2]= {2\pi} \, L_n\ast\psi_\rho(x)\xrightarrow[n\to\infty]~{2\pi} \, \psi_\rho(x)=\frac{\pi}{a}.\\
\end{array}
$$
{
Moreover, the convergences above are uniform on $[\alpha, \beta]$. Indeed, we can write 
\[
K_n\ast\psi_\rho(x) - \frac{1}{2a} = \frac{1}{2a}  \left(\frac{1}{2\pi} \int_{[-\pi, \pi] \backslash [x-a, x+a]} K_n(y)dy\right)
\]
so that 
\begin{equation}\label{eq.uni1}
0  \leq K_n\ast\psi_\rho(x)  -\frac{1}{2a} \leq \frac{1}{2a}  \left(\frac{1}{2\pi} \int_{[-\pi, \pi] \backslash [-\delta, \delta]} K_n(y)dy\right) \xrightarrow[n\to\infty]{} 0,
\end{equation}
and in the same way 
\begin{equation}\label{eq.uni2}
0  \leq L_n\ast\psi_\rho(x)  -\frac{1}{2a} \leq \frac{1}{2a}  \left(\frac{1}{2\pi} \int_{[-\pi, \pi] \backslash [-\delta, \delta]}  L_n(y)dy\right)\xrightarrow[n\to\infty]{} 0.
\end{equation}
Besides, since $[\alpha,\beta]\subset [-a+\delta,a-\delta]/[-\delta,\delta ]$, combining Equation \eqref{eq.corel3} and Lemma \ref{lm.fej.leb} or using Lemma 2 of \cite{ADP19}, we also have 
\[
\max_{x\in[\alpha,\beta]} \left|\frac{\Esp\left[f_n(x)f_n'(x)\right]}{n}\right|\xrightarrow[n\to\infty]~0.
\]
As a result, the integrand $\frac{1}{n}\sqrt{I_n(x)}$ in Kac--Rice formula is bounded on $[\alpha, \beta]$ by dominated convergence, we deduce that
\[
\frac{\Esp[\mathcal{N}(f_n,[\alpha,\beta])]}{n}= \frac{1}{\pi}\int_{[\alpha,\beta]} \frac{1}{n}\sqrt{I_n(x)}dx =\frac{\beta-\alpha}{\pi}\frac{1}{\sqrt{3}}+ o_\delta(1).
\]
We stress that the above $o_\delta(1)$ possibly depends on the parameter $\delta$ but this will not be a problem in virtue of the next last step.
}

\par
\medskip
\noindent
\underline{iii) At the neighborhood of $\{a,-a,0\}$:}
\par\medskip

The roots of the trigonometric polynomial $f_n$ coincide with the roots in the unit circle of the algebraic polynomial
\[
Q_n(x):=x^n \sum_{k=1}^n \frac{a_k}{2} \left(x^k+\frac{1}{x^k}\right)+\frac{b_k}{2 i} \left(x^k-\frac{1}{x^k}\right).
\]
The constant coefficient of $Q_n$ is given by $\frac{a_n+i b_n}{2}$ and we have
\[
\mathbb{E}\left[\log\left(a_n^2+b_n^2\right)\right]\ge \mathbb{E}\left[\log(|a_n|)\right]
=\int_\mathbb{R}\log(|x|)\frac{e^{-\frac{x^2}{2}}}{\sqrt{2\pi}} dx
>-\infty.
\]

Besides, the coefficient of highest degree of $Q_n$ is also $\frac{a_n+i b_n}{2}$ and the previous computation applies in the same way. Finally the coefficient of degree $p\in\{1,\cdots,2n-1\}$ is given by

$$A_{n,p}=\frac{a_p-i b_p}{2}+\frac{a_{n-p}+i b_{n-p}}{2}.$$

Recalling that $a_k,b_k\sim\mathcal{N}(0,1)$ for every $k\ge 1$ it implies that $\max_{p\le 2 n} \mathbb{E}\left[\left|A_{n,p}\right|\right]<\infty.$ Then, $Q_n$ fulfills the assumptions of the Corollary 2.2 in \cite{PY15} which asserts that for every $r\in]0,1[$ and every $[\alpha,\beta]\subset0,2\pi]$, 

$$\left| \frac{\text{Card}\left(Q_n^{-1}\left(\{0\}\right)\cap\left\{z\in\mathbb{C}\,|\,r\le |z|\le\frac{1}{r}\,,\,\text{arg}(z)\in[\alpha,\beta]\right\}\right)}{n}-\frac{\beta-\alpha}{2\pi}\right|\le C_r \frac{\sqrt{\log(n)}}{n}.$$

It is clear that the number of roots in the angular sector $\left\{z\in\mathbb{C}\,|\,r\le |z|\le\frac{1}{r}\,,\,\text{arg}(z)\in[\alpha,\beta]\right\}$ is greater than the number of roots of $f_n$ in $[\alpha,\beta]$. As matter of fact, for some absolute constant $C$ and for every $\delta>0$ we then have 

\[
\frac{\Esp\left[\mathcal{N}(f_n,[\pm a-\delta, \pm a+\delta])\right]}{n}\le \frac{\delta}{\pi}+ C \sqrt{\frac{\log(n)}{n}}, \quad 
\frac{\Esp\left[\mathcal{N}(f_n,[-\delta,\delta])\right]}{n}\le \frac{\delta}{\pi}+ C \sqrt{\frac{\log(n)}{n}}.
\]

\par
\medskip
\noindent
\underline{Conclusion:}
\par\medskip

Using additivity in Kac--Rice formula, one may write
\begin{eqnarray*}
&&\frac{\Esp[\mathcal{N}(f_n,[-\pi,\pi])]}{n}=\underbrace{\frac{\Esp[\mathcal{N}(f_n,[-a+\delta,-\delta])]}{n}+\frac{\Esp[\mathcal{N}(f_n,[\delta,a-\delta])]}{n}}_{\text{Step (ii)}}\\
&&\\
&&+\underbrace{\frac{\Esp[\mathcal{N}(f_n,[-a-\delta,-a+\delta])]}{n}+\frac{\Esp[\mathcal{N}(f_n,[a-\delta,a+\delta])]}{n}+\frac{\Esp[\mathcal{N}(f_n,[-\delta,\delta])]}{n}}_{\text{Step (iii)}}\\
&&\\
&&+\underbrace{\frac{\Esp[\mathcal{N}(f_n,[-\pi,-a-\delta])]}{n}+\frac{\Esp[\mathcal{N}(f_n,[a+\delta,\pi])}{n}]}_{\text{Step (i)}}
\end{eqnarray*}
Gathering the conclusions of the three above steps, letting first $n\to\infty$ and then $\delta\to 0$, we indeed obtain the asymptotics stated in Corollary \ref{cor.creneau}, namely
\[
\frac{\Esp[\mathcal{N}(f_n,[0,2\pi])]}{n}=\frac{2\pi-2a}{\pi\sqrt{2}}+ \frac{2a}{\pi\sqrt{3}}+o(1).
\]

\subsubsection{Continuous spectral density with a simple nodal set}\label{proof-main-Thm}
In this section, we give the proof Theorem of \ref{theo.main.C1}. So 
let us consider $\psi_\rho$ a density function which is piecewise {continuous} and whose support satisfies the condition below:

$$\{\psi_\rho=0\}=\bigcup_{i=1}^p [a_i,b_i] \cup \bigcup_{j=1}^q \{c_j\}.$$

Mimicking the step (iii) of the proof of Corollary \ref{cor.creneau} in the last subsection, we can deal with some small neighborhood of the finite set $E:=\{0\}\cup \bigcup_{i=1}^p \{a_i,b_i\}\cup \bigcup_{j=1}^q\{c_j\}$. Namely, relying again on \cite[Corollary 2.2]{PY15}, for any $\delta>0$ we have

$$\frac{1}{n}\Esp\left[\mathcal{N}(f_n,E+B(0,\delta))\right]\le C_{q,p} \left(\delta +\sqrt{\frac{\log(n)}{n}}\right).$$

It is then sufficient to deal with $\mathcal{N}(f_n,[\alpha,\beta])$ where $[\alpha,\beta]$ is either included in the interior of the set $\{\psi_\rho=0\}$ or in the set $\{\psi>0\}\cap(0,2\pi)$. We details these two cases below.
\par
\medskip
\noindent
\underline{(a) $[\alpha,\beta]\subset \left[0,2\pi\right]/\left(E+B(0,\delta)\right)$ :}
\par
\medskip

{By our assumptions on $\psi_\rho$, we know that $\psi_\rho$ is both continuous and bounded from below on $[\alpha,\beta]$ so that there exists some positive constants $c,C$ such that 
\begin{equation}\label{eq.encadreindi}
c \, \mathds{1}_{[\alpha,\beta]}(x) \leq \psi_{\rho}(x) \leq C \, \mathds{1}_{[\alpha,\beta]}(x). 
\end{equation}
Then one may proceed exactly as in the step (ii) of the proof of Corollary \ref{cor.creneau}, combining the upper and lower bounds in Equation \eqref{eq.encadreindi} with the ones of Equations \eqref{eq.uni1} and \eqref{eq.uni2}, to deduce that uniformly in $x \in [\alpha, \beta]$, one has
\[
\frac{L_n\ast\psi_\rho(x)}{K_n\ast\psi_\rho(x)} \to 1, \quad \frac{K_n'\ast\psi_\rho(x)}{n K_n\ast\psi_\rho(x)} \to 0.
\]
The fact that these convergences are uniform implies that the integrand $\frac{1}{n}\sqrt{I_n(x)}$ in Kac--Rice formula is again bounded on $[\alpha, \beta]$ and by dominated convergence, as above we deduce that
\[
\frac{\Esp[\mathcal{N}(f_n,[\alpha,\beta])]}{n}= \frac{1}{\pi}\int_{[\alpha,\beta]} \frac{1}{n}\sqrt{I_n(x)}dx =\frac{\beta-\alpha}{\pi}\frac{1}{\sqrt{3}}+ o_\delta(1).
\]
}

\par
\medskip
\noindent
\underline{(b) $[\alpha,\beta]\subset \bigcup_{i=1}^p ]a_i+\delta,b_i-\delta[$ :}
\par\medskip

Let us adapt to this setting the computations of the step (i) of the proof of Corollary \ref{cor.creneau}. For convenience, we set $\text{Supp}(\psi_\rho):=\{\psi_\rho>0\}$ and $x-\text{Supp}(\psi_\rho):=\{x-u,\,u \in \text{Supp}(\psi_\rho)\}$. Then we have
\[
\begin{array}{ll}
\displaystyle{n\Esp\left[f_n^2(x)\right]} & \displaystyle{=\int_0^{2\pi} K_n(y) \psi_\rho(x-y) dy
= \int_{x-\text{Supp}(\psi_\rho)}K_n(y)\psi_\rho(x-y)dy}\\
\\
& \displaystyle{=\int_{x-\text{Supp}(\psi_\rho)}\frac{1}{2\sin^2\left(\frac{y}{2}\right)}\psi_\rho(x-y)dy-\int_{x-\text{Supp}(\psi_\rho)}\frac{\cos(ny)}{2\sin^2\left(\frac{y}{2}\right)}\psi_\rho(x-y)dy}.
\end{array}
\]
If $x\in [\alpha,\beta]$ and since $[\alpha,\beta]\subset \bigcup_{i=1}^p ]a_i+\delta,b_i-\delta[$ then $\text{dist}\left(x,\text{Supp}(\psi_\rho)\right)\ge \delta.$ Then, one can use Riemann--Lebesgue lemma again, to do so we first write

\begin{eqnarray*}
&&\int_{x-\text{Supp}(\psi_\rho)}\frac{\cos(ny)}{2\sin^2\left(\frac{y}{2}\right)}\psi_\rho(x-y)dy=\int_{\text{Supp}(\psi_\rho)}\frac{\cos(n(x-u))}{2\sin^2\left(\frac{x-u}{2}\right)}\psi_\rho(u)dy\\
&=& \cos(nx) \int_{\text{Supp}(\psi_\rho)}\frac{\cos(n u)}{2\sin^2\left(\frac{x-u}{2}\right)}\psi_\rho(u)dy+\sin(nx) \int_{\text{Supp}(\psi_\rho)}\frac{\sin(n u)}{2\sin^2\left(\frac{x-u}{2}\right)}\psi_\rho(u)dy.
\end{eqnarray*}

{Now, we recall that $\text{Supp}(\psi_\rho)$ can be written as a finite union of open intervals $\bigcup_{i=1}^r ]d_i,e_i[$. Let $\epsilon>0$, on each interval $]d_i,e_i[$ one may find $\psi_i\in\mathcal{C}^\infty_c(]d_i,e_i[)$ such that $\int_{d_i}^{e_i}|\psi_\rho(x)-\psi_i(x)|dx<\epsilon$. We then build a global approximation $\psi_\epsilon$ such that $\text{Supp}(\psi_\epsilon=0)=\text{Supp}(\psi_\rho=0)$ and for each $i\in\{1,\cdots,r\}$ we have $\psi_\epsilon=\psi_i$ on $]d_i,e_i[$. Then we may write
\begin{eqnarray*}
&&\left|\int_{\text{Supp}(\psi_\rho)}\frac{\cos(n u)}{2\sin^2\left(\frac{x-u}{2}\right)}\psi_\rho(u)dy-\int_{\text{Supp}(\psi_\rho)}\frac{\cos(n u)}{2\sin^2\left(\frac{x-u}{2}\right)}\psi_\epsilon(u)dy\right|\\
&\le&\sum_{i=1}^r\int_{d_i}^{e_i}\frac{\left|\psi_i(x)-\psi_\rho(x)\right|}{2\sin^2\left(\frac{x-u}{2}\right)}dx\le C_r\frac{\epsilon}{\delta^2}.
\end{eqnarray*}
By construction, we know that $\psi_\epsilon$ is $\mathcal{C}^1$ on $]d_i,e_i[$. One can then make an integration by parts formula as in step (i) of the proof of Corollary \ref{cor.creneau} which gives
\[
\int_{x-\text{Supp}(\psi_\epsilon)}\frac{\cos(ny)}{4\sin^2\left(\frac{y}{2}\right)}\psi_\epsilon(x-y)dy=O_\epsilon\left(\frac{1}{n\delta^3}\right).
\]
We insist on the fact that the above remainder depends on $\epsilon$ but this will not be a problem since we shall let $n\to\infty$ before letting $\epsilon,\delta\to0$.}
Besides, using that  $\sin^2\le 1$ and that $\int_{\text{supp}(\psi_\rho)}\psi_\rho(u)du=1$ we may deduce that $$\int_{x-\text{Supp}(\psi_\rho)}\frac{1}{2\sin^2\left(\frac{y}{2}\right)}\psi_\rho(x-y)dy\ge \frac{1}{2}.$$ {Thus we can write the following expansion:
\[
n\Esp\left[f_n^2(x)\right]=\int_0^{2\pi} K_n(y) \psi_\rho(x-y) dy=\underbrace{\int_{x-\text{Supp}(\psi_\rho)}\frac{1}{2\sin^2\left(\frac{y}{2}\right)}\psi_\rho(x-y)dy}_{\ge\frac 1 2}+O_\epsilon\left(\frac{1}{n\delta^3}\right)+O\left(\frac{\epsilon}{\delta^2}\right).
\]
}
Following exactly the same lines as in step (i) above, in fact this case is actually simpler and does not require Riemann--Lebesgue lemma {nor approximation of $\psi_\rho$ by smoother functions}, we get that
\[
\frac{1}{n}\Esp[f_n'(x)^2]=\frac{1}{4}\times\int_{x-\text{Supp}(\psi_\rho)}\frac{1}{\sin^2\left(\frac{y}{2}\right)}\psi_\rho(x-y)dy  + O \left(\frac{1}{n\delta^{3}}\right).
\]
{as well as
\[
\Esp[f_n(x)f_n'(x)]=\frac{1}{2}K_n'\ast \mu_{\rho}(x)= \frac{1}{4}\times\int_{x-\text{Supp}(\psi_\rho)}\frac{\sin(ny)}{\sin^2(y/2)}\psi_\rho(x-y)dy +O\left(\frac{1}{n\delta^3} \right).
\]
As before, up to approximating $\psi_\rho$ by a smoother function $\psi_\epsilon$ in $L^1\left(\text{Supp}(\psi_\rho)\right)$ and performing an integration by parts to quantify Riemann-Lebesgue convergence we can write the 
\[
\int_{x-\text{Supp}(\psi_\rho)}\frac{\sin(ny)}{\sin^2(y/2)}\psi_\rho(x-y)dy=O_\epsilon\left(\frac{1}{n\delta^3 }\right)+O\left(\frac{\epsilon}{\delta^2}\right)
\]
Finally, plugging these estimates into Kac--Rice formula leads to the desired estimate:
\[
\frac{\Esp[\mathcal{N}(f_n,[\alpha,\beta])]}{n}= \frac{1}{\pi}\int_{\alpha}^{\beta} \frac{1}{n}\sqrt{I_n(x)}dx =\frac{\beta-\alpha}{\pi}\frac{1}{\sqrt{2}}+ O_\epsilon \left(\frac{1}{n\delta^{3}}\right)+O\left(\frac{\epsilon}{\delta^2}\right).
\]
As before, we conclude the proof by additivity in the Kac--Rice formula and we first let $n\to\infty$, then we let $\epsilon\to 0$ and at the very end $\delta\to 0$.
}

\subsubsection{Spectral density with H\"older derivative}
Finally, we now give the proof of Theorem \ref{theo.main}, i.e. we get rid of the assumption on the nodal set $\{\psi_{\rho}=0\}$ in Theorem \ref{theo.main.C1} by requiring a slightly stronger regularity. 
{For simplicity assume here that $\psi_\rho$ is globally $\mathcal{C}^{1,\alpha}\left(\mathbb{T}\right)$} which means that $\psi_\rho$ is derivable on $\mathbb{R}$, $2\pi$-periodic and that
\[
\forall (x,y)\in[0,2\pi]^2, \left|\psi_\rho'(x)-\psi_\rho'(y)\right|\le [\psi_\rho']_{\alpha}~|x-y|^{\alpha}\,\,\text{where}\,\,[\psi_\rho']_{\alpha}>0.
\]
{We refer to Remark \ref{cas-ouvert-rem} at the end of the section for the case where $\psi_\rho$ is only $\mathcal{C}^1$ with H\"{o}lder derivative on an open set of full measure.} Recall that, based on Lemma \ref{lem.speedfejer}, for some constant $C_{\rho}$ that only depends on the spectral density $\psi_\rho$, we then have
\[
\begin{array}{l}
\displaystyle{\forall n\ge 1, \forall x \in \mathbb{R},\,\left|K_n\ast\psi_\rho(x)-\psi_\rho(x)\right|\le \frac{C_{\rho}}{n}},\\
\\
\displaystyle{\forall n\ge 1, \forall x \in \mathbb{R},\,\left|L_n\ast\psi_\rho(x)-\psi_\rho(x)\right|\le \frac{C_{\rho}}{n}}.
\end{array}
\]
The proof of Theorem \ref{theo.main} is divided in three distinct steps that we sketch below before giving details in the sequel.

\begin{enumerate}
\item Exploiting the rate of convergence recalled just above, we may prove that the integrand $\frac{\sqrt{I_n(x)}}{n}$ in Kac--Rice formula is bounded from above uniformly for $n\ge 1$ and $x\in[0,2\pi]$. As such, one is left to establish the pointwise convergence and use the dominated convergence.
\item As before, using the Fej\'er--Lebesgue Theorem at every point $x$ such that $\psi_{\rho}(x)>0$ we shall obtain that
$$\frac{L_n\ast\psi_{\rho}(x)}{K_n\ast\psi_{\rho}(x)}\sim\frac{\psi_\rho(x)}{\psi_\rho(x)}\sim 1,\,\,\frac{K_n'\ast \psi_\rho(x)}{n K_n\ast \psi_\rho(x)}\to 0$$
and so
$$\frac{\sqrt{I_n(x)}}{n}\xrightarrow[n\to\infty]~\frac{1}{\sqrt{3}}.$$
This step only requires the continuity of $\psi_\rho$ and the regularizing properties of $K_n$ and $L_n$ which were already used in the previous sections.
\item It remains to consider the case where $\psi_\rho(x)=0$ which is not covered by the previous situation as the limit in the Kac--Rice formula integrand involves the indeterminate form $\frac{0}{0}$. To bypass this problem, we must exploit the derivability of $\psi_\rho$ and establish the convergence in $L^2([0,2\pi])$ of $n \left(K_n \ast \psi_\rho-\psi_\rho\right)$ and $\frac{2}{n \alpha_n} \left(L_n \ast \psi_\rho-\psi_\rho\right)$ towards the same non-degenerate limit denoted by $\mathcal{L}[\psi_\rho](x)$. On the other hand, up to extracting a subsequence the convergence is also almost everywhere and as a consequence we obtain for $\psi_\rho(x)=0$
\begin{eqnarray*}
&&\frac{1}{n^2 \alpha_n}\frac{L_n\ast \psi_\rho(x)}{K_n\ast \psi_\rho(x)}=\frac{1}{n^2 \alpha_n}\frac{L_n\ast \psi_\rho(x)-\psi_\rho(x)}{K_n\ast \psi_\rho(x)-\psi_\rho(x)}\\
&&=\frac{1}{2}\frac{ \frac{2}{n \alpha_n}\left(L_n\ast \psi_\rho(x)-\psi_\rho(x)\right)}{n \left(K_n\ast \psi_\rho(x)-\psi_\rho(x)\right)}\xrightarrow[n\to\infty]~\frac{1}{2}\frac{\overbrace{\mathcal{L}[\psi_\rho](x)}^{>0}}{\mathcal{L}[\psi_\rho](x)}=\frac{1}{2}.
\end{eqnarray*}
Finally one is left to use the dominated convergence Theorem and let $n\to\infty$ in Kac--Rice formula.
\end{enumerate}

\par
\medskip
\noindent
\underline{Step 1: bounded integrand in Kac--Rice formula}~
\par
\medskip

First of all, we have the trivial upper bound

\begin{eqnarray*}
\frac{\sqrt{I_n(x)}}{n}&=&\sqrt{\frac{1}{n^2 \alpha_n }\frac{L_n \ast \psi_{\rho}(x) }{ K_n \ast \psi_{\rho}(x)}-\left(\frac{K_n' \ast \psi_{\rho}(x)}{2n K_n \ast \psi_{\rho}(x)}\right)^2}\\
&\le& \underbrace{\sqrt{\frac{(n+1)(2n+1)}{6 n^2}}}_{\to\frac{1}{\sqrt{3}}} \sqrt{\frac{L_n \ast \psi_{\rho}(x) }{ K_n \ast \psi_{\rho}(x)}}.
\end{eqnarray*}
Hence, one is left to bound the ratio $\frac{L_n \ast \psi_{\rho}(x) }{ K_n \ast \psi_{\rho}(x)}$. Let us assume that $\psi_\rho(x)>2 \frac{C_{\rho}}{n}$, relying on Lemma \ref{lem.speedfejer}, we may write
\begin{equation*}
\frac{L_n \ast \psi_{\rho}(x) }{ K_n \ast \psi_{\rho}(x)}\le\frac{\psi_\rho(x)+\frac{C_{\rho}}{n}}{\psi_\rho(x)-\frac{C_{\rho}}{n}}
= \frac{1+\frac{C_{\rho}}{n\psi_\rho(x)}}{1-\frac{C_{\rho}}{n\psi_\rho(x)}}
\le \frac{\frac 3 2}{\frac 1 2}=3.
\end{equation*}
Let us assume now that $\psi_\rho(x)\le 2 \frac{C_{\rho}}{n}$. Relying on the Lemma \ref{lem.nondeg}, there exists another constant $c_{\rho}>0$ that only depends on $\psi_\rho$ such that 
\[
K_n\ast \psi_\rho(x)\ge \frac{c_{\rho}}{n}.
\]
Thus we get
\begin{equation*}
\frac{L_n \ast \psi_{\rho}(x) }{ K_n \ast \psi_{\rho}(x)}\le \frac{n}{c_{\rho}} \left( \psi_\rho(x)+\frac{C_{\rho}}{n}\right)\le \frac{n}{c_{\rho}}\frac{3 C_{\rho}}{n}=3\frac{C_{\rho}}{c_{\rho}}.
\end{equation*}
Gathering the previous estimates entails that
$$\forall n\ge 1, \, \forall x\in\mathbb{R},\,\, \frac{L_n \ast \psi_{\rho}(x) }{ K_n \ast \psi_{\rho}(x)}\le 3+3 \frac{C_{\rho}}{c_{\rho}},$$
which concludes this first step.
\par
\medskip
\noindent
\underline{Step 2: the case where $\psi_{\rho}(x)>0$}~
\par
\medskip
For the second step, as in the previous sections, we rely on the fact that both $L_n$ and $K_n$ are regularizing kernels such that by Lemma \ref{lm.fej.leb}
\[
L_n \ast \psi_\rho(x)\to\psi_\rho(x)>0, \quad K_n \ast \psi_\rho(x)\to\psi_\rho(x)>0, \quad \frac{1}{n} K_n' \ast\psi_\rho(x)\to 0.
\]
Plugging theses estimates in Kac--Rice formula entails that

\begin{eqnarray*}
\frac{\sqrt{I_n(x)}}{n}=\sqrt{\underbrace{\frac{1}{n^2 \alpha_n }}_{\to \frac 1 3}\underbrace{\frac{L_n \ast \psi_{\rho}(x) }{ K_n \ast \psi_{\rho}(x)}}_{\to 1}-\underbrace{\left(\frac{K_n' \ast \psi_{\rho}(x)}{2n K_n \ast \psi_{\rho}(x)}\right)^2}_{\to 0}}\to \frac{1}{\sqrt{3}}.
\end{eqnarray*}

\par
\medskip
\noindent
\underline{Step 3: the case where $\psi_{\rho}(x)=0$}~
\par
\medskip

This last step  where $\psi_\rho(x)=0$ is arguably more delicate as it requires some regularity of the spectral density. While one could study the limit of the Kac--Rice integrand in the pointwise sense it is simpler to consider convergence in $L^2=L^2([0,2\pi])$ which will be sufficient for our purpose.
First, using the Fourier representation of Fej\'er kernel we can write

\begin{eqnarray*}
n\left(K_n\ast \psi_\rho(x)-\psi_\rho(x)\right)&\stackrel{L^2}{=}&~n\sum_{|k|> n}\underbrace{\rho(k)}_{=\hat{\psi}_\rho(k)}e^{i k x}-\sum_{|k|\le n}|k|\rho(k)e^{i k x}.
\end{eqnarray*}
Since $\psi_\rho$ is $\mathcal{C}^1$ then $(k \rho(k))_{k\in\mathbb{Z}}\in \textit{l}^{\hspace{0.05cm}2}(\mathbb{Z})$ and thus
\[
\left|\left| n\sum_{|k|> n}\rho(k)e^{i k x}\right|\right|^2_{L^2}=n^2\sum_{|k|>n} \rho(k)^2\le\sum_{|k|>n} k^2\rho(k)^2\xrightarrow[n\to\infty]~0.
\]
In the same way,
$$\sum_{|k|\le n}|k|\rho(k)e^{i k x}\xrightarrow[n\to\infty]{L^2}~\sum_{k\in\mathbb{Z}} |k|\rho(k)e^{i k x}.$$
Thus,

\begin{equation}\label{final-Kn}
n\left(K_n\ast \psi_\rho(x)-\psi_\rho(x)\right)\xrightarrow[n\to\infty]{L^2}~-\sum_{k\in\mathbb{Z}} |k|\rho(k)e^{i k x}:=\mathcal{L}\left[\psi_\rho\right](x).
\end{equation}

We imitate this strategy for $\frac{1}{n \alpha_n} \left(L_n \ast \psi_\rho-\psi_\rho\right)$ though the Fourier representation is slightly more involved. The Fourier decomposition of $L_n$ reads as (see e.g. \cite[p 209]{ADP19})
\[
L_n(x):= \alpha_n \sum_{r=-n}^{n}    \left(  \frac{1}{{n}} \sum_{k=1}^{n-|r|} k(|r|+k) \right) e^{i r x}.
 \]
so that
\begin{eqnarray*}
\frac{1}{n \alpha_n}\left(L_n \ast \psi_\rho-\psi_\rho\right)&=&\frac{1}{n} \sum_{|k|\le n-1}\rho(k)\left(\frac{1}{n}\sum_{r=1}^{n-|k|}r\left(|k|+r\right)\right)e^{i k x}-\frac{(n+1)(2n+1)}{6 n}\psi_\rho(x)\\
&=&\sum_{k\in\mathbb{Z}} C_{k,n} \rho(k) e^{i k x},
\end{eqnarray*}
where 
\[
C_{k,n}:=\mathds{1}_{\{|k|\le n-1\}}\frac{1}{n^2}\sum_{r=1}^{n-|k|}r\left(|k|+r\right)-\frac{(n+1)(2n+1)}{6 n}.
\]
Hence, one is left to study the convergence of $(C_{k,n}~\rho(k))_{k\in\mathbb{Z}}$ in $\textit{l}^{\hspace{0.05cm}2}(\mathbb{Z})$ as $n\to\infty$.
We may split $C_{k,n}$ in two terms by writing that
\[
C_{k,n}=\underbrace{\mathds{1}_{\{|k|\le n-1\}}\frac{1}{n^2}\sum_{r=1}^{n-|k|}r^2-\frac{(n+1)(2n+1)}{6 n}}_{:=A_{k,n}}+\underbrace{\mathds{1}_{\{|k|\le n-1\}}\frac{|k|}{n^2}\sum_{r=1}^{n-|k|}r}_{:=B_{k,n}}.
\]
Let us first deal with the term $A_{k,n}$. One has

\begin{eqnarray*}
&&\mathds{1}_{\{|k|\le n-1\}}\frac{1}{n^2}\sum_{r=1}^{n-|k|}r^2-\frac{(n+1)(2n+1)}{6 n}=\frac{1}{n^2}\left(\mathds{1}_{\{|k|\le n-1\}}\sum_{r=1}^{n-|k|}r^2-\sum_{r=1}^n r^2\right)\\
&&=-\frac{\sum_{n-|k|+1}^{n} r^2}{n^2}\mathds{1}_{\{|k|\le n-1\}}-\frac{(n+1)(2n+1)}{6 n} \mathds{1}_{\{|k|\ge n\}}
\end{eqnarray*}
Besides, as before, since $(k\rho(k))_{k\in\mathbb{Z}}\in \textit{l}^{\hspace{0.05cm}2}(\mathbb{Z})$ we get
\begin{eqnarray*}
\sum_{k\in\mathbb{Z}}\left(\frac{(n+1)(2n+1)}{6 n} \mathds{1}_{\{|k|\ge n\}}\right)^2\rho(k)^2~\le C n^2\sum_{|k|\ge n}\rho(k)^2\xrightarrow[n\to\infty]~0.
\end{eqnarray*}
On the other hand, we have

\begin{eqnarray*}
&&\left|\frac{1}{n^2}\sum_{n-|k|+1}^{n} r^2\mathds{1}_{\{|k|\le n-1\}}-|k|\mathds{1}_{\{|k|\le n-1\}}\right|=\mathds{1}_{\{|k|\le n-1\}}\frac{1}{n^2}\sum_{n-|k|+1}^n \left(n^2-r^2\right)\\
&&=\mathds{1}_{\{|k|\le n-1\}} \frac{1}{n^2} \sum_{n-|k|+1}^n(n-r)(n+r) \le \mathds{1}_{\{|k|\le n-1\}}  \frac{2 k}{n}.
\end{eqnarray*}
As a result we derive that
\begin{eqnarray*}
\sum_{k\in\mathbb{Z}}
&&\left|\frac{1}{n^2}\sum_{n-|k|+1}^{n} r^2\mathds{1}_{\{|k|\le n-1\}}-|k|\mathds{1}_{\{|k|\le n-1\}}\right|^2\rho(k)^2\\
&&\le \sum_{|k|\le n-1} \frac{4 k^2}{n^2}\rho(k)^2\le \frac{4}{n^2}\sum_{k\in\mathbb{Z}}k^2 \rho(k)^2\xrightarrow[n\to\infty]~0.
\end{eqnarray*}
Noticing that $(|k| \mathds{1}_{\{|k|\le n-1\}}\rho(k))_{k\in\mathbb{Z}}\xrightarrow[n\to\infty]{\textit{l}^{\hspace{0.05cm}2}(\mathbb{Z})}~(|k|\rho(k))_{k\in\mathbb{Z}}$ and gathering the previous facts implies 

\begin{equation}\label{intermediate1}
(A_{k,n}~\rho(k))_{k\in\mathbb{Z}}\xrightarrow[n\to\infty]{\textit{l}^{\hspace{0.05cm}2}(\mathbb{Z})}~(-|k|\rho(k))_{k\in\mathbb{Z}}.
\end{equation}
One is then left to deal with the only remaining sequence
\[
(B_{k,n}~\rho(k))_{k\in\mathbb{Z}}=\left(\rho(k)\mathds{1}_{|k|\le n-1}\frac{|k|}{n^2}\sum_{r=1}^{n-|k|}r\right)_{k\in\mathbb{Z}}.
\]
We simply notice that this sequence is dominated by $\left(|k| \rho(k)\right)_{k\in\mathbb{Z}}$ which by our assumptions belongs to $\textit{l}^{\hspace{0.05cm}2}(\mathbb{Z})$ and that for each fixed $k$:

\begin{eqnarray*}
&&\mathds{1}_{|k|\le n-1}\frac{|k|}{n^2}\sum_{r=1}^{n-|k|}r=|k|\mathds{1}_{|k|\le n-1}\left(\frac{1}{n}\sum_{r=1}^{n-|k|} \frac{r}{n}\right)\to~|k|\int_0^1 x dx=\frac{|k|}{2}.
\end{eqnarray*}
Using dominated convergence in $\textit{l}^{\hspace{0.05cm}2}(\mathbb{Z})$ entails that 
\begin{equation}\label{intermediate2}
(B_{k,n}~\rho(k))_{k\in\mathbb{Z}}\xrightarrow[n\to\infty]{\textit{l}^{\hspace{0.05cm}2}(\mathbb{Z})}~\left(\frac{|k|}{2}\rho(k)\right)_{k\in\mathbb{Z}}.
\end{equation}
Therefore, gathering the limits \eqref{intermediate1} and \eqref{intermediate2}, we get

\begin{equation}\label{intermediate3}
(C_{k,n}~\rho(k))_{k\in\mathbb{Z}}\xrightarrow[n\to\infty]{\textit{l}^{\hspace{0.05cm}2}(\mathbb{Z})}~\left(-\frac{|k|}{2}\rho(k)\right)_{k\in\mathbb{Z}}.
\end{equation}
As a consequence we conclude that
\begin{equation}\label{final-L_n}
\frac{1}{n\alpha_n} \left(L_n\ast \psi_\rho-\psi_\rho\right)\xrightarrow[n\to\infty]{L^2}~-\frac{1}{2}\sum_{k\in\mathbb{Z}}|k|\rho(k)e^{i k x}=\frac{1}{2} \mathcal{L}[\psi_\rho](x).
\end{equation}
\par
\medskip
\noindent
\underline{Conclusion of the proof:}
\par\medskip

We argue by contradiction and we assume that as $n$ goes to infinity
\[
\frac{\Esp\left[\mathcal{N}(f_n,[0,2\pi])\right]}{n}\not \rightarrow \frac{\lambda(\{\psi_{\rho}=0\})}{\pi \sqrt{2}} + \frac{2\pi - \lambda(\{\psi_{\rho}=0\})}{\pi\sqrt{3}}.
\]
Then, for some subsequence $(n_p)_{p\ge 1}$ and $\eta>0$ we would have
\[
\left|\frac{\Esp\left[\mathcal{N}(f_{n_p},[0,2\pi])\right]}{n_p}- \frac{\lambda(\{\psi_{\rho}=0\})}{\pi \sqrt{2}} - \frac{2\pi - \lambda(\{\psi_{\rho}=0\})}{\pi\sqrt{3}}\right|\ge \eta.
\]
Up to extracting another subsequence, thanks to Equations \eqref{final-Kn} and \eqref{final-L_n}, we may assume that
\[
n_p\left(K_{n_p}\ast\psi_\rho-\psi_\rho\right) \to \mathcal{L}[\psi_\rho](x), \qquad \frac{1}{n_p\alpha_{n_p}} \left(L_{n_p}\ast\psi_\rho-\psi_\rho\right)\to \frac{1}{2}\mathcal{L}[\psi_\rho](x).
\]
By Kac--Rice formula, we would get
\[
\frac{\Esp\left[\mathcal{N}(f_{n_p},[0,2\pi])\right]}{n_p}=\int_{\{\psi_\rho\neq 0\}} \frac{\sqrt{I_{n_p}(x)}}{n_p}dx+\int_{\{\psi_\rho= 0\}} \frac{\sqrt{I_{n_p}(x)}}{n_p}dx.
\]
The first integral would converge towards $\frac{1}{\sqrt{3}}\lambda\left(\psi_\rho\neq 0\right)$ by combining the conclusions of Step 1 and Step 2 above, together with dominated convergence. To deal with the second integral we must observe the following facts:

\begin{itemize}
\item If $\psi_\rho(x)=0$ then:
\begin{eqnarray*}
&&n_p \left(K_{n_p}\ast\psi_{\rho}(x)-\psi_\rho(x)\right)=\frac{1}{2\pi}\int_0^{2\pi}\frac{\sin^2(\frac{n_p y}{2})}{\sin^2(\frac{y}{2})}\psi_\rho(x-y)dy\\
&&\ge \frac{1}{2\pi}\int_0^{2\pi}\sin^2(\frac{n_p y}{2})\psi_\rho(x-y)dy\\
&&\xrightarrow[n\to\infty]{\text{Riemann-Lebesgue}}~\frac{1}{4\pi}\int_0^{2\pi}\psi_\rho(x-y)dy=\frac{1}{2}.
\end{eqnarray*}

This implies in particular that $\mathcal{L}[\psi_\rho](x)\ge \frac{1}{2}$ provided that $\psi_\rho(x)=0$.
\item If $\psi_\rho(x)=0$ then $\psi_\rho'(x)=0$ since $\psi_\rho$ is non negative. Thus, $$K_{n_p}'\ast \psi_\rho(x)=K_{n_p}\ast\psi_\rho'(x)\to \psi_\rho'(x)=0.$$ Based on the previous observation we then have that $\frac{K_{n_p}'\ast\psi_\rho(x)}{n_p K_{n_p}\ast\psi_\rho(x)} \to 0$.
\end{itemize}
Then, relying on the third step of the proof, we would have
\[
\psi_\rho(x)=0 \Rightarrow \frac{\sqrt{I_{n_p}(x)}}{n_p}\to \frac{1}{\sqrt{2}}.
\]
By the first step of the proof and dominated convergence again, the second integral would converge to $\frac{1}{\sqrt{2}}\lambda\left(\psi_\rho=0\right)$ which brings a contradiction and thus achieves the proof of Theorem \ref{theo.main}.

{
\begin{rmk}\label{cas-ouvert-rem}Note that we gave the proof of Theorem \ref{theo.main} is the case where $\psi_{\rho}$ is globally $\mathcal C^{1, \alpha}$. If $\psi_{\rho}$ is only $\mathcal C^{1, \alpha}$ on a open set of full measure, then the complementary set is a compact set of zero measure and can thus be covered by a finite number of intervals of arbitrary small length. On each of these intervals, using Corollary 2.2 of \cite{PY15} as we have done in Step (iii) of the proof of Corollary \ref{cor.creneau}, we obtain that, as $n$ goes to infinity, the normalized expected number of zeros of $f_n$ in the union of these intervals is arbitrary small, hence the conclusion of Theorem \ref{theo.main}.
\end{rmk}
}

\begin{rmk}As a by product, the previous proof suggests that the Kac-Rice integrand $\sqrt{I_n(x)}/n$ converges pointwise towards $\frac{1}{\sqrt{2}}\textbf{1}_{\psi(x)=0}+\frac{1}{\sqrt{3}}\textbf{1}_{\psi(x)\neq 0}$. Note that, in the case $\psi(x)=0$, for simplicity the previous proof uses convergences in $L^2$ but one could actually work in the pointwise sense. In particular, the limit is not continuous hence the convergence cannot be uniform. Besides, the pointwise limit is not constant which highlights the fact that the roots are not equidistributed anymore as soon as the spectral density $\psi_\rho$ vanishes. Moreover, Corollary 2.2 of \cite{PY15} guarantees the angular equidistribution of the roots on a domain of type $\{|r|\le |z|\le \frac{1}{r}\}$ , thus it seems that letting $r\to 1$ breaks the equidistribution in this setting. Finally, assuming that $\psi_\rho$ is globally $\mathcal{C}^{1,\alpha}$ provides a natural instance where $\sqrt{I_n}/n$ is uniformly bounded and thus improves in this particular setting the upper bound provided by the same Corollary 2.2 in \cite{PY15}.
\end{rmk}

\section{Salem--Zygmund type Central Limit Theorems}\label{sec.SZCLT}
The goal of this section is to give the proofs of Theorems \ref{prop.conv.fc} and \ref{thm.SZ.func}, that is the unidimensional and functional Central Limit Theorems \`a la Salem--Zygmund stated in the introduction. In order to do so, let us first give a simple estimate on the two points correlation function of the model.
 
\subsection{On the two points correlation function}
Recall that if $f_n$ is defined by Equation \eqref{def.model}, we have $\mathbb E[f_n^2(x)]= {2\pi} \, K_n \ast \mu_{\rho}(x)$ where $\mu_{\rho}$ is the associated spectral measure. Similarly the two points correlation function can be expressed as a kind of convolution with the following polarized Fej\'er kernel defined on $[0,2\pi]^2$ by
\begin{equation}\label{eq.bivariate.fejer}
K_n(x,y):= \frac{1}{n} \frac{\sin \left(\frac{n x}{2}\right)\sin \left(\frac{n y}{2}\right)}{\sin\left(\frac{x}{2}\right)\sin\left(\frac{y}{2}\right)}.
\end{equation}
\begin{lma} \label{lem.twopoints}
For any $x,y \in [0, 2\pi]$, we have 
\begin{equation}\label{eq.twopoints}
\mathbb E[f_n(x)f_n(y)] =\cos\left({\frac{(n+1)}{2}(x-y)}\right) \int_0^{2\pi} K_n(x-u,y-u)\mu_{\rho}(du).
\end{equation}
\end{lma}
\begin{proof}
If $f_n$ is defined by Equation \eqref{def.model}, using the fact that $\Esp[a_k a_{\ell}]=\Esp[b_k b_{\ell}]=\rho(k-\ell)$, standard computations give
\[
\mathbb E[f_n(x)f_n(y)] =\frac{1}{n} \sum_{k,l=1}^n \rho(k-l) \cos(kx-l y) = \Re \left(  \frac{1}{n} \sum_{k,l=1}^n \rho(k-l) e^{i(kx-l y)} \right). 
\]
Writing that for all integer $k$, 
\[
\rho(k)=\widehat{\mu_{\rho}}(k) =\int_0^{2\pi} e^{-i ku} \mu_{\rho}(du),
\]
it results that
\[
\begin{array}{ll}
\displaystyle{\frac{1}{n} \sum_{k,l=1}^n \rho(k-l) e^{i(kx-l y)}} 
 =\displaystyle{ \int_0^{2\pi}   \left( \frac{1}{n} \sum_{k,l=1}^n  e^{i(k(x-u)-l (y-u))} \right) \mu_{\rho}(du)}\\
\\
 =\displaystyle{ e^{i(x-y)}  \int_0^{2\pi}   \left( \frac{1}{n} \sum_{k,l=0}^{n-1}  e^{i k(x-u)} e^{-i \,l (y-u)} \right) \mu_{\rho}(du)}\\
\\
 =\displaystyle{e^{i(x-y)}  \int_0^{2\pi}   \left( \frac{1}{n} \frac{1-e^{i n (x-u)}}{1-e^{i(x-u)}} \frac{1-e^{-i \,n (y-u)}}{1-e^{-i (y-u)}} \right) \mu_{\rho}(du)}
\\
\\
 =\displaystyle{e^{i\frac{(n+1)}{2}(x-y)} \int_0^{2\pi} K_n(x-u,y-u) \mu_{\rho}(du)}.
\end{array}
\]
Hence, we obtain 
\[
\mathbb E[f_n(x)f_n(y)] =\cos\left({\frac{(n+1)}{2}(x-y)}\right) \int_0^{2\pi} K_n(x-u,y-u)\mu_{\rho}(du).
\]
Remark that when $x=y$, we indeed fall back on $\mathbb E[f_n(x)^2]={2\pi} \, K_n \ast \mu_{\rho}(x)$.
\end{proof}

\begin{lma} \label{lm.polfej}
For all $x,y \in [0,2\pi]$, $n \geq 1$ and for $\varepsilon>0$ arbitrarily chosen, if $d(x,y)$ denotes the distance between $x$ and $y$ modulo $2\pi$, there exist a constant $C>0$ such that 
\begin{eqnarray*}
&&\int_0^{2\pi} K_n(x-u,y-u)\mu_{\rho}(du)\\
&\leq& C\left(\frac{\sqrt{K_n \ast \mu_{\rho}(x)}+\sqrt{K_n \ast \mu_{\rho}(y)}}{\sqrt{n}\varepsilon} +\frac{1}{n\varepsilon^2}+ \sqrt{K_n\ast \mu_{\rho}(x)}\sqrt{K_n \ast \mu_{\rho}(y)}\mathds{1}_{d(x,y)\leq 2\varepsilon} \right).
\end{eqnarray*}
\end{lma}

\begin{proof}
Notice that  
\[
\int_0^{2\pi}K_n(x-u,y-u)\mu_{\rho}(du)= \Esp_{U}\left[\frac{1}{n} \frac{\sin\left(\frac{n(x-U)}{2}\right) \sin \left(\frac{n(y-U)}{2}\right)}{\sin \left(\frac{x-U}{2}\right)\sin \left(\frac{y-U}{2}\right)} \right],
\]
where $U$ is a random variable with values in $[0,2\pi]$ whose law is $\mu_{\rho}$.
Let us fix $\varepsilon>0$ and recall that $d(x,y)$ is the distance between $x$ and $y$ modulo $2\pi$. If $U$ is far enough from both points $x$ and $y$, we have for a positive constant $C$ which may change from line to line
\[
\left| \, \mathbb E_{U} \left[ \frac{1}{n} \frac{\sin \left( \frac{n (x-U)}{2}\right)}{\sin \left( \frac{(x-U)}{2}\right)} \mathds{1}_{d(x,U)>\varepsilon} \frac{\sin \left( \frac{n (y-U)}{2}\right)}{\sin \left( \frac{(y-U)}{2}\right)} \mathds{1}_{d(y,U)>\varepsilon}\right]\, \right|\leq \frac{C}{n \varepsilon^2}.
\]
If $U$ is close to $x$ and far enough from $y$, by Cauchy--Schwarz inequality, we get
\[
\left| \,\mathbb E_{U} \left[ \frac{1}{n} \frac{\sin \left( \frac{n (x-U)}{2}\right)}{\sin \left( \frac{(x-U)}{2}\right)} \mathds{1}_{d(x,U)\leq \varepsilon} \frac{\sin \left( \frac{n (y-U)}{2}\right)}{\sin \left( \frac{(y-U)}{2}\right)} \mathds{1}_{d(y,U)>\varepsilon}\right] \, \right|\leq \frac{C}{\sqrt{n} \varepsilon} \mathbb E_U \left[ \frac{1}{n} \frac{\sin^2 \left( \frac{n (x-U)}{2}\right)}{\sin^2 \left( \frac{(x-U)}{2}\right)} \right]^{1/2},
\]
that is to say
\[
\left| \,\mathbb E_{\mu_{\rho}} \left[ \frac{1}{n} \frac{\sin \left( \frac{n (x-U)}{2}\right)}{\sin \left( \frac{(x-U)}{2}\right)} \mathds{1}_{d(x,U)\leq \varepsilon} \frac{\sin \left( \frac{n (y-U)}{2}\right)}{\sin \left( \frac{(y-U)}{2}\right)} \mathds{1}_{d(y,U)>\varepsilon}\right] \, \right| \leq 
\frac{C}{\sqrt{n} \varepsilon}  \sqrt{K_n \ast \mu_{\rho}(x)}.
\]
Finally, if $U$ is close to both $x$ and $y$, Cauchy--Schwarz inequality gives again
\[
\left| \,\mathbb E_{\mu_{\rho}} \left[ \frac{1}{n} \frac{\sin \left( \frac{n (x-U)}{2}\right)}{\sin \left( \frac{(x-U)}{2}\right)}  \frac{\sin \left( \frac{n (y-U)}{2}\right)}{\sin \left( \frac{(y-U)}{2}\right)} \mathds{1}_{\substack{d(x,U)\leq \varepsilon \\ d(y,U)\leq \varepsilon} }\right] \, \right|\leq C \sqrt{K_n \ast \mu_{\rho}(x)} \sqrt{K_n \ast \mu_{\rho}(y)} \mathds{1}_{d(x,y) \leq 2\varepsilon}.
\]
\end{proof}

From Lemma \ref{lem.twopoints} and Lemma \ref{lm.polfej}, one can then deduce the following estimates which will play a key role in the proof of the Central Limit Theorems \`a la Salem--Zygmund. 
Recall that the probability $\mathbb P$ and the expectation $\mathbb E$ are the ones associated with the random coefficients $(a_k,b_k)$ of the trigonometric polynomial $f_n(x)$. Now, if $X$ and $Y$ are two independent uniform random variables in $[0,2\pi]$, independent of the random coefficients $(a_k,b_k)$, we will denote by $\mathbb P_{X,Y}$ and $\mathbb E_{X,Y}$ the associate law and expectation and by $\mathbb P_X$, $\mathbb P_Y$, $\mathbb E_X$, $\mathbb E_Y$ the marginals.

\begin{lma} \label{lm.cov} Let $X$ and $Y$ be two independent random variables with uniform distribution in $[0,2\pi]$, independent of the random coefficients $(a_k,b_k)$. There exists a universal constant $C>0$ such that, for any $\varepsilon>0$ and any integer $n \geq 1$
\[
\Esp_{X,Y}\left|\Esp[f_n(X)f_n(Y)]\right| \leq C \left(\frac{1}{n\varepsilon^2}+\frac{1}{\sqrt{n}\varepsilon}+\sqrt{\varepsilon}\right).
\]
In particular, choosing $\varepsilon = n^{-1/3}$, we have 
\[
\Esp_{X,Y}\left|\Esp[f_n(X)f_n(Y)]\right| \leq \frac{C}{n^{1/6}}.
\]
\end{lma}

\begin{proof}
Combining Equation \eqref{eq.twopoints} of Lemma \ref{lem.twopoints} and Lemma \ref{lm.polfej}, we have 
\[
\begin{array}{ll}
\displaystyle{\Esp_{X,Y}\left|\Esp[f_n(X)f_n(Y)]\right|}  & \displaystyle{ \leq \frac{C}{\sqrt{n}\epsilon} \left(\mathbb E_X\left[ \sqrt{K_n \ast \mu_{\rho}(X)}\right]+\mathbb E_Y\left[\sqrt{K_n \ast \mu_{\rho}(Y)} \right]\right) }\\
\\
& \displaystyle{+\frac{1}{n\epsilon^2} +\Esp_{X,Y}\left[ \sqrt{K_n \ast \mu_{\rho}(x)}\sqrt{K_n \ast \mu_{\rho}(y)}\mathds{1}_{d(x,y)\leq 2 \epsilon} \right].}
\end{array}
\]
By Cauchy--Schwarz inequality, we have $\mathbb E_X\left[ \sqrt{K_n \ast \mu_{\rho}(X)}\right]^2 \leq \mathbb E_X[ K_n \ast \mu_{\rho}(X)] $. Since $\|K_n\|_1=1$ and $\mu_{\rho}$ is a probability measure, by Fubini we have
\[
\mathbb E_X[ K_n \ast \mu_{\rho}(X)] = \frac{1}{2\pi} \int_0^{2\pi}\int_0^{2\pi} K_n(x-u) \mu_{\rho}(du) dx = \int_0^{2\pi} \left(\frac{1}{2\pi}\int_0^{2\pi} K_n(x-u) dx\right) \mu_{\rho}(du)  =1.
\] 
By Cauchy--Schwarz again, we get
\[
\mathbb E_{X,Y} \left[  \sqrt{K_n \ast \mu_{\rho}(X)} \sqrt{K_n \ast \mu_{\rho}(Y)} \mathds{1}_{d(X,Y) | \leq 2 \varepsilon}\right]  \leq \sqrt{\mathbb E_{X,Y} \left[ K_n \ast \mu_{\rho}(Y) \mathds{1}_{d(X,Y) \leq 2\varepsilon}\right] }
\] 
and again using Fubini inversion (integrating first in $X$ and then in $Y$)
\[
\mathbb E_{X,Y} \left[ K_n \ast \mu_{\rho}(Y) \mathds{1}_{d(X,Y) | \leq 2\varepsilon}\right]  =4\varepsilon. 
\]
\end{proof}

\subsection{A unidimensional Central Limit Theorem}\label{sec.un.SZ}

We are now in position to give the proof of Theorem \ref{prop.conv.fc} stated in the introduction. We follow the main global strategy as the original proof of the Central Limit Theorem by Salem--Zygmund in \cite{SZ54}, i.e. we first establish a $L^2$ estimate and then conclude by a Borel--Cantelli type argument. Recall that, as defined just before Lemma \ref{lm.cov}, 
the probability $\mathbb P$ and the expectation $\mathbb E$ are associated with the random coefficients, whereas $\mathbb P_X$, $\mathbb E_X$, $\mathbb P_{X,Y}$ and $\mathbb E_{X,Y}$ are associated with the random evaluation points $X$ and $Y$, that are uniformly distributed in $[0, 2\pi]$, independent and independent of the coefficients.
Since $\mathbb E[f_n^2(X)]={2 \pi}\,K_n\ast \mu_{\rho}(X) $, if $Y$ is an independent copy of $X$, Fubini inversion and direct calculation yield
\[
\begin{array}{ll}
\displaystyle{\Delta_n} & :=\displaystyle{\mathbb E \left[ \left|\, \mathbb E_X \left[ e^{i t f_n(X)}  \right] - \mathbb E_X \left[ e^{- \frac{t^2}{2} {2 \pi}\, K_n\ast \mu_{\rho}(X)}  \right]\, \right|^2\right]}\\
\\
& = \displaystyle{\mathbb E_{X,Y} \left[   
e^{-\frac{t^2}{2} \mathbb E[(f_n(X)-f_n(Y))^2]} -e^{- \frac{t^2}{2} {2 \pi} \left( K_n\ast \mu_{\rho}(X) +K_n\ast \mu_{\rho}(Y)\right) }
\right].}
\end{array}
\]
Thanks to the fact that $x \mapsto e^{-x}$ is $1-$Lipschitz on $\mathbb R^+$, we deduce that 
\[
\begin{array}{ll}
\displaystyle{\Delta_n} 
& \leq \displaystyle{\frac{t^2}{2} \, \mathbb E_{X,Y} \left[  \left|    \mathbb E[(f_n(X)-f_n(Y))^2] - {2 \pi} \left( K_n\ast \mu_{\rho}(X) +K_n\ast \mu_{\rho}(Y)\right) \right|
\right]}\\
\\
& = \displaystyle{t^2 \, \mathbb E_{X,Y} \left[  \left|    \mathbb E[f_n(X)f_n(Y)]  \right|\right].}
\end{array}
\]
Using Lemma \ref{lm.cov} with $\varepsilon = n^{-1/3}$, we obtain that there exists a universal constant $C>0$ such that
\[
\Delta_n \leq Ct^2n^{-1/6}.
\]
By Borel--Cantelli Lemma, we then deduce that $\Prob$-almost surely, as $n$ goes to infinity, we get
\[
\left| \Esp_X\left[e^{itf_{n^{7}}(X)}-e^{-\frac{t^2}{2} {2 \pi}\, K_{n^7} \ast \mu_{\rho}(X)}\right]\right| \to 0.
\]
Now, let $m$ be a positive integer, there exists a unique $n$ such that  $n^7 < m \leq (n+1)^7$. Using Birkhoff-Khinchine Theorem, we have then the following Lemma whose proof is given in Section \ref{Gaussian-ergo} of the Appendix.

\begin{lma}\label{lem.birk}
As $m$ goes to infinity
\[
\left| \, \mathbb E_X \left[ e^{i t f_{n^7}(X)} \right]-\mathbb E_X \left[ e^{i t f_{m}(X)} \right] \right|= O\left(\frac{1}{m^{1/14}} \right).
\]
\end{lma}
\noindent
Combining this last Lemma \ref{lem.birk} and Lemma \ref{lm.fej.leb}, we deduce that as $m$ and hence $n$ both go to infinity, then $\Prob-$almost surely
\[
\left|\Esp_X\left[e^{-\frac{t^2}{2} {2 \pi}\, K_{n^7}\ast \mu_{\rho}(X)}-e^{-\frac{t^2}{2} {2 \pi}\, K_m\ast \mu_{\rho}(X)}\right]\right| \to 0.
\]
Therefore, by triangular inequality, we obtain that as $m$ goes to infinity, $\Prob-$almost surely,
\[
\left|\Esp_X\left[e^{itf_m(X)}-e^{-\frac{t^2}{2} {2 \pi}\, K_m\ast \mu_{\rho}(X)}\right]\right|  \to 0. \]
By dominated convergence theorem, using Lemma \ref{lm.fej.leb} again, we finally obtain that $\Prob-$almost surely
\[
\lim_{m \to +\infty} \Esp_X\left[e^{itf_m(X)}\right]= \frac{1}{2\pi} \int_0^{2\pi}e^{-\frac{t^2}{2}{2 \pi}\, \psi_{\rho}(x)}dx= \Esp_{X,N}\left[e^{it\sqrt{{2 \pi}\, \psi_{\rho}(X)}N}\right].\]

\subsection{A functional Central Limit Theorem}\label{sec.SZ.func}
In this section, we give a detailed proof of Theorem \ref{thm.SZ.func} stated in the introduction. Let $g_{n}$ the stochastic process defined on $[0,2\pi]$ by
\[
g_n(t):=f_n\left(X+\frac{t}{n}\right), \;\; t \in [0, 2\pi],
\]
then Theorem \ref{thm.SZ.func} precisely asserts that, $\mathbb P-$almost surely, $(g_n(t))_{t \in [0, 2\pi]}$ converges in distribution under $\mathbb P_X$ towards an explicit Gaussian process in the $\mathcal C^1$ topology. As classically done, in order to establish this statement, we will first prove the convergence of finite dimensional marginals and then we will invoke a tightness argument.

\subsubsection{Convergence of finite dimensional marginals}
Let us first establish the convergence of the finite dimensional marginals of $(g_n(t))_{t \in [0,2\pi]}$.

\begin{prop}\label{pro.fin.marg}$\Prob-$almost surely, as $n$ goes to infinity, the finite marginals of the localized process  $(g_n(t))_{t\in[0,2\pi]}$ converge to the ones of a process $(g_{\infty}(t))_{t\in [0,2\pi]}$ of the form $\sqrt{\psi_{\rho}(X)}N$, where the process $N=(N_t)_{t\in [0,2\pi]}$ is the stationary Gaussian process with $\sin_c$ as covariance function and independent of $X$. 
\end{prop}

We fix a positive integer $M$, and two $M-$uplets $t=(t_1, \ldots, t_M) \in [0, 2\pi]^M$ and $\lambda=(\lambda_1, \ldots, \lambda_M)\in \mathbb R^M$ and we set
\[
Z_n(X,t,\lambda) := \sum_{p=1}^M \lambda_p g_n(t_p) = \sum_{p=1}^M \lambda_p f_n\left(X+\frac{ t_p}{n}\right).
\] 
Proving Proposition \ref{pro.fin.marg} then amounts to show that $\mathbb P-$almost surely, we have 
\[
\lim_{n \to +\infty} \mathbb E_X\left[e^{ i Z_n(X,t,\lambda)}\right] =\mathbb E_X\left[e^{ -\frac{1}{2} \times {2 \pi}\sum_{p,q=1}^M \lambda_p \lambda_q \psi_{\rho}(X) \sin_c(t_p-t_q)} \right].
\]

The proof follows globally the same lines as its one dimensional analogue Theorem \ref{prop.conv.fc}. First the variance of $Z_n(X,t,\lambda)$ under $\mathbb P$ can be represented by a convolution with a Fej\'er-like kernel. 

\begin{lma}\label{lem.rep.Z}
We have the representation
\[
\mathbb E\left[Z_n(X,t,\lambda)^2\right] = {2\pi}\, K_n^{t,\lambda} \ast \mu_{\rho}(X) = \left( \int_0^{2\pi} K_n^{t,\lambda}(x)dx\right) \bar{K}_n^{t,\lambda}\ast \mu_{\rho}(X),
\]
with 
\[
K_n^{t,\lambda}(x):=\frac{1}{n}\left| \sum_{p=1}^{M}\lambda_p e^{i\frac{(n+1)}{2n}t_p} \frac{\sin\left(\frac{n}{2}(x+\frac{t_p}{n})\right)}{\sin\left(\frac{x+\frac{t_p}{n}}{2}\right)}\right|^2, \quad 
\bar{K}_n^{t,\lambda}(x):=\frac{{2\pi} \, K_n^{t,\lambda}(x)}{\int_0^{2\pi}K_n^{t,\lambda}(x)dx}.
\]
\end{lma}

\begin{proof}[Proof of Lemma \ref{lem.rep.Z}]By linearity, we have
\begin{equation} \label{eq.RnX}
\mathbb E\left[Z_n(X,t,\lambda)^2\right]=\sum_{p,q=1}^{M} \lambda_p \lambda_q \mathbb E\left[f_n\left(X+\frac{t_p}{n} \right)f_n\left(X+\frac{t_q}{n} \right)\right]
\end{equation}
and as in the proof of Lemma \ref{lem.twopoints} we have
\[
\begin{array}{ll}
\displaystyle{\mathbb E\left[f_n\left(X+\frac{t_i}{n} \right)f_n\left(X+\frac{t_j}{n} \right)\right]}
& = \displaystyle{\Re \left(  \frac{1}{n} \sum_{k,l=1}^n \rho(k-l) e^{ i k (X+\frac{t_p}{n}) - i l (X+\frac{t_q}{n})} \right)}. 
\end{array}
\]
Writing for all integer $k$ that $\rho(k)=\int_0^{2\pi} e^{-i kx} \mu_{\rho}(dx)$, we obtain 
\[
\begin{array}{ll}
&\displaystyle{\frac{1}{n} \sum_{k,l=1}^n \rho(k-l) e^{ i k (X+\frac{t_p}{n}) - i l (X+\frac{t_q}{n})}} 
 =\displaystyle{ e^{i\frac{t_p-t_q}{n}}  \int_0^{2\pi}   \left( \frac{1}{n} \sum_{k,l=0}^{n-1}  e^{i k(X+\frac{t_p}{n}-x)} e^{-i \,l (X+\frac{t_q}{n}-x)} \right) \mu_{\rho}(dx)}\\
\\\\
 & =\displaystyle{e^{i\frac{t_p-t_q}{n}}  \int_0^{2\pi}   \left( \frac{1}{n} \frac{1-e^{i n (X+\frac{t_p}{n}-x)}}{1-e^{i(X+\frac{t_q}{n}-x)}} \frac{1-e^{-i \,n (X+\frac{t_p}{n}-x)}}{1-e^{-i (X+\frac{t_q}{n}-x)}} \right) \mu_{\rho}(dx)}
\\\\
 & =\displaystyle{e^{i\frac{(n+1)}{2n}(t_p-t_q)} \int_0^{2\pi}   \left( \frac{1}{n} \frac{\sin \left( \frac{n (X+\frac{t_p}{n}-x)}{2}\right)}{\sin \left( \frac{(X+\frac{t_p}{n}-x)}{2}\right)} \frac{\sin \left( \frac{n (X+\frac{t_q}{n}-x)}{2}\right)}{\sin \left( \frac{(X+\frac{t_q}{n}-x)}{2}\right)} \right) \mu_{\rho}(dx)}\\
&= \displaystyle{\frac{1}{n} \int_0^{2\pi} \underbrace{\left(  e^{i\frac{(n+1)}{2n}t_p}\frac{\sin \left( \frac{n (X+\frac{t_p}{n}-x)}{2}\right)}{\sin \left( \frac{(X+\frac{t_p}{n}-x)}{2}\right)} \right)}_{:=z_p} \underbrace{e^{-i\frac{(n+1)}{2n}t_q}\left(\frac{\sin \left( \frac{n (X+\frac{tq}{n}-x)}{2}\right)}{\sin \left( \frac{(X+\frac{t_q}{n}-x)}{2}\right)} \right)}_{:=\overline{z_q}} \mu_{\rho}(dx)}.
\end{array}
\]
Summing up on $p$ and $q$, by symmetry we obtain
\[
\begin{array}{ll}
\displaystyle{\sum_{p,q=1}^M \lambda_p \lambda_q \left( \frac{1}{n} \sum_{k,l=1}^n \rho(k-l) e^{ i k (X+\frac{t_p}{n}) - i l (X+\frac{t_q}{n})}\right)} & \displaystyle{=  \frac{1}{n} \int_0^{2\pi}  \left( \sum_{p,q=1}^M \lambda_p z_p \lambda_q \overline{z_q}\right) \mu_{\rho}(dx)}\\
\\
& \displaystyle{=  \frac{1}{n} \int_0^{2\pi}\left| \sum_{p=1}^M \lambda_p z_p \right|^2 \mu_{\rho}(dx)}.
\end{array}
\]
As a result 
\[
\begin{array}{ll}
\mathbb E\left[Z_n(X,t,\lambda)^2\right] & \displaystyle{= \int_0^{2\pi}\frac{1}{n} \left|\sum_{p=1}^M \lambda_p e^{i\frac{(n+1)}{2n}t_p}\frac{\sin \left( \frac{n (X-x+\frac{t_p}{n})}{2}\right)}{\sin \left( \frac{(X-x+\frac{t_p}{n})}{2}\right)}  \right|^2 \mu_{\rho}(dx)}.
\end{array}
\]
With the definitions of $K_n^{t,\lambda}$ and $\bar{K}_n^{t,\lambda}$ above, it gives
\[
\mathbb E\left[Z_n(X,t,\lambda)^2\right]= {2\pi}\, K_n^{t,\lambda} \ast \mu_{\rho}(X)=\left(\int_0^{2\pi} K_n^{t,\lambda}(x) dx \right)\times \bar{K}_n^{t,\lambda} \ast \mu_{\rho}(X).
\]
\end{proof}
As the standard Fej\'er kernel, the normalized kernel $\bar{K}_n^{t,\lambda}$ is a good trigonometric kernel and we have the  corresponding Fej\'er--Lebesgue Theorem whose proof is given in Section \ref{sec.trigo.app} of the Appendix. 
\begin{lma} \label{lm.fej.leb.multi}
As $n$ goes to infinity, for almost every $x\in[0,2\pi]$, we have 
\[
\bar{K}_n^{t,\lambda} \ast \mu_{\rho}(x) \rightarrow \psi_{\rho}(x).
\]
\end{lma}
\noindent
Let us now make explicit the asymptotics of the normalization factor. 

\begin{lma} \label{lm.fej.leb.multibis}
As $n$ goes to infinity, we have 
\[
\int_0^{2\pi} K_n^{t,\lambda}(x) dx =2 \pi \sum_{p, q=1}^m \lambda_p \lambda_q \sin_c\left(t_p-t_q\right) + O\left( \frac{1}{n} \right).
\]
\end{lma}

\begin{proof}[Proof of Lemma \ref{lm.fej.leb.multi}]Note that the integral $ \frac{1}{2\pi}\int_0^{2\pi} K_n^{t,\lambda}(x) dx$ can be seen as the convolution of  the kernel 
$K_n^{t,\lambda}$ with the normalized Lebesgue measure on $[0, 2\pi]$ which is the spectral measure associated with $\rho(k)=\delta_0(k)$ i.e. the independent case. As a result, making $\rho(k-l)=\delta_{k,l}$ in the proof of Lemma \ref{lem.rep.Z}, we get
\[
\frac{1}{2\pi} \int_0^{2\pi} K_n^{t,\lambda}(x) dx = \sum_{p,q=1}^{M}\lambda_p \lambda_q  \left(\frac{1}{n}\sum_{k=1}^{n}\cos \left(\frac{k(t_p-t_q)}{n}\right)\right).
\]
Since the cosine function has a bounded derivative, by standard comparison results between Riemann sums and their limits, we conclude
$$\left| \frac{1}{2\pi} \int_0^{2\pi}K_n^{t,\lambda}(x)dx-\sum_{p, q=1}^M \lambda_p \lambda_q \sin_c\left(t_p-t_q\right) \right|=O\left(\frac{1}{n}\right).$$
\end{proof}
\noindent
Combining Lemmas \ref{lem.rep.Z}, \ref{lm.fej.leb.multi} and \ref{lm.fej.leb.multibis}, we have thus as $n$ goes to infinity
\begin{equation}\label{eq.conv.varZ}
\lim_{n \to +\infty} \mathbb E\left[Z_n(X,t,\lambda)^2\right] = 2\pi \psi_{\rho}(X) \sum_{p, q=1}^M \lambda_p \lambda_q \sin_c\left(t_p-t_q\right).
\end{equation}
Note that in the independent case, since $\psi_{\rho} \equiv 1/2\pi$, we recover the standard $\sin_c$ correlation function. We proceed now as in the proof of Theorem \ref{prop.conv.fc} and establish an $\mathbb L^2$ estimate. 

\begin{lma}As $n$ goes to infinity, we have 
\[
\Delta_n:=\Esp \left| \Esp_X\left[e^{i Z_n(X,t,\lambda)}\right]-e^{-\frac{1}{2}\mathbb E\left[Z_n(X,t,\lambda)^2\right] }\right|^2=O\left(n^{-1/6}\right).
\]
\end{lma}
\begin{proof}
Exactly as in the proof of Theorem \ref{prop.conv.fc} given in Section \ref{sec.un.SZ}, if $X$ and $Y$ are two independent random variables with uniform distribution in $[0,2\pi]$, independent of the random coefficients $(a_k,b_k)$, we have 
\[
\Delta_n \leq \mathbb E_{X,Y} \left[ \left| \mathbb E\left[Z_n(X,t,\lambda)Z_n(Y,t,\lambda)\right] \right| \right].
\]
Moreover, as in the proof of Lemma \ref{lm.polfej}, if $U$ is an independent variable with law $\mu_{\rho}$, one can write 
\[\begin{array}{l}
\mathbb E\left[Z_n(X,t,\lambda)Z_n(Y,t,\lambda)\right]= \displaystyle{\sum_{p,q=1}^{M}\lambda_p \lambda_q \Esp\left[ f_n\left(X+\frac{t_p}{n}\right)f_n\left(Y+\frac{t_q}{n}\right)\right]}\\
\\
=\displaystyle{\sum_{p,q}\lambda_p \lambda_q \cos \left (\frac{n+1}{2}(X-Y+\frac{t_p-t_q}{n})\right)\Esp_{U} \left[\frac{1}{n}\frac{\sin\left(\frac{n}{2}(X-U+\frac{t_p}{n}) \right)}{\sin\left(\frac{X-U}{2} +\frac{t_p}{2n}\right)}\frac{\sin\left(\frac{n}{2}(Y-U+\frac{t_q}{n}) \right)}{\sin\left(\frac{Y-U}{2} +\frac{t_q}{2n}\right)}  \right]}.
\end{array}
\]
Proceeding as in the proofs of Lemmas \ref{lm.polfej} and \ref{lm.cov}, one then deduces that 
\[
\mathbb E_{X,Y} \left[ \left| \mathbb E\left[Z_n(X,t,\lambda)Z_n(Y,t,\lambda)\right] \right| \right] \leq \sum_{p,q=1}^{M} |\lambda_p \lambda_q | \times O\left(\frac{1}{n\varepsilon^2}+\frac{1}{\sqrt{n}\varepsilon}+\sqrt{\varepsilon}\right).
\]
In particular, choosing $\varepsilon = n^{-1/3}$, we get the desired result.

\end{proof}

As above, to deduce the almost sure asymptotics starting from the $\mathbb L^2$ estimate, we invoke a Borel--Cantelli argument. 
Along the subsequence $n^7$, $\Prob-$almost surely, we have 
\[
\lim_{n \to +\infty} \left| \Esp_X\left[e^{i Z_{n^7}(X,t,\lambda)}\right] - \Esp_X\left[e^{-\frac{1}{2}\mathbb E\left[Z_{n^7}(X,t,\lambda)^2\right]}\right]\right| = 0.
\]
Then Birkhoff--Khinchine Theorem allows to establish the next lemma, which is the multidimensional analogue of Lemma \ref{lem.birk} and whose proof is also given in Section \ref{Gaussian-ergo} of the Appendix. 

\begin{lma}\label{lem.birk2}
As $m$ goes to infinity, if we choose $n$ such that $n^{7}<m\leq (n+1)^{7}$, then
\[
\left|\Esp_X\left[e^{iZ_{n^7}(X,t,\lambda)}\right]- \Esp_X\left[e^{iZ_{m}(X,t,\lambda)}\right]\right|=O\left( \frac{1}{\sqrt{n}} \right) =O\left( \frac{1}{m^{1/14}} \right).
\]
\end{lma}
\noindent
Therefore, we obtain that $\mathbb P-$almost surely, as $m$ goes to infinity
\[
\Esp_X\left[e^{i Z_m(X,t,\lambda)}\right]-e^{-\frac{1}{2}\mathbb E\left[Z_m(X,t,\lambda)^2\right] } \xrightarrow[m \to +\infty]{} 0,
\]
and by dominated convergence, using Equation \eqref{eq.conv.varZ}, we can conclude that
\[
\Esp_X\left[e^{i Z_m(X,t,\lambda)}\right]  \xrightarrow[m \to +\infty]{} \Esp_X\left[ \exp\left( - \frac{1}{2} \times {2 \pi}\,\psi_{\rho}(X) \sum_{p, q=1}^M \lambda_p \lambda_q \sin_c\left(t_p-t_q\right)\right)\right].
\]

\subsubsection{Tightness}
The convergence of the finite marginals of $g_n$ now established, in order to conclude that $g_n$ converges in distribution for the $\mathcal{C}^{1}$ topology towards $g_{\infty}$, we need to verify some tightness criteria for the $\mathcal{C}^1-$topology, which is the object of the following proposition. 
\begin{prop}\label{prop.tight}
Almost surely w.r.t. $\Prob$, the family of distributions under $\Prob_X$ of $(g_n(t))_{t \in [0,2\pi]}$ for $n\geq 1$ is tight w.r.t. the $\mathcal{C}^1$ topology on $\mathcal{C}^{1}([0,2\pi])$. 
\end{prop}
\begin{proof}
The strategy of the proof is the same as in Proposition 2 of [AP19], namely it is sufficient  to establish a Lamberti-type criteria for $\Esp_{X} \left|g_n(t)-g_n(s)\right|^2$ and $\Esp_{X}\left|g_n'(t)-g_n'(s)\right|^2$.  For sake of self-containedness, we recall the key elements of the proof.  
As detailed in Section \eqref{Gaussian-ergo}, of the Appendix, Birkhoff--Khinchine Theorem ensures that 
\[
C(\omega):=\sup_{n\geq 1}\frac{1}{2n}\sum_{k=1}^{n}(a_k^2+b_k^2)\]
 is $\Prob-$almost surely bounded. Using orthogonality in $ L^{2}([0,2\pi])$ for cosine and sine functions, we obtain that $\Prob$-almost surely, for all $s,t\in [0,2\pi]$,
\begin{eqnarray*}
\Esp_X\left|g_n(t)-g_n(s)\right|^2 &=&\frac{2}{n}\sum_{k=1}^{n}(a_k^2+b_k^2)\sin^2\left(\frac{k}{2n}(t-s)\right) \leq C(\omega)|t-s|^2,\\
\Esp_X\left|g_n'(t)-g_n'(s)\right|^2 &=&\frac{2}{n}\sum_{k=1}^{n}\frac{k^2}{n^2}(a_k^2+b_k^2)\sin^2\left(\frac{k}{2n}(t-s)\right) \leq C(\omega)|t-s|^2,
\end{eqnarray*}
hence the result.
\end{proof}

\if{
\begin{rmk} Recall that under the log-integrability assumption $\log(\psi_{\rho}) \in L^{1+\eta}([0,2\pi])$ for some $\eta \in (0,1)$, we have $\Prob$-a.s. that $\psi_{\rho}>0$. 
Since the limit process $g_{\infty}$ can be represented as $\sqrt{\psi_{\rho}(X)} N$ with $\Prob$-a.s. $\psi_{\rho}>0, ~X$ a uniform random variable and $N$ a standard Gaussian process that is non-degenerate, the non-degeneracy of $g_{\infty}$ follows.
\end{rmk}
}\fi

\section{From the limit theorems to the nodal asymptotics} \label{sec.nodal}
In this section we shall establish Theorems \ref{thm.asym.mean} and \ref{thm.as.asym} which show, under mild conditions on the spectral measure, the universal asymptotic behavior of the number of zeros, first in expectation and then almost surely. Note that, under the hypotheses of both Theorems, the spectral density $\psi_\rho$ is assumed to be positive almost everywhere. 

\medskip

Let us start by recalling the following deterministic result, the proof of which consists in a simple Fubini argument between the empirical measure of the roots of  a function $f$ and the Lebesgue measure over $[0,2\pi]$. We refer to \cite[Lemma 3]{AP19} for more details.
 
\begin{lma} \label{lm.num.zeros.unif} Let $X$ a random variable which is uniformly distributed over $[0,2\pi]$. Let us assume that $f$ is a $2\pi$-periodic function with a finite number of zeros in a period, then for any $0<h<2\pi$, we have
\[
\frac{h}{2\pi}\times \mathcal{N}(f,[0,2\pi])=\Esp_X\left[\mathcal{N}(f,[X,X+h])\right].
\]
\end{lma}

Applying Lemma \ref{lm.num.zeros.unif} to $f_n$ and $h=\frac{2\pi}{n}$, we obtain for $X$ a uniform random variable on $[0,2\pi]$ which is independent of the sequences $(a_k)_{k\geq 1}$ and $(b_k)_{k \geq 1}$ that
\begin{equation} \label{est.numb.zero.X}
\frac{\mathcal{N}(f_n,[0,2\pi])}{n} =\Esp_X \left[\mathcal{N}\left(f_n,\left[X,X+\frac{2\pi}{n}\right]\right)\right].
\end{equation}
The previous Equation (\ref{est.numb.zero.X}) legitimates the introduction of the stochastic process $(g_n(t))_{t \in [0,2\pi]}$ defined by  $g_n(t):= f_n \left(X+\frac{t}{n}\right)$, which is naturally the one studied in the last Section \ref{sec.SZ.func}, given that we have
\begin{equation}\label{numb.zeros.gn}
\frac{\mathcal{N}(f_n,[0,2\pi])}{n} =\Esp_X \left[\mathcal{N}\left(g_n,\left[0,2\pi\right]\right)\right].
\end{equation}

Indeed, one can then reasonably guess that the convergence of the random processes $\{g_n(\cdot)\}_{n\ge 1}$ in a suitable functional space will imply the convergence in law of the sequence of random variables $\{\mathcal{N}\left(g_n,\left[0,2\pi\right]\right)\}_{n\ge 1}.$

\par
\medskip

Let us make this heuristics rigorous. The limit process $g_\infty$, which is given by Theorem \ref{thm.SZ.func}, may be interpreted as a stationary Gaussian process with correlation $\sin_c$ multiplied by a random and independent amplitude given by $\sqrt{\psi_\rho(X)}$. As mentioned above, under the assumptions of Theorems \eqref{thm.asym.mean} and \eqref{thm.as.asym},  $\psi_\rho(X)>0$ almost surely. Given that the stationary Gaussian process with correlation $\sin_c$ is non degenerated, we derive that the limit process $g_\infty$ is non degenerated as well. Hence,

\begin{equation}\label{non-dege}
\mathbb{P}\otimes\mathbb{P}_X\text{-a.s.},\,\forall t\in[0,2\pi],\,|g_\infty(t)|+|g_\infty'(t)|>0.
\end{equation}

On the other hand, one has the following deterministic result, which ensures the continuity of the number of roots with respect to $\mathcal{C}^1$ topology provided that the limit is non degenerated.
\begin{equation}\label{deter-cv}
\left.\begin{array}{l}
u_n\xrightarrow[n\to\infty]{\mathcal{C}^1([0,2\pi])}~u\\
\\
\inf_{t \in[0,2\pi]}\left(|u(t)|+|u'(t)|\right)>0
\end{array}
\right\}
\Rightarrow~\text{Card}\left(u_n^{-1}\left(\left\{0\right\}\right)\cap[0,2\pi]\right)\to\text{Card}\left(u^{-1}\left(\left\{0\right\}\right)\cap[0,2\pi]\right).
\end{equation}

Combining Theorem \eqref{thm.SZ.func}, the $\mathcal C^1$ continuity \eqref{deter-cv} of the number of zeros and the non-degeneracy of $g_\infty$ given by \eqref{non-dege},  together with the continuous mapping Theorem implies the following proposition which will be central in our  forthcoming proofs.

\begin{prop} \label{prop.conv.dist}
Consider the localized process $(g_n(t))_{t \in [0,2\pi]}:=\left(f_n\left(X+\frac{t}{n}\right)\right)_{t \in [0,2\pi]}$ 
\begin{itemize}
\item[(a)] $\Prob$-almost surely, as $n$ goes to infinity, $\mathcal{N}(g_n,[0,2\pi])$ converges in distribution under $\Prob_X$ towards $\mathcal{N}(g_{\infty}, [0,2\pi])$.
\item[(b)] As $n$ goes to infinity, the number of zeros $\mathcal{N}(g_n,[0,2\pi])$ converges in distribution under $\Prob\otimes\Prob_X$ towards $\mathcal{N}(g_{\infty}, [0,2\pi])$.
\end{itemize}
\end{prop}

Note that assertion (b) is a direct consequence of assertion (a) as it is sufficient to integrate it with respect to $\Prob$ and pass to the limit. Indeed, if $h$ is a continuous and bounded test function and as $n$ goes to infinity we have
\[
\mathbb E_X\left[  h \left( \mathcal{N}(g_n,[0,2\pi])\right) \right] \to \mathbb E_X\left[  h \left( \mathcal{N}(g_{\infty},[0,2\pi])\right) \right], 
\]
then by dominated converge, we have also
\[
\Esp_{\Prob \otimes \Prob_X}\left[  h \left( \mathcal{N}(g_n,[0,2\pi])\right) \right] \to  \Esp_{\Prob \otimes \Prob_X} \left[  h \left( \mathcal{N}(g_{\infty},[0,2\pi])\right) \right]. 
\]

\subsection{Study of the mean number of real zeros of $f_n$.}
The object of this section is to give the proof of Theorem \ref{thm.asym.mean}, which, thanks to Equation \eqref{numb.zeros.gn}, turns out to be equivalent to showing that 
\begin{equation} \label{eq.conv.first.mom}
\lim_{n \to +\infty}\Esp_{\Prob \otimes \Prob_X}\left[\mathcal{N}(g_n,[0,2\pi])\right]= \Esp_{\Prob \otimes \Prob_X}\left[\mathcal{N}(g_{\infty},[0,2\pi])\right].
\end{equation}
We stress the fact that here, no condition is required on the singular component $\mu_{\rho}^s$ of the spectral measure $\mu_{\rho}$.
In order to obtain the convergence of the first moment in the equation \eqref{eq.conv.first.mom}, the convergence in distribution given by the Proposition \ref{prop.conv.dist} under $\Prob\otimes\Prob_X$ is not sufficient. We need to prove some equi-integrability condition. To achieve this, we first aim at proving some logarithmic moment estimates for $f_n(X)=g_n(0)$.

\begin{lma} \label{lm.est.log.P}
Let $\gamma >1$. If $\log(\psi_{\rho})\in L^{\gamma}$, there exists a constant $C_{\gamma}>0$ such that uniformly on $n \geq 1$, we have
\[
\Esp_{\Prob \otimes \Prob_X} \left| \log(|f_n(X)|)\right|^{\gamma} \leq C_{\gamma}(1+\|\log(\psi_{\rho})\|_{L^{\gamma}}).
\]
\end{lma}
\begin{proof}
Notice that, {conditionally on $X$}, $f_n(X)$ is a centered Gaussian variable whose variance under $\Prob$ is given by $\Esp[f_n(X)^2]={2 \pi}\,  K_n \ast \mu_{\rho}(X)$. Since $K_n \ast \mu_{\rho}(X)\geq K_n \ast \psi_{\rho}(X)>0,$ $\Prob_X$-almost surely, we have
$$\Esp_{\Prob \otimes \Prob_X} \left|\log(|f_n(X)|)\right|^{\gamma} \leq C_{\gamma} \left(\Esp_{\Prob\otimes\Prob_X}\left| \log \left( \left|\frac{f_n(X)}{\sqrt{K_n\ast \mu_{\rho}(X)}}\right| \right)\right|^{\gamma}+\Esp_X\left|\log(K_n \ast \mu_{\rho}(X)\right|^{\gamma}\right).$$
Now, for every fixed $X$, under $\Prob,~\frac{f_n(X)}{\sqrt{K_n \ast \mu_{\rho}(X)}}$ is a standard Gaussian variable hence by Fubini inversion
$$\Esp_{\Prob \otimes \Prob_X} \left| \log \left( \left|\frac{f_n(X)}{\sqrt{K_n\ast \mu_{\rho}(X)}}\right| \right)\right|^{\gamma}=\kappa_{\gamma}<+\infty.$$
To finish the proof, one is left to control the term $\Esp_X\left|\log(K_n\ast \psi_{\rho}(X)\right|^{\gamma}$. 
\begin{itemize}
\item Assume first that $K_n \ast \mu_{\rho}(X) \in (0,1]$. Since $|\log(\cdot)|$ is non increasing on $(0,1]$, 
\[
\left|\log(K_n \ast \mu_{\rho}(X))\right|=\left| \log\left(K_n \ast \psi_{\rho}(X)+K_n\ast \mu_s(X)\right)\right| \leq |\log(K_n \ast \psi_{\rho}(X))|.
\]
Our assumption implies that $K_n \ast \psi_{\rho}(X) \in (0,1]$, hence by Jensen inequality, we have
\[
\log(K_n \ast \psi_{\rho}(X))\geq K_n \ast \log(\psi_{\rho})(X).
\]
As a result, we obtain that
\[
|\log(K_n\ast \psi_{\rho}(X)|^{\gamma}\hspace{-1 cm}\stackrel{\begin{array}{c}x\mapsto|x|^\gamma\text{decreases on}~\mathbb{R}^{-}\\\underbrace{K_n\ast \psi_\rho(X)<1}\\\\\end{array}}{\leq} \hspace{-1cm}|K_n\ast \log(\psi_{\rho})(X)|^{\gamma} \hspace{- 1cm}\stackrel{\begin{array}{l}\underbrace{\text{convexity of} ~x\mapsto |x|^\gamma}\\\\
\end{array}}{\leq} \hspace{-1cm}K_n \ast |\log(\psi_{\rho})|^{\gamma}(X).
\]

\item Assume now that $K_n\ast \mu_{\rho}(X)>1$. There exists a constant $C_{\gamma}$ such that  $|\log(x)|^\gamma \le C_\gamma |x|$ on $[1,+\infty[$ and thus
\[\left|\log(K_n \ast \mu_{\rho}(X))\right|^{\gamma} \leq C_{\gamma} K_n \ast \mu_{\rho}(X).\]
\end{itemize}
 Putting everything together, 
\[
|\log(K_n\ast \mu_{\rho}(X))|^{\gamma} \leq C_{\gamma}\left(K_n \ast \mu_{\rho}(X)+ K_n \ast |\log(\psi_{\rho})|^{\gamma}(X) \right).
\]
 and taking the expectation w.r.t. $\Prob_X$ in the previous inequality, we obtain the following estimate which is uniform on $n\ge 1$.
\begin{eqnarray*}
\Esp_X\left|\log(K_n\ast\psi_{\rho}(X))\right|^{\gamma}& \leq &C_{\gamma}\left(1+ \Esp_X\left[K_n \ast \left| \log(\psi_{\rho})\right|^{\gamma}(X)\right]\right)\\
&\leq &C_{\gamma}(1+\|\log(\psi_{\rho})\|_{L^{\gamma}})
\end{eqnarray*}
\end{proof}
Let us now see how the previous logarithmic estimate can be used to obtain some moment estimates for the number of zeros of $g_n$, that is to say for $\mathcal{N}(g_n,[0,2\pi])$.
\begin{prop}\label{prop.sup.mom}Let $\eta>0$. If $\log(\psi_{\rho}) \in L^{1+\eta}$, we have
\[
\sup_{n \geq 1}\Esp_{\Prob \otimes \Prob_X}\left [\left| \mathcal{N}(g_n,[0,2\pi])\right|^{1+\eta/2} \right] < +\infty.
\]
\end{prop}
\begin{proof}
We have classically
$$\Esp_{\Prob \otimes \Prob_X} \left[\mathcal{N}(g_n,[0,2\pi])^{1+\eta/2}\right] =(1+\eta/2)\int_0^{+\infty}s^{\eta/2} \Prob \otimes \Prob_X \left(\mathcal{N}(g_n,[0,2\pi])>s\right)ds.$$
By iterating Rolle Lemma $\left\lfloor s\right\rfloor$ times, we have
\[
\Prob\otimes \Prob_X\left(\mathcal{N}(g_n,[0,2\pi])>s\right) \leq \Prob\otimes \Prob_X\left(|g_n(0)|\leq \frac{(2\pi)^{\left\lfloor s\right\rfloor}}{\left\lfloor s\right\rfloor!}\|g_n^{(\left\lfloor s\right\rfloor)}\|_{\infty}\right),
\]
so that for any $R>0$,
\begin{equation} \label{eq.P.opt}
\Prob \otimes \Prob_X\left(\mathcal{N}(g_n,[0,2\pi])>s\right) \leq \Prob\otimes \Prob_X\left(|g_n(0)|\leq \frac{(2\pi)^{\left\lfloor s\right\rfloor}R}{\left\lfloor s\right\rfloor!}\right) +\Prob\otimes \Prob_X\left(\|g_n^{(\left\lfloor s\right\rfloor)}\|_{\infty}>R\right).
\end{equation}
Applying Markov inequality, we get
\[
\Prob\otimes \Prob_X\left(\|g_n^{(\left\lfloor s\right\rfloor)}\|_{\infty}>R\right)\leq \frac{\Esp_{\Prob\otimes \Prob_X} \|g_n^{(\left\lfloor s\right\rfloor)}\|_{\infty}^{2}}{R^2}.
\]
Using $L^2$-Sobolev embedding, we can compare the uniform norm with the norm $\| \cdot\|_2$ of $L^2([0,2\pi])$ of the derivatives, more precisely
\[
\Esp_{X}\|g_n^{(\left\lfloor s\right\rfloor)}\|_{\infty}^{2} \leq C \left(\Esp_{X}\|g_n^{(\left\lfloor s\right\rfloor)}\|_2^{2}+\Esp_{X} \|g_n^{(\left\lfloor s\right\rfloor+1)}\|_2^{2} \right).
\]
But for each $\ell \geq 1$, since for $(X,Y)$ uniformly distributed on $[0,2\pi]$ and independent we have $X+\frac{Y}{n}\sim X$, we deduce that
$$\Esp_X \|g_n^{(\ell)}\|_2^2=\frac{1}{2n}\sum_{k=1}^{n}\left(\frac{k}{n}\right)^{2\ell}(a_k^2+b_k^2) \leq \frac{1}{2n}\sum_{k=1}^{n}(a_k^2+b_k^2),$$
which yields that
\[\Esp_{\Prob \otimes \Prob_X} \|g_n^{(s)}\|_{2}^{2} \leq 1.
\]
Hence
\begin{equation} \label{eq.sup.P.mark}
\sup_{n\geq 1}\Prob \otimes \Prob_X \left(\|g_n^{(\left\lfloor s\right\rfloor)}\|_{\infty}>R\right) \leq \frac{1}{R^2}.
\end{equation}
We can choose $R(s)=(1+|s|)^{\eta/2+1}$.\\
On the other hand, supposing that for $s$ large enough, $\frac{(2\pi)^{\left\lfloor s\right\rfloor}R}{\left\lfloor s\right\rfloor!} <1$, we get by Markov inequality that
\begin{eqnarray*}
\Prob\otimes \Prob_X\left(|g_n(0)|\leq \frac{(2\pi)^{\left\lfloor s\right\rfloor}R}{\left\lfloor s\right\rfloor!}\right)&=&\Prob \otimes \Prob_X\left( |f_n(X)|\leq \frac{(2\pi)^{\left\lfloor s\right\rfloor}R}{\left\lfloor s\right\rfloor!}\right)\\
&=&\Prob\otimes\Prob_X\left(\left| \log |f_n(X)| \right| \geq \left | \log \left(\frac{(2\pi)^{\left\lfloor s\right\rfloor} R}{\left\lfloor s\right\rfloor!}\right)\right| \right)\\
&\leq&\frac{\Esp_{\Prob \otimes \Prob_X}\left[\left| \log (|f_n(X)|)\right|^{1+\eta}\right]}{\left|\log\left(\frac{(2\pi)^{\left\lfloor s\right\rfloor}R(s)}{\left\lfloor s\right\rfloor!} \right)\right|^{1+\eta}}\leq \frac{C_{\eta}(1+\|\log(\psi_{\rho})^{1+\eta}\|_{L^{1}})}{\left|\log\left(\frac{(2\pi)^{\left\lfloor s\right\rfloor}R(s)}{\left\lfloor s\right\rfloor!} \right)\right|^{1+\eta}}.
\end{eqnarray*}
Hence, for all $\eta >0$ and $s$ large enough,
\begin{equation}\label{eq.sup.P}
\sup_{n\geq 1}\Prob \otimes \Prob_X\left(|g_n(0)|\leq \right) =O\left(\left|\frac{1}{\left\lfloor s\right\rfloor \log (\left\lfloor s\right\rfloor)}\right|^{1+\eta}\right).
\end{equation}
Combining Equations (\ref{eq.P.opt}), (\ref{eq.sup.P.mark}) and (\ref{eq.sup.P}), we thus get that
\[
\sup_{n\geq 1}\int_0^{+\infty}s^{\eta/2}\Prob \otimes \Prob_X\left(\mathcal{N}(g_n,[0,2\pi])>s\right) ds < +\infty,
\]
hence the result.
\end{proof}
We can now complete the proof of Theorem \ref{thm.asym.mean}.  
Indeed, Combining the last Proposition \ref{prop.sup.mom} of equi-integrability with the convergence in distribution established earlier in Proposition \ref{prop.conv.dist}, we obtain the convergence of the first moment:
$$\lim_{n\to +\infty}\Esp\left[ \frac{\mathcal{N}(f_n,[0,2\pi])}{n}\right]= \Esp_{\Prob \otimes \Prob_X}[\mathcal{N}(g_{\infty},[0,2\pi])].$$
Now, recall that $g_{\infty}(\cdot)=\sqrt{\psi_{\rho}(X)}N(\cdot)$ with $X$ uniformly distributed on $[0,2\pi]$ and $N$ a stationary Gaussian process with $\sin_c$ covariance function that is independent of $X$. Since $\psi_{\rho}>0$ almost surely, the zeros of $g_{\infty}$ are the same as the ones of $N(\cdot)$. A simple us of Kac--Rice formula for $(N(t))_{t\in\mathbb{R}}$ then implies that

$$\Esp_{\Prob \otimes \Prob_X}[\mathcal{N}(g_{\infty},[0,2\pi])]=\Esp_{\Prob}\left[\mathcal{N}(N,[0,2\pi])\right]=\frac{2}{\sqrt{3}},$$
which concludes the proof.

\subsection{Study of the almost-sure number of real zeros of $f_n$}
In this last section, we reinforce the hypotheses on the spectral density to pass from the convergence in expectation stated in Theorem \ref{thm.asym.mean} to the almost sure convergence stated in Theorem \ref{thm.as.asym}.  We suppose this time that $\mu_{\rho}$ is purely absolutely continuous i.e. $\mu_{\rho}(dx)= \psi_{\rho}(x)dx$ and that the spectral density satisfies the following conditions.
\begin{description}
\item[A.1] There exists $\alpha \in (0,1]$ such that $\psi_{\rho}$ verifies a Besov condition of order $\alpha>0$ i.e. for $\delta >0$,
\[
\omega^{\ast}(\psi_{\rho},\delta):=\sup_{|h|\leq \delta}\left\| \psi_{\rho}(\cdot+h)+\psi_{\rho}(\cdot-h)-2\psi_{\rho}(\cdot)\right\|_{L^{1}([0,2\pi])} = O(\delta^{\alpha}).
\]
\item[A.2] There exists $\gamma>0$ such that $\frac{1}{\psi_{\rho}^{\gamma}} \in L^1\left([0,2\pi]\right)$ . 
\end{description}
In order to ease the reading of the proof of Theorem \ref{thm.as.asym}, let us describe below the main steps of the argumentation.
\par
\medskip
\noindent
\underline{Sketch of the proof of Theorem \ref{thm.as.asym}:}
\par
\noindent
\begin{itemize}
\item[(1)] The point a) of Proposition \ref{prop.conv.dist} ensures that $\mathbb{P}$-almost surely, we have 
$$\mathcal{N}\left(g_n,[0,2\pi]\right)\xrightarrow[n\to\infty]{\text{Law under}~\mathbb{P}_X}\mathcal{N}\left(g_\infty,[0,2\pi]\right).$$

\item[(2)] As in the proof of Theorem \ref{thm.asym.mean}, the previous convergence is not sufficient to take the limit as $n\to\infty$ in the expectation $\mathbb{E}_X$. Like before, one seeks to establish that

$$\Prob\text{-a.s.},\,\exists \eta>1,\,\sup_n \, \mathbb{E}_{X}\left(|\log(f_n(X)|\right)<C_{\eta,\omega}.$$

Unfortunately, Lemma \ref{lm.est.log.P} does not apply here since we consider only the expectation with respect to $\mathbb{P}_X$ in the above estimate. To circumvent this issue we proceed in two distinct steps. We first show that

$$\forall \eta>0,\,\forall \theta>0,\,\mathbb{P}\text{-a.s.},\,\exists C_{\eta,\theta,\omega}>0\,\,\text{s.t.} \,\,\forall n\ge 1,\,\mathbb{E}_X\left[|\log(X)|^{1+\eta}\right]\le C_{\eta,\theta,\omega} n^\theta.$$

In other words, almost surely, the quantity $\mathbb{E}_X\left[|\log(X)|^{1+\eta}\right]$ grows slower than any polynomial. When studying the uniform integrability with respect to $\mathbb{P}_X$ of $\mathcal{N}(g_n,[0,2\pi])$ we can make the following truncation at $s=n^\lambda$ for a suitable $\lambda>0$ (which depends on the Besov regularity of $\psi_ \rho$ in the assumption A.1.)

\begin{eqnarray*}
\Esp_{X} \left[\mathcal{N}(g_n,[0,2\pi])^{1+\eta/2}\right] &=&(1+\eta/2)\int_0^{+\infty}s^{\eta/2} \Prob_X \left(\mathcal{N}(g_n,[0,2\pi])>s\right)ds\\
&=&(1+\eta/2)\underbrace{\int_0^{n^\lambda}s^{\eta/2} \Prob_X \left(\mathcal{N}(g_n,[0,2\pi])>s\right)ds}_{:=I_1}\\
&+&(1+\eta/2)\underbrace{\int_{n^\lambda}^{\infty}s^{\eta/2} \Prob_X \left(\mathcal{N}(g_n,[0,2\pi])>s\right)ds}_{:=I_2}.
\end{eqnarray*}

Then, the integral $I_2$ will be handled using an almost sure small enough polynomial growth of $\mathbb{E}_X\left[|\log(X)|^{1+\eta}\right]$ and the integral $I_1$ will be handled using a quantitative convergence in distribution in Theorem \ref{prop.conv.fc}.

\end{itemize}

Remark that Assumption \textbf{A.2} is sharper than the previous assumption on the log-integrability of the spectral density $\psi_{\rho}$. Indeed, we have the following Lemma.
\begin{lma}\label{lm.mn.log}
If $\frac{1}{\psi_{\rho}}\in L^{\gamma}([0,2\pi])$ for some $\gamma >0$, then for all $r>1, ~\log(\psi_{\rho}) \in L^{r}([0,2\pi])$. 
\end{lma}
\begin{proof}Set $r>1$, since $|x|^\gamma |\log(x)|^r\xrightarrow[x\to 0]~0$, for some constant $C_{r,\gamma}$ one gets get
$$\forall x\in]0,1[,\,|\log(x)|^r\le \frac{C_{r,\gamma}}{|x|^\gamma}.$$

Besides, if $|x|>1$ we also have $|\log(x)|^r \le C_r |x|$ which gives in turn

$$\forall x\in]1,+\infty[,\,|\log(x)|^r\le |x|.$$

Gathering the two previous estimates leads to

$$\forall x\in]0,+\infty[,\,|\log(\psi_\rho(x))|^r\le \frac{C_{r,\gamma}}{\psi_\rho(x)^\gamma}+C_r|\psi_\rho|.$$
The right hand side of the above inequality being integrable, the proof is complete.
\end{proof}

\begin{lma} \label{lm.est.puiss}For $n$ large enough, then for all $\eta>0$, $\Prob$-almost surely, for all $\theta \in (0,1)$, there exists a constant $C(\omega,\eta,\theta)$ such that 
$$\Esp_X |\log(|f_n(X)|)|^{1+\eta}\leq C(\omega,\eta,\theta) n^{\theta}.$$
\end{lma}
\begin{proof}Our approach uses Borel--Cantelli Lemma once again. Set $\beta >\theta^{-1}>1$ such that $(1+\eta)\beta>1$. Lemma \ref{lm.est.log.P} and Jensen inequality ensure that
\[
\Esp_{\Prob \otimes \Prob_X} \left|\log(|f_n(X)|)\right|^{(1+\eta)\beta} \leq C(\eta,\beta, \|\log(\psi_{\rho})^{(1+\eta)\beta}\|_{L^1}).
\]
Then , Markov inequality coupled with Jensen inequality, gives that for $\beta >\frac{1}{\theta}$, 
\begin{eqnarray*}
\Prob\left(\Esp_X\left|\log(|f_n(X)|)\right|^{1+\eta} \geq n^{\theta}\right)&\leq& \frac{1}{n^{\theta \beta}} \Esp_{\Prob}\left [\Esp_{X} \left[\left|\log(|f_n(X)|) \right|^{1+\eta}\right]^{\beta}\right]\\
&\leq&  \frac{1}{n^{\theta \beta}} \Esp_{\Prob\otimes \Prob_X}\left|\log(|f_n(X)|) \right|^{(1+\eta)\beta}\\
&\leq&\frac{C(\eta,\beta, \|\log(\psi_{\rho})^{(1+\eta)\beta}\|_{L^1})}{n^{\theta \beta}}.
\end{eqnarray*}
Since $\theta \beta >1$, the series $\sum_n \frac{1}{n^{\theta \beta}}$ is convergent. Therefore, by Borel--Cantelli Lemma with respect to $\Prob$, for $n$ sufficiently large, we have $\Prob$-almost surely that
\[
\Esp_X\left[\left| \log(|f_n(X)|)\right|^{1+\eta}\right] \leq C(\eta,\theta,\omega) n^{\theta}.
\]
\end{proof}

We can now complete the proof of Theorem \ref{thm.as.asym}.
Following the steps of the proof of Proposition \ref{prop.sup.mom}, we must show that for some $\eta>0$, 
\begin{equation}\label{Equi-ps}
\sup_{n\geq 1} \Esp_{X} \left[ \left | \mathcal{N}(g_n,[0,2\pi])\right|^{1+\frac{\eta}{2}}\right] <+\infty. 
\end{equation}
Combining this equi-integrability result with the convergence in distribution established in Theorem \ref{thm.SZ.func}, we shall obtain that $\mathbb P$ almost surely 
\begin{equation*}
\lim_{n \to +\infty} \frac{\mathcal{N}(f_n,[0,2\pi])}{n}=\Esp_X\left[\mathcal{N}(g_{\infty},[0,2\pi])\right].
\end{equation*}
Now, as in the proof of Theorem \ref{thm.asym.mean}, the zeros of $g_{\infty}$ are the same as the ones of a stationary Gaussian process with $\sin_c$ covariance function, therefore we will indeed obtain that $\mathbb P$ almost surely 
\[
\lim_{n \to +\infty} \frac{\mathcal{N}(f_n,[0,2\pi])}{n}=\frac{2}{\sqrt{3}}.
\]

Let us now focus on the proof of the uniform estimate \eqref{Equi-ps}. To do so, we consider $\chi\in\mathcal{C}^\infty_c$, by Fourier inversion we may write

\begin{eqnarray*}
\mathbb{E}_X\left[\chi\left(f_n(X)\right)\right]&=&\frac{1}{2\pi}\int_{\mathbb{R}}\hat{\chi}(\xi)\mathbb{E}_X\left[\exp\left(-i\xi f_n(X)\right)\right]d\xi\\
\mathbb{E}_{X,N}\left[\chi\left(\sqrt{{2\pi}\, K_n\ast \psi_{\rho}(X)}N\right)\right]&=&\frac{1}{2\pi}\int_{\mathbb{R}}\hat{\chi}(\xi)\mathbb{E}_X\left[\exp\left(-\frac{\xi^2}{2}{2\pi}\, K_n\ast {\psi_{\rho}}(X)\right)\right]d\xi\\
\mathbb{E}_{X,N}\left[\chi\left(\sqrt{{2\pi}\, \psi_\rho(X)}N\right)\right]&=&\frac{1}{2\pi}\int_{\mathbb{R}}\hat{\chi}(\xi)\mathbb{E}_X\left[\exp\left(-\frac{\xi^2}{2}{2\pi}\,  {\psi_{\rho}}(X)\right)\right]d\xi
\end{eqnarray*}
Subtracting the two first equations and using Cauchy--Schwarz inequality entails that
\begin{eqnarray*}
&&\left|\mathbb{E}_X\left[\chi\left(f_n(X)\right)\right]-\mathbb{E}_{X,N}\left[\chi\left(\sqrt{{2\pi}\, {K_n \ast \psi_{\rho}}(X)}N\right)\right]\right|\\\
&&\le\frac{1}{2\pi}\int_{\mathbb{R}}|\hat{\chi}(\xi)|\left|\mathbb{E}_X\left[e^{-i\xi f_n(X)}\right]-\mathbb{E}_X\left[e^{-i\frac{\xi^2}{2} {2\pi}\, {K_n \ast \psi_{\rho}}(X)}\right]\right|d\xi\\
&&\le\frac{1}{2\pi}\sqrt{\int_\mathbb{R}|\hat{\chi}(\xi)|^2(|\xi|^2+1)^2d\xi}~~\sqrt{\int_{\mathbb{R}}\frac{\left|\mathbb{E}_X\left[e^{-i\xi f_n(X)}\right]-\mathbb{E}_X\left[e^{-i\frac{\xi^2}{2} {2\pi}\, {K_n \ast \psi_{\rho}}(X)})\right]\right|^2}{(|\xi|^2+1)^2}d\xi}.
\end{eqnarray*}
Taking the expectation with respect to $\mathbb{P}$, recalling that (see the beginning of Section \ref{sec.un.SZ}) 
\[
\forall t\in\mathbb{R},\,\forall n\ge 1,\,\mathbb{E}\left[ \left|\, \mathbb E_X \left[ e^{i t f_n(X)}  \right] - \mathbb E_X \left[ e^{- \frac{t^2}{2} {2\pi}\, K_n\ast \mu_{\rho}(X)}  \right]\, \right|^2\right]\le C \frac{t^2}{n^{\frac 1 6}},
\]
 {and noticing that $\mu_\rho=\psi_\rho(x)dx~\Rightarrow K_n\ast\mu_\rho=K_n \ast \psi_\rho$ we get that}
 
\[
\mathbb{E}\left[\sqrt{\int_{\mathbb{R}}\frac{\left|\mathbb{E}_X\left[e^{-i\xi f_n(X)}\right]-\mathbb{E}_X\left[e^{-i\frac{\xi^2}{2} {2\pi}\, {K_n \ast \psi_{\rho}}(X)}\right]\right|^2}{(|\xi|^2+1)^2}d\xi}\right]\le \frac{C}{n^{\frac{1}{12}}}\sqrt{\mathbb{E}\left[\int_{\mathbb{R}}\frac{|\xi|^2}{(|\xi|^2+1)^2}d\xi\right]}=\frac{\tilde{C}}{n^{\frac{1}{12}}}.
\]
As a a result, along the subsequence $n^{25}$, one deduces that

\begin{eqnarray*}
\sum_{n=1}^\infty n^{\frac{25}{24}}\mathbb{E}\left[\sqrt{\int_{\mathbb{R}}\frac{\left|\mathbb{E}_X\left[e^{-i\xi f_{n^{25}}(X)}\right]-\mathbb{E}_X\left[e^{-i\frac{\xi^2}{2} {2\pi}\, K_{n^{25}}\ast\psi_\rho(X)}\right]\right|^2}{(|\xi|^2+1)^2}d\xi}\right]<\infty.
\end{eqnarray*}
Thus, by a Borel--Cantelli argument we derive that, $\mathbb{P}$ almost surely, for some constant $C(\omega)>0$,

\begin{equation}\label{a.s.-esti1}
\left(\int_{\mathbb{R}}\frac{\left|\mathbb{E}_X\left[e^{-i\xi f_{n^{25}}(X)}\right]-\mathbb{E}_X\left[e^{-i\frac{\xi^2}{2} {2\pi K_{n^{25}}\ast\psi_\rho(X)}}\right]\right|^2}{(|\xi|^2+1)^2}d\xi\right)^{\frac 1 2}\le\frac{C(\omega)}{n^{\frac{25}{24}}}.
\end{equation}
On the other hand, using Plancherel isometry, provided that $\text{Supp}(\chi)\subset[-M,M]$,  we have
\begin{eqnarray*}
\int_\mathbb{R}|\hat{\chi}(\xi)|^2(|\xi|^2+1)^2d\xi&=&\int_{\mathbb{R}}|\hat{\chi}(\xi)|^2(|\xi|^4+2|\xi|^2+1)d\xi\\
&\le &2\int_{\mathbb{R}}\left(\chi''(x)^2+\chi'(x)^2+\chi(x)^2\right)dx\\
&\le& 2 M \left(\|\chi''\|_\infty^2+\|\chi'\|_\infty^2+\|\chi\|_\infty^2\right).
\end{eqnarray*}
Gathering this last estimate with \eqref{a.s.-esti1} entails that, for any $\chi\in\mathcal{C}^\infty_c$ such that
\begin{itemize}
\item $\max\left(\|\chi'''\|_\infty,\|\chi''\|_\infty,\|\chi'\|_\infty,\|\chi\|_\infty\right)\le 1$,
\item $\text{Supp}(\chi)\subset[-M,M]$,
\end{itemize}
then we have
\begin{equation}\label{a.s.-esti2}
\left|\mathbb{E}_X\left[\chi\left(f_{n^{25}}(X)\right)\right]-\mathbb{E}_{X,N}\left[\chi\left(\sqrt{{2\pi K_{n^{25}}\ast\psi_\rho(X)}}N\right)\right]\right|\le \frac{C(\omega)}{n^{\frac{25}{24}}}\sqrt{M}.
\end{equation}

Given the assumption A.1 on the Besov regularity of $\psi_{\rho}$, since the first Fourier coefficient of the Fej\'er kernel of order $n$ is equal to $1-\frac{1}{n}$, we can apply Theorem 1.5.8 p.69-70 of \cite{Butz71}, with the corresponding notations $\chi_p(u)=K_n(u)$, $\hat{\chi}_p(u)=1-\frac{1}{n}$ and $\omega^*(X_{2\pi},f,h)=\text{O}(h^{\alpha}),$ to deduce that
\[
\|K_n\ast\psi_{\rho}-\psi_{\rho}\|_{L^1([0,2\pi])}= O(n^{-\alpha/2}).
\]
Under the condition A.1, we have thus
\begin{eqnarray*}
&&\left|\mathbb{E}_X\left[\chi\left(\sqrt{{2\pi} \, K_{n^{25}}\ast\psi_\rho(X)}N\right)\right]-\mathbb{E}_{X,N}\left[\chi\left(\sqrt{{2\pi} \,\psi_\rho(X)}N\right)\right]\right|\\
&\le &\frac{1}{2\pi}\int_{\mathbb{R}}|\hat{\chi}(\xi)|\left|\mathbb{E}_X\left[e^{-\frac{\xi^2}{2}{2\pi} \, K_{n^{25}}\ast\psi_\rho(X)}\right]-\mathbb{E}\left[e^{-\frac{\xi^2}{2} {2\pi} \,\psi_\rho(X)}\right]\right|d\xi\\
&\le& \frac{1}{2\pi}\int_{\mathbb{R}}|\hat{\chi}(\xi)|\xi^2 d\xi~\times~\mathbb{E}_X\left[\left|K_{n^{25}}\ast\psi_\rho(X)-\psi_\rho(X)\right|\right]\\
&\le& \frac{C\times \sqrt{M}}{n^{\frac{25 \alpha}{2}}}.
\end{eqnarray*}
{Indeed, as previously one may write
\begin{eqnarray*}
\frac{1}{2\pi}\int_{\mathbb{R}}|\hat{\chi}(\xi)|\xi^2 d\xi&=&\frac{1}{2\pi}\int_{\mathbb{R}}|\hat{\chi}(\xi)|\xi^2 \frac{|\xi|+1}{|\xi|+1}d\xi\\
&\le& \frac{1}{2\pi}\left(\int_{\mathbb{R}}|\hat{\chi}(\xi)|^2\xi^4(1+|\xi|)^2\right)^{\frac 1 2}\left(\int_{\mathbb{R}} \frac{1}{(|\xi|+1)^2}d\xi\right)^{\frac 1 2}\\
&\stackrel{\text{Plancherel}}{\le}& C \left(\int_{\mathbb{R}} \left(|\chi'''(x)|^2+|\chi''(x)|^2+|\chi'(x)|^2+|\chi(x)|^2\right)dx\right)^{\frac 1 2}\\
&\le& C\sqrt{M}.
\end{eqnarray*}
}
\noindent
Gathering the latter with \eqref{a.s.-esti2} provides

\begin{equation}\label{a.s.-esti3}
\left|\mathbb{E}_X\left[\chi\left(f_{n^{25}}(X)\right)\right]-\mathbb{E}_{X,N}\left[\chi\left(\sqrt{{2\pi} \,\psi_\rho(X)}N\right)\right]\right|\le \left(\frac{C(\omega)}{n^{\frac{25}{24}}}+\frac{C}{n^{\frac{25 \alpha}{2}}}\right)\sqrt{M}.
\end{equation}
Now, take $\chi\in\mathcal{C}^\infty_c$ with the condition $\max\left(\|\chi'''\|_\infty,\|\chi''\|_\infty,\|\chi'\|_\infty,\|\chi\|_\infty\right)\le 1$ but not necessarily supported in $[-M,M]$. We can build $\tau_M\in\mathcal{C}^\infty_c$ with the conditions

\begin{itemize}
\item $\tau_M=1$ on $[-M+1,M-1]$,
\item $\tau_M=0$ on $[-M,M]^c$,
\item $0\le \tau_M\le 1$ and $\max(\|\tau_M'''\|_\infty,\|\tau_M''\|_\infty,\|\tau_M'\|_\infty,\|\tau_M\|_\infty)\le C$ for some absolute $C>0$ and any $M>1$.
\end{itemize}
We then write $\chi_M=\chi \times \tau_M$ and we obtain

\begin{eqnarray*}
&&\left|\mathbb{E}_X\left[\chi\left(f_{n^{25}}(X)\right)\right]-\mathbb{E}_{X,N}\left[\chi\left(\sqrt{{2\pi} \, \psi_\rho(X)}N\right)\right]\right|\\
&&\le \left|\mathbb{E}_X\left[\chi_M\left(f_{n^{25}}(X)\right)\right]-\mathbb{E}_{X,N}\left[\chi_M\left(\sqrt{{2\pi} \,\psi_\rho(X)}N\right)\right]\right|\\
&&+\left|\mathbb{E}_X\left[\chi\left(f_{n^{25}}(X)\right)\right]-\mathbb{E}_X\left[\chi_M\left(f_{n^{25}}(X)\right)\right]\right|\\
&&+\left|\mathbb{E}_{X,N}\left[\chi_M\left(\sqrt{{2\pi} \,\psi_\rho(X)}N\right)\right]-\mathbb{E}_{X,N}\left[\chi\left(\sqrt{{2\pi} \,\psi_\rho(X)}N\right)\right]\right|\\
&&\le\left(\frac{C(\omega)}{n^{\frac{25}{24}}}+\frac{C}{n^{\frac{25 \alpha}{2}}}\right)\sqrt{M}+\mathbb{P}_X\left(\left|f_{n^{25}}(X)\right|>M\right)+\mathbb{P}_X\left(\left|\sqrt{{2\pi} \, \psi_\rho(X)}N\right|>M\right).
\end{eqnarray*}
Besides, combining Markov inequality with Birkhoff--Khinchine Theorem (see Section \ref{Gaussian-ergo} below) 
\[
\mathbb{P}_X\left(\left|f_{n^{25}}(X)\right|>M\right)\le \frac{1}{M^2}\frac{1}{n^{25}}\sum_{k=1}^{n^{25}}\frac{a_k^2+b_k^2}{2}= O(1/M^2),
\]
and using Markov inequality again
\[
\mathbb{P}_X\left(\sqrt{{2\pi} \, \psi_\rho(X)}N>M\right)\le \frac{{2\pi} \,}{M^2}\mathbb{E}_{X,N}\left[\psi_\rho(X) N^2\right]=\frac{{1}}{M^2}.
\]

We finally gets, for some constant $C(\omega)$, for every $M>0$ and every $\chi\in\mathcal{C}^\infty_c$ which satisfies $\max\left(\|\chi'''\|_\infty,\|\chi''\|_\infty,\|\chi'\|_\infty,\|\chi\|_\infty\right)\le 1$:

\begin{equation}\label{a.s.-esti4}
\left|\mathbb{E}_X\left[\chi\left(f_{n^{25}}(X)\right)\right]-\mathbb{E}_{X,N}\left[\chi\left(\sqrt{{2\pi} \,  \psi_\rho(X)}N\right)\right]\right|\le \frac{C(\omega)}{M^2}+\sqrt{M} \frac{C(\omega)}{n^\beta},
\end{equation}
where $\beta=\min\left(\frac{25}{24}+\frac{25\alpha}{2}\right).$ On the other hand, relying on the proof of Lemma \ref{lem.birk} (with the subsequence $n^7$ replaced by $n^{25}$) we may infer for $n^{25}\le m\le (n+1)^{25}$ that
\[
\mathbb{E}_X\left[\left|f_{n^{25}}(X)-f_m(X)\right|^2\right]\le \frac{C(\omega)}{n}.
\]
Thus we also have
\[
\left|\mathbb{E}_X\left[\chi\left(f_{n^{25}}(X)\right)\right]-\mathbb{E}_X\left[\chi\left(f_{m}(X)\right)\right]\right|\le \frac{C(\omega)}{\sqrt{n}}.
\]
Combining the latter with \eqref{a.s.-esti4} ensures that, 

\begin{equation}\label{a.s.-esti5}
\left|\mathbb{E}_X\left[\chi\left(f_{m}(X)\right)\right]-\mathbb{E}_{X,N}\left[\chi\left(\sqrt{{2\pi} \, \psi_\rho(X)}N\right)\right]\right|\le C(\omega)\left(\frac{1}{M^2}+\sqrt{M} \frac{1}{m^{\frac{\beta}{25}}}+\frac{1}{m^{\frac{1}{50}}}\right).
\end{equation}

Setting $\theta=\min\left(\frac{\beta}{25},\frac{1}{50}\right)$ and optimizing in $M>0$ implies that $\mathbb{P}$-a.s. there exists a constant $C(\omega)>0$ such that for any $n\ge 1$ and any $\chi\in\mathcal{C}^\infty_c$ with $\max\left(\|\chi''\|_\infty,\|\chi'\|_\infty,\|\chi\|_\infty\right)\le 1$:

\begin{equation}\label{a.s.-esti6}
\left|\mathbb{E}_X\left[\chi\left(f_{m}(X)\right)\right]-\mathbb{E}_{X,N}\left[\chi\left(\sqrt{{2\pi} \,  \psi_\rho(X)}N\right)\right]\right|\le \frac{C(\omega)}{m^{\frac{4 \theta}{5}}}.
\end{equation}

Now, let us assume that $0 \leq \ge \chi \leq 1$, $\chi=1$ on $[-1,1]$ and $\chi=0$ on $\mathbb{R}/[-2,2]$. In particular, we have $\mathds{1}_{[-2,2]}\ge\chi\ge\mathds{1}_{[-1,1]}$. Let $\delta>0$ that will be chosen later. We apply the estimate \eqref{a.s.-esti5} to $\chi\left(\frac{\cdot}{\delta}\right)$ for which one may write 
$\max\left(\|\chi^{(i)}\|_\infty\,;\,i\in\{0,1,2,3\}\right) \le \frac{C}{\delta^3}$ and we obtain that

\begin{eqnarray*}
\mathbb{P}_X\left(\left|f_n(X)\right|\le \delta\right)&\le&\mathbb{E}_X\left[\chi\left(\frac{f_n(X)}{\delta}\right)\right] \le \mathbb{E}_{X,N}\left[\chi\left(\frac{\sqrt{{2\pi} \, \psi_\rho(X)}N}{\delta}\right)\right]+\frac{C(\omega)}{\delta^3 n^{\frac{4 \theta}{5}}}.
\end{eqnarray*}
On the other hand, by the assumption A.2 we know that $\frac{1}{\psi_\rho}\in L^\gamma([0,2\pi]$ which entails that
\begin{eqnarray*}
\mathbb{E}_{X,N}\left[\chi\left(\frac{\sqrt{{2\pi} \,  \psi_\rho(X)}N}{\delta}\right)\right]&\le& \mathbb{P}_{X,N}\left(\sqrt{{2\pi} \,  \psi_\rho(X)}N\le 2 \delta\right)\\
&=&\mathbb{P}_X\left(\lambda \sqrt{{2\pi} \,  \psi_\rho(X)}\le 2\delta\right)+C \lambda\\
&=&\mathbb{P}_X\left({2\pi} \,  \psi_\rho(X)^\gamma\le \left(\frac{2\delta}{\lambda}\right)^{2\gamma}\right) + C\lambda\\
&\le& C \lambda+\mathbb{E}_X\left[\frac{1}{{2\pi} \,\psi_\rho(X)^\gamma}\right]\left(\frac{2\delta}{\lambda}\right)^{2\gamma}.
\end{eqnarray*}
Optimizing in $\lambda$ gives for some constant $C>0$ and any $\delta>0$ the following estimate
\begin{equation}\label{a.s.-esti7}
\mathbb{E}_{X,N}\left[\chi\left(\frac{\sqrt{{2\pi} \,\psi_\rho(X)}N}{\delta}\right)\right]\le C \delta^{\frac{2\gamma}{2\gamma+1}}.
\end{equation}
Gathering the bounds \eqref{a.s.-esti6} and \eqref{a.s.-esti5} gives the bound

\begin{equation}\label{a.s.-esti8}
\mathbb{P}\text{-a.s.},\,\exists C(\omega)>0,\,\text{s.t.}\,\,\forall\delta>0,\,\,\mathbb{P}_X\left(\left|f_n(X)\right|\le \delta\right)\le C \delta^{\frac{2\gamma}{2\gamma+1}}+\frac{C(\omega)}{\delta^3 n^{\frac{4 \theta}{5}}}.
\end{equation}

The rest of the proof follows the same strategy as in the proof of Theorem \ref{thm.asym.mean}. Namely we write, for some $\lambda>0$ to be determined later

\begin{eqnarray}
\nonumber\Esp_{X} \left[\mathcal{N}(g_n,[0,2\pi])^{1+\eta/2}\right] &=&(1+\eta/2)\int_0^{+\infty}s^{\eta/2} \Prob_X \left(\mathcal{N}(g_n,[0,2\pi])>s\right)ds\\
\nonumber&=&(1+\eta/2)\int_0^{n^\lambda}s^{\eta/2} \Prob_X \left(\mathcal{N}(g_n,[0,2\pi])>s\right)ds\\
&+&(1+\eta/2)\int_{n^\lambda}^{\infty}s^{\eta/2} \Prob_X \left(\mathcal{N}(g_n,[0,2\pi])>s\right)ds.\label{eq-N-PX-UI}
\end{eqnarray}
Then we write, using again repeatedly the Rolle Theorem
\begin{eqnarray*}
\Prob_X\left(\mathcal{N}(g_n,[0,2\pi])>s\right) &\leq& \Prob_X\left(|g_n(0)|\leq \frac{(2\pi)^{\left\lfloor s\right\rfloor}}{\left\lfloor s\right\rfloor!}\|g_n^{(\left\lfloor s\right\rfloor)}\|_{\infty}\right)\\
&\le&\Prob_X\left(|g_n(0)|\leq \frac{(2\pi)^{\left\lfloor s\right\rfloor}}{\left\lfloor s\right\rfloor!}M\right)\\
&&+\Prob_X\left(\|g_n^{(\left\lfloor s\right\rfloor)}\|_{\infty}\ge M\right).
\end{eqnarray*}
Besides, via Sobolev estimates, and using the orthogonality of $(\cos(k X),\sin(kX))$ with respect to the Lebesgue measure on $[0,2\pi]$ we obtain that

\begin{eqnarray*}
\Prob_X\left(\|g_n^{(\left\lfloor s\right\rfloor)}\|_{\infty}\ge M\right)&\le&\frac{1}{M^2} \left(\mathbb{E}_X\left[ g_n^{\lfloor s\rfloor}(X)^2\right]+\mathbb{E}_X\left[ g_n^{\lfloor s+1\rfloor}(X)^2\right]\right)\\
&= &\frac{1}{M^2}~~\frac{1}{n}\sum_{k=1}^n \left(\left(\frac{k}{n}\right)^{2\lfloor s\rfloor}+\left(\frac{k}{n}\right)^{2\lfloor s+1\rfloor}\right)\left(\frac{a_k^2+b_k^2}{2}\right)\\
&\le& \frac{C(\omega)}{M^2}.
\end{eqnarray*}
where the last inequality is again due to Birkhoff--Khinchine Theorem.
Together, the two previous bounds imply that

\begin{equation}\label{borne-Px(N)}
\Prob_X\left(\mathcal{N}(g_n,[0,2\pi])>s\right)\le \frac{C(\omega)}{M^2}+\Prob_X\left(|g_n(0)|\leq \frac{(2\pi)^{\left\lfloor s\right\rfloor}}{\left\lfloor s\right\rfloor!}M\right)\\
\end{equation}

Now, we can choose $M(s)=(1+|s|)^{\eta/2+1}$ and notice that, for $s$ large enough, $\frac{(2\pi)^{\left\lfloor s\right\rfloor}M(s)}{\left\lfloor s\right\rfloor!} <1$. As a result, provided that $s$ is large enough, we get by Markov inequality and the fact that $x\mapsto |\log(x)|$ decreases on $]0,1[$:
\begin{eqnarray*}
\Prob_X\left(|g_n(0)|\leq \frac{(2\pi)^{\left\lfloor s\right\rfloor}M(s)}{\left\lfloor s\right\rfloor!}\right)&=&\Prob_X\left( |f_n(X)|\leq \frac{(2\pi)^{\left\lfloor s\right\rfloor}M(s)}{\left\lfloor s\right\rfloor!}\right)\\
&=&\Prob_X\left(\left| \log |f_n(X)| \right| \geq \left | \log \left(\frac{(2\pi)^{\left\lfloor s\right\rfloor} M(s)}{\left\lfloor s\right\rfloor!}\right)\right| \right)\\
&\leq&\frac{\Esp_{X}\left[\left| \log (|f_n(X)|)\right|^{1+\eta}\right]}{\left|\log\left(\frac{(2\pi)^{\left\lfloor s\right\rfloor}M(s)}{\left\lfloor s\right\rfloor!} \right)\right|^{1+\eta}}.
\end{eqnarray*}

Now, we rely on Lemma \ref{lm.est.puiss} which guarantees that for some constant $C(\omega)>0$ and $\epsilon>0$ small enough we have
\[
\forall n\ge 1,\,\Esp_X |\log(|f_n(X)|)|^{1+\eta}\leq C(\omega,\epsilon) n^{\epsilon}.
\]
Hence, for all $\eta >0$ and $s$ large enough, (and recalling that $g_n(0)=f_n(X)$) we get

\begin{equation}\label{eq.sup.P-X}
\sup_{n\geq 1}\Prob_X\left(|f_n(X)|\leq  \frac{(2\pi)^{\left\lfloor s\right\rfloor}M(s)}{\left\lfloor s\right\rfloor!}\right) \le \frac{C(\omega,\eta)n^\epsilon}{s^{1+\eta}}.
\end{equation}
Plugging the estimate \eqref{eq.sup.P-X} inside \eqref{borne-Px(N)} ensures that, for $s$ large enough, $\mathbb{P}$-a.s.,

\begin{equation}\label{eq-N-PX}
\forall n\ge 1, \, \Prob_X\left(\mathcal{N}(g_n,[0,2\pi])>s\right)\le\frac{C(\omega)}{(1+s)^{2+\eta}}+\frac{C(\omega,\eta)n^\epsilon}{s^{1+\eta}}.
\end{equation}
Now we combine \eqref{eq-N-PX} and \eqref{eq-N-PX-UI} and we obtain
\begin{eqnarray*}
\Esp_{X} \left[\mathcal{N}(g_n,[0,2\pi])^{1+\eta/2}\right] &=&(1+\eta/2)\int_0^{+\infty}s^{\eta/2} \Prob_X \left(\mathcal{N}(g_n,[0,2\pi])>s\right)ds\\
&=&(1+\eta/2)\int_0^{n^\lambda}s^{\eta/2} \Prob_X \left(\mathcal{N}(g_n,[0,2\pi])>s\right)ds\\
&+&(1+\eta/2)\int_{n^\lambda}^{\infty}s^{\eta/2} \Prob_X \left(\mathcal{N}(g_n,[0,2\pi])>s\right)ds\\
&\le&(1+\eta/2)\int_0^{n^\lambda}s^{\eta/2}\Prob_X\left(|f_n(X)|\leq  \frac{(2\pi)^{\left\lfloor s\right\rfloor}M(s)}{\left\lfloor s\right\rfloor!}\right)
ds\\
&\stackrel{\eqref{eq-N-PX}}{+}&(1+\eta/2)\int_{n^\lambda}^{\infty}s^{\eta/2} \left(\frac{C(\omega)}{(1+s)^{2+\eta}}+\frac{C(\omega,\eta)n^\epsilon}{s^{1+\eta}}\right)ds.
\end{eqnarray*}
Then, in view of using the estimate \eqref{a.s.-esti8} we set $a_n$ the least positive number such that

$$s>a_n \Rightarrow \frac{(2\pi)^{\left\lfloor s\right\rfloor}M(s)}{\left\lfloor s\right\rfloor!} \le \frac{1}{n^{\frac{\theta}{5}}}.$$

Due to the growth of the term $\lfloor s \rfloor !$ it is clear that for $n$ large enough we have $a_n<n^{\lambda}$.
\[
\begin{array}{lll}
& I := \displaystyle{\int_0^{n^\lambda}s^{\eta/2}\Prob_X\left(|f_n(X)|\leq  \frac{(2\pi)^{\left\lfloor s\right\rfloor}M(s)}{\left\lfloor s\right\rfloor!}\right)ds} \\ 
&=\displaystyle{ \int_0^{a_n}s^{\eta/2}\Prob_X\left(|f_n(X)|\leq  \frac{(2\pi)^{\left\lfloor s\right\rfloor}M(s)}{\left\lfloor s\right\rfloor!}\right)ds
+\int_{a_n}^{n^\lambda} \Prob_X\left(|f_n(X)|\leq  \frac{(2\pi)^{\left\lfloor s\right\rfloor}M(s)}{\left\lfloor s\right\rfloor!}\right)}\\
& \le  \displaystyle{\int_0^{a_n}s^{\eta/2}\Prob_X\left(|f_n(X)|\leq  \frac{(2\pi)^{\left\lfloor s\right\rfloor}M(s)}{\left\lfloor s\right\rfloor!}\right)ds
+n^\lambda \Prob_X\left(|f_n(X)|\leq  \frac{1}{n^{\frac{\theta}{5}}}\right)}\\
& \displaystyle{\stackrel{\eqref{a.s.-esti8}~\text{with}~\delta=n^{-\frac \theta 5}}{\le}}   \displaystyle{C(\omega)\left(\left(\frac{1}{n^{\frac{\theta}{5}}}\right)^{\frac{2\gamma}{2\gamma+1}}+\frac{1}{n^{\frac{\theta}{5}}}\right.
+\left.\int_0^{a_n} \left(\frac{(2\pi)^{\left\lfloor s\right\rfloor}M(s)}{\left\lfloor s\right\rfloor!}\right)^{\frac{2\gamma}{2\gamma+1}}ds+n^\lambda\times \frac{1}{n^{\frac{\theta}{5}}}\right)}.
\end{array}
\]
Provided that $\lambda$ is chosen small enough, each of the above term is uniformly bounded in $n$ which guarantees the desired uniform integrability and achieves the proof.

\begin{rmk}
In the same way, a similar proof would give that $\Prob$-almost surely, for any compact $[a,b]$ of $[0,2\pi]$,
\[\lim_{n \to +\infty} \frac{\mathcal{N}(f_n,[a,b])}{n}=\frac{b-a}{\pi\sqrt{3}}.\]
Indeed, assuming this time that $X$ is uniformly distributed over $[a,b]$, one may study the convergence of the stochastic process $g_n(x)=f_n\left(X+\frac{x}{n}\right)$ towards a non-degenerate limit and apply the same strategy.
\end{rmk}
\section{Appendix}\label{sec.appendix}
\subsection{Birkhoff--Khinchine Theorem for Gaussian sequences}\label{Gaussian-ergo}

In this section, we recall the necessary material about Gaussian ergodicity and the technical details ensuring the validity of Lemma \ref{lem.birk} and \ref{lem.birk2} stated in Sections \ref{sec.un.SZ} and \ref{sec.SZ.func} respectively. First, one can build the stationary sequence $\{a_k\}_{k\ge 1}$  as the coordinates on the space $\mathbb{R}^{\mathbb{N}}$ endowed with the cylindrical topology and equipped with the Gaussian measure $m$ defined as

$$
\forall p\ge 1,\forall (A_1,\cdots,A_p)\in\mathcal{B}\left(\mathbb{R}\right)^p,\,m\left(A_1\times A_2\times\cdots\times A_p\times\mathbb{R}\times\mathbb{R}\times\cdots\right)=\mathbb{P}\left(\bigcap_{i=1}^p\{a_i\in A_i\}\right).
$$
Then, the shift operates on this space and preserves the measure $m$ since the sequence $\{a_k\}_{k\ge 1}$ is stationary. We refer to \cite[pages 188]{CFS80} for an introduction to these kinds of dynamical systems.  One can use the celebrated Birkhoff--Khinchine Theorem for this transformation, see e.g. \cite[pages 11]{CFS80}, and we get that 

$$\mathbb{P}-\text{a.s.}, \;\; \frac{1}{n}\sum_{k=1}^n a_k^2\to \mathbb{E}\left[(x_1,\cdots,)\mapsto x_1^2\,\Big{|}\,\mathcal{I}\right],$$
where $\mathcal{I}$ is the sigma field generated by all the mappings from $\mathbb{R}^\mathbb{N}\to\mathbb{R}$ that are shift-invariant $m$-almost surely. In particular, $\mathbb P-$almost surely, $\frac{1}{n^7}\sum_{k=1}^{n^7} a_k^2+b_k^2$ has a finite limit as $n$ goes to infinity, which is the key ingredient to establish Lemma \ref{lem.birk}.

\begin{proof}[Proof of Lemma \ref{lem.birk}] Recall that $m$ is a large positive integer and $n$ is defined as the unique integer such that $n^7 < m \leq (n+1)^7$.
By Cauchy--Schwarz inequality, we have
\[
\left| \, \mathbb E_X \left[ e^{i t f_{n^7}(X)} \right]-\mathbb E_X \left[ e^{i t f_{m}(X)} \right]\, \right| \leq t \, \mathbb E_X \left[ | f_{n^7}(X) - f_{m}(X)|^2\right]^{1/2}.
\]
Otherwise, we have the estimate
\begin{equation}\label{Birko-as}
\begin{array}{ll}
&\displaystyle{\mathbb E_X \left[ | f_{n^7}(X) - f_{m}(X)|^2\right] \leq 2 \left(  1 - \sqrt{\frac{n^7}{m}}\right)^2 \mathbb E_X \left[ f_{n^7}(X)^2\right]}\\
\\
&+\displaystyle{2 \,\mathbb E_X \left[\left| \frac{1}{\sqrt{m}} \sum_{k=n^7+1}^m a_k \cos(kX)+b_k \sin(kX)\right|^2  \right]}\\
\\
&= \displaystyle{2 \left(  1 - \sqrt{\frac{n^7}{m}}\right)^2 \frac{1}{n^7} \sum_{k=1}^{n^7} \frac{a_k^2 + b_k^2}{2} +\frac{2}{m}\sum_{k=n^7+1}^m \frac{a_k^2 + b_k^2}{2}  }.
\end{array}
\end{equation}
~\\
By Birkhoff--Khinchine Theorem, we have first 
\[
\left(  1 - \sqrt{\frac{n^7}{m}}\right)^2 \frac{1}{n^7}\sum_{k=1}^{n^7} \frac{a_k^2 + b_k^2}{2}\sim \frac{1}{4 n^{9}}\sum_{k=1}^{n^7} \frac{a_k^2 + b_k^2}{2}=O\left(\frac{1}{n^2}\right).
\]
Besides, one may write
\begin{eqnarray*}
&&\frac{1}{m}\sum_{k=n^7+1}^m \frac{a_k^2 + b_k^2}{2}\le \frac{1}{n^7} \sum_{k=n^7+1}^{(n+1)^7} \frac{a_k^2 + b_k^2}{2}\\
&&=\frac{(n+1)^7}{n^7}\frac{1}{(n+1)^7}\sum_{k=1}^{(n+1)^7} \frac{a_k^2 + b_k^2}{2}-\frac{1}{n^7}\sum_{k=1}^{n^7} \frac{a_k^2 + b_k^2}{2}\\
&=&\left(\underbrace{\frac{1}{(n+1)^7}\sum_{k=1}^{(n+1)^7} \frac{a_k^2 + b_k^2}{2}-\frac{1}{n^7}\sum_{k=1}^{n^7} \frac{a_k^2 + b_k^2}{2}}_{:=R_n}\right)+\underbrace{\frac{\sum_{k=0}^6 \binom{7}{k} n^k}{n^7}}_{=O\left(\frac 1 n\right)}\times\underbrace{\frac{1}{(n+1)^7}\sum_{k=1}^{(n+1)^7} \frac{a_k^2 + b_k^2}{2}}_{=O\left(1\right)}.
\end{eqnarray*}
Next we write
\[
\begin{array}{ll}
R_n & =\displaystyle{\sum_{k=1}^{n^7}\left(\frac{a_k^2+b_k^2}{2}\right)\left(\frac{1}{(n+1)^7}-\frac{1}{n^7}\right)+\frac{a_{(n+1)^7}+b_{(n+1)^7}}{2 (n+1)^7}}\\
\\
&=\displaystyle{\underbrace{\frac{1}{n^7} \sum_{k=1}^{n^7}\left(\frac{a_k^2+b_k^2}{2}\right)}_{=O\left(1\right)}
\underbrace{\left(\frac{1}{(1+1/n)^7}-1\right)}_{=O\left(\frac 1 n\right)}+\frac{1}{n+1}~\underbrace{\frac{a^2_{(n+1)^7}+b^2_{(n+1)^7}}{2 (n+1)^6}}_{=o(1)}}.
\end{array}
\]
Let us finally detail the $o(1)$ in the above equation. For any sequence $\{X_k\}_{k\ge 1}$ of standard Gaussian random variables, any $\beta>0$  and any $\epsilon>0$ we have

\[
\sum_{k=1}^\infty\mathbb{P}\left(|X_k|>\epsilon k^\beta\right)\le \sum_{k=1}^\infty \frac{1}{\sqrt{2\pi}}\int_{\epsilon k^\beta}^\infty e^{-\frac{x^2}{2}}dx<\infty.
\]

Then Borel--Cantelli Lemma implies that $X_n/n^\beta\to 0$ almost surely, in particular in our case, we have $|a_{(n+1)^7}/(n+1)^3|=o(1)$, hence the result.

\end{proof}

\begin{proof}[Proof of Lemma \ref{lem.birk2}]
Let us rewrite
\[
Z_{n}(X,t,\lambda)=\sum_{p=1}^M \lambda_j g_n(t_p) = \frac{1}{\sqrt{n}} \sum_{k=1}^n a_k \alpha_{k,n}(X)+b_k \beta_{k,n}(X),
\]
where
\[
\alpha_{k,n}(X):=\sum_{p=1}^M \lambda_p \cos \left( k X+\frac{k t_p}{n}\right), \quad \beta_{k,n}(X):=\sum_{p=1}^M \lambda_p \sin \left( k X+\frac{k t_p}{n}\right).
\]
Due to the orthogonality of trigonometric functions, a straightforward computation then yields the following orthogonality relations, for all $1\leq k \leq n$ 
and $1\leq k' \leq n'$
\[
\mathbb E_X \left[ \alpha_{k,n}(X) \alpha_{k',n'}(X)\right] =  \mathbb E_X \left[ \beta_{k,n}(X) \beta_{k',n'}(X)\right] = \delta_{k,k'}\times \frac{1}{2} \sum_{p,q=1}^M \lambda_p \lambda_q \cos\left( \frac{k t_p}{n}-\frac{k t_q}{n'}\right),
\]
and 
\[
\mathbb E_X \left[ \alpha_{k,n}(X) \beta_{k',n'}(X)\right] =\delta_{k,k'}\times \frac{1}{2} \sum_{p,q=1}^M \lambda_p \lambda_q \sin\left( \frac{k t_p}{n}-\frac{k t_q}{n'}\right),
\]
so that by symmetry
\[
\mathbb E_X \left[ \alpha_{k,n}(X) \beta_{k',n}(X)\right] = 0.
\]
In particular, we have 
\[
\mathbb E_X \left[ \alpha_{k,n}(X)^2\right] =  \mathbb E_X \left[ \beta_{k,n}(X)^2\right] \leq  \frac{1}{2}  \times ||\lambda||_1^2.
\]
and 
\begin{equation}\label{eq.devcos}
\left \lbrace \begin{array}{l}
\displaystyle{\mathbb E_X \left[ (\alpha_{k,n}(X)-\alpha_{k,n'}(X))^2\right] \leq 2\pi ||\lambda||_1^2 \left| \frac{k}{n}-\frac{k}{n'} \right|}, \\
\\
 \displaystyle{\mathbb E_X \left[ (\beta_{k,n}(X)-\beta_{k,n'}(X))^2\right] \leq 2\pi ||\lambda||_1^2 \left| \frac{k}{n}-\frac{k}{n'} \right|}.
 \end{array}\right.
\end{equation}
Recall that $m$ is a positive integer and $n$ is defined as the unique integer such that $n^7 < m \leq (n+1)^7$. By triangular inequality and Cauchy--Schwarz inequality, we have 
\[
\begin{array}{rl}
\left|\mathbb E_X\left[e^{i Z_{n^7}(X,t,\lambda)}\right]  -\mathbb E_X\left[e^{i Z_{m}(X,t,\lambda)}\right] \right| & \leq \mathbb E_X \left[|Z_{n^7}(X,t,\lambda)-Z_{m}(X,t,\lambda)|  \right]  \\
\\
& \leq \sqrt{\Esp_X\left[ U^2\right]} + \sqrt{\Esp_X\left[ V^2\right]} + \sqrt{\Esp_X\left[ W^2\right]} 
\end{array}
\]
where
\[
\begin{array}{ll}
U& :=\displaystyle{\left(\frac{1}{\sqrt{n^{7}}}-\frac{1}{\sqrt{m}}\right) \left(\sum_{k=1}^{n^{7}} a_k \alpha_{k,n^{7}}(X)+b_k \beta_{k,n^{7}}(X)\right)},\\
\\
V & :=\displaystyle{\frac{1}{\sqrt{m}}\left(\sum_{k=1+n^{7}}^{m}a_k \alpha_{k,m}(X)+b_k\beta_{k,m}(X)\right)},\\
 \\
W& :=\displaystyle{\frac{1}{\sqrt{m}}\left(\sum_{k=1}^{n^{7}}a_k\left(\alpha_{k,n^{7}}(X)-\alpha_{k,m}(X)\right)+b_k\left(\beta_{k,n^{7}}(X)-\beta_{k,m}(X)\right)\right)}.
\end{array}
\]
Using the orthogonality relations above, and again Birkhoff--Khinchine Theorem, we have then 
\[
\mathbb E_X[U^2] \leq  ||\lambda||_1^2 \left(\frac{1}{\sqrt{n^{7}}}-\frac{1}{\sqrt{m}}\right)^2 \sum_{k=1}^{n^7} \frac{a_k^2+b_k^2}{2} = ||\lambda||_1^2 \underbrace{\left(1-\sqrt{\frac{n^7}{m}}\right)^2}_{O(1/n^2)} \underbrace{\left( \frac{1}{n^7} \sum_{k=1}^{n^7} \frac{a_k^2+b_k^2}{2}\right)}_{O(1)}.
\]
In the same way and proceeding as in the proof of Lemma \ref{lem.birk} above,  we get
\[
\mathbb E_X [V^2] \leq ||\lambda||_1^2 \times  \frac{1}{m} \sum_{k=1+n^{7}}^m \frac{a_k^2 +b_k^2}{2} = O\left(\frac{1}{n}\right).
\]
Finally, using again the orthogonality relations and Equation \eqref{eq.devcos}, we obtain
\[
\mathbb E_X [W^2] \leq ||\lambda||_1^2 \times 4\pi \times \left( 1-\frac{n^{7}}{m}\right)  \times \left( \frac{1}{m} \sum_{k=1}^{n^{7}} (a_k^2 +b_k^2)\right)=O\left(\frac{1}{n}\right) .
\]
As a conclusion, we get that 
\[
\left|\mathbb E_X\left[e^{i Z_{n^7}(X,t,\lambda)}\right]  -\mathbb E_X\left[e^{i Z_{m}(X,t,\lambda)}\right] \right|  =O \left(\frac{1}{\sqrt{n}}  \right)=O \left(\frac{1}{m^{1/14}}  \right).
\]
hence the result.
\end{proof}
\subsection{Trigonometric kernels and convolutions}

\label{sec.trigo.app}

%

Recall the definitions of the kernel $K_n^{t,\lambda}$ and its normalized version $\bar{K}_n^{t,\lambda}$ given in Lemma 
\ref{lem.rep.Z}. 
\[
K_n^{t,\lambda}(x):=\frac{1}{n}\left| \sum_{p=1}^{M}\lambda_p e^{i\frac{(n+1)}{2n}t_p} \frac{\sin\left(\frac{n}{2}(x+\frac{t_p}{n})\right)}{\sin\left(\frac{x+\frac{t_p}{n}}{2}\right)}\right|^2, \quad 
\bar{K}_n^{t,\lambda}(x):=\frac{{2\pi} \, K_n^{t,\lambda}(x)}{\int_0^{2\pi}K_n^{t,\lambda}(x)dx}.
\]

\begin{lma}\label{lem.trig.kernel}
The function $\bar{K}_n^{t,\lambda}$ is a good trigonometric kernel, i.e. it satisfies the following properties:
\begin{enumerate}
\item $\bar{K}_n^{t,\lambda}\geq 0$ and $\frac{1}{2\pi}\int_{0}^{2\pi}\bar{K}_n^{t,\lambda}(x)dx = 1.$
\item For $n$ large enough, there exists a constant $C=C(t,\lambda)>0$ such that uniformly in $x \in [0,2\pi]$
\[
\sup_{x \in [0,2\pi]} \bar{K}_n^{t,\lambda}(x) \leq C n, 
\qquad 
\bar{K}_n^{t,\lambda}(x)  \leq C \left(  \sum_{p=1}^M \frac{1}{n\left(x+\frac{t_i}{n}\right)^2}\right).
\]
\item For any $\delta>0$ and for $n$ large enough,
\[
\lim_{n \to +\infty}\int_{|x|> \delta} 
\bar{K}_n^{t,\lambda}(x)dx =0.
\]
\end{enumerate}
\end{lma}
\begin{proof}
The first point results from the fact that $K_n^{t,\lambda}$ is trivially non negative and from the very definition of $\bar{K}_n^{t,\lambda}$.
Notice that by Lemma \ref{lm.fej.leb.multibis}, as $n$ goes to infinity $\int_0^{2\pi}K_n^{t,\lambda}(x)dx$ converges to 
\[
{2 \pi} \sum_{p, q=1}^M \lambda_p \lambda_q \sin_c\left(t_p-t_q\right) =\mathbb E \left[ \left| \sum_{p=1}^M \lambda_p N_{t_p} \right|^2\right]>0
\]
where $(N_t)$ is the standard Gaussian $\sin_c$ process, which is known to be non degenerate. Therefore, for $n$ large enough, $\int_0^{2\pi}K_n^{t,\lambda}(x)dx$ is positive and we can give all the estimates on $K_n^{t,\lambda}$ without loss of generality. 
Precisely, for the second point, remark that by triangular inequality, for all $x \in [0,2\pi]$,
\[
K_n^{t,\lambda}(x) \leq 2^{p-1}\sum_{i=1}^{p}\lambda_i^{2} K_n\left(x+\frac{t_i}{n}\right).
\]
where $K_n$ is the standard Fej\'er kernel. The desired estimates then follow directly from the ones established in Lemma \ref{lm.prop.L} for the standard Fej\'er kernel $K_n$. The third point is an immediate consequence of the second, since for $n$ large enough, $K_n^{t,\lambda}(x)=O(1/n)$ for $|x|>\delta$.
\end{proof}

We can now give the proof of Lemma \ref{lm.fej.leb.multi}, i.e. the Fej\'er--Lebesgue type convergence associated to the kernel $\bar{K}_n^{t,\lambda}$.

\begin{proof}[Proof of Lemma  \ref{lm.fej.leb.multi}]
The proof is very close from the one of Lemma \ref{lm.fej.leb}, with the slight difference that, contrary to $K_n$ and $L_n$, the new kernel $\bar{K}_n^{t,\lambda}$ is not even. Let us set
\[
\varphi_x^{\pm}(t):=\mu_{\rho}([0,x\pm t])-t \psi_{\rho}(x), \quad \Phi_x^{\pm}(t):=\int_{0}^{t}  |d\varphi_x^{\pm}(u)|.
\]
Set in the same way $E^{\pm}:=\{x \in [0,2\pi], \Phi_x^{\pm}(t)=o(t)\}$. By Theorem 8.4, p 106 of \cite{Zyg03}, Lebesgue almost all point $x$ in $[-\pi,\pi]$ belongs to the set $E^+ \cap E^-$.
Otherwise, we have the representation 
\[
\bar{K}_n^{t,\lambda} \ast \mu_{\rho}(x)-\psi_{\rho}(x)=\frac{1}{{2\pi}} \left(  \int_0^{\pi}\bar{K}_n^{t,\lambda}(u)d\varphi_x^{-}(u)+ \int_0^{\pi}\bar{K}_n^{t,\lambda}(-u)d\varphi_x^{+}(u)du\right).
\] 
The rest of the proof then follows the exact same lines as the one of Lemma \ref{lm.fej.leb}. Indeed, for all $x\in E^{+}\cap E^{-}$,  using the estimates for $\bar{K}_n^{t,\lambda}$ of the second point of Lemma \ref{lem.trig.kernel} above and an integration by parts, one deduces that the two last integrals converge to zero as $n$ goes to infinty.
\end{proof}


\end{document}